\documentclass[11pt]{article}
\usepackage{geometry}
\usepackage{subcaption}
\usepackage{amsthm,amsmath,amssymb,amsfonts}
\usepackage{microtype,mathrsfs}
\usepackage{epsfig,color}
\usepackage{mathrsfs}
\usepackage{dsfont}
\usepackage{enumerate}
\usepackage{arydshln}	
\usepackage{multirow}
\usepackage[ruled,vlined]{algorithm2e}

\def\N{{\rm I\kern-0.16em N}}
\def\R{{\rm I\kern-0.16em R}}
\def\E{{\rm I\kern-0.16em E}}
\def\Z{{\mathbb{Z}}}
\numberwithin{equation}{section}
\font\eka=cmex10
\usepackage{ae}
\def\ind{\mathrel{\hbox{\rlap{%
\hbox to 7.5pt{\hrulefill}}\raise6.6pt\hbox{\eka\char'167}}}}
\begin{document}

\newtheorem{thm}{Theorem}[section]
\newtheorem{defn}[thm]{Definition}
\newtheorem{prop}[thm]{Proposition}
\newtheorem{cor}[thm]{Corollary}
\newtheorem{rmk}[thm]{Remark}
\newtheorem{lem}[thm]{Lemma}
\newtheorem{exm}[thm]{Example}

\title{\textbf{Distances between distributions via Stein's method}}

\date{\today}
\renewcommand{\thefootnote}{\fnsymbol{footnote}}
\begin{small}
  \author{Marie Ernst\footnotemark[1] \, and Yvik
    Swan\footnotemark[2]}
\end{small}
\footnotetext[1]{Université de Liège.}

\footnotetext[2]{Université libre de Bruxelles.}

\maketitle

\abstract{}{ We build on the formalism developed in
  \cite{ernst2019first} to propose new representations of solutions to
  Stein equations. We provide new uniform and non uniform bounds on
  these solutions (a.k.a.\ Stein factors). We use these
  representations to obtain representations for differences between
  expectations in terms of solutions to the Stein equations. We apply
  these to compute abstract Stein-type bounds on Kolmogorov, Total
  Variation and Wasserstein distances between arbitrary
  distributions. We apply our results to several illustrative
  examples, and compare our results with current literature on the
  same topic, whenever possible. In all occurrences our results are
  competitive.  } \medskip


\smallskip

\noindent
%{\it 2010 AMS subject classification:}  60F05, 60G09.%\end{abstract}

%\tableofcontents

\section{Introduction} %and main results}

Consider two random variables $X_n, X_{\infty}\in \R$ such that
$ \mathcal{L}(X_n) \approx \mathcal{L}(X_{\infty})$.  It is of course
of great importance to be able to quantify this proximity in terms of
a relevant quantity $\mathcal{D}(X_n, X_{\infty})$, say. The
literature contains many such discrepancy metrics, including
Hellinger, L\'evy, Prokhorov, $f$-divergences, relative entropy,
... See e.g.\ \cite{gibbs2002choosing} for an overview. In this paper
we shall focus on the following three:
\begin{itemize}
\item {Kolmogorov distance}: $\mathrm{Kol}(X_n, X_{\infty})  = \sup_{z\in \R} \left| \mathbb{P}(X_n \le z) - \mathbb{P}(X_{\infty} \le z) \right| $  
\item {Total Variation distance}: $\mathrm{TV}(X_n, X_{\infty})  = \sup_{B\subset \R} \left| \mathbb{P}(X_n \in B) - \mathbb{P}(X_{\infty} \in B) \right| $
  \item {Wasserstein distance}: $  \mathrm{Wass}(X_n, X_{\infty}) =
    \int_{-\infty}^{\infty} \left| \mathbb{P}(X_n \le
      z) - \mathbb{P}(X_{\infty} \le z) \right| \mathrm{d}z$
  \end{itemize}
 It is generally non-trivial to determine
  bounds $L_1\le \mathcal{D}(X_n, X_{\infty}) \le L_2$ with $L_1, L_2$
  meaningful and computable quantities. Such bounds typically depend
  on the choice of metric, as well as the nature of the   ``target''
  law ($\mathcal{L}(X_{\infty})$, say)  and of the ``approximating''
  law ($\mathcal{L}(X_n)$, say). Famous examples include the
  following: 
  \begin{exm}[Berry-Esseen bound $\sim$ 1942] \label{ex:bebound} Let
    $X_n = n^{-1/2}\sum_{i=1}^n X_i$ with $X_i$ iid mean 0 variance 1
    and $X_{\infty} \sim \mathcal{N}(0,1)$. Then
    ${\mathrm{Kol}(X_n, X_{\infty}) \le {C}{n^{-1/2}} \mathbb{E} \big[
      |X_1|^3 \big]}$ for $C \in (0.40973, 0.4748)$.
\end{exm}

% Other famous exmamples are the following. 

% \begin{exm}[Fourth moment theorems (Nourdin and Peccati 2009)]
%   Let $F_n$ living in the $q$th Wiener chaos of an isonormal Gaussian
%   process, and $F_{\infty} \sim \mathcal{N}(0,1)$.  If
%   $\mathbb{E}[F_n^2] = 1$ then %there exist $0 < c < C< \infty$ such

% \begin{align*}
% {c\, \mathbf{M}(F_n) \le   \mathrm{TV}(F_n, F_{\infty}) \le  C\, \mathbf{M}(F_n)}
% \end{align*}
% with
% $  \mathbf{M}(F_n) = \max \left\{ | \mathbb{E}[F_n^3]|, |\mathbb{E}[F_n^4-3]|
%   \right\}. 
% $

% \end{exm}

\begin{exm}[Le Cam's inequality $\sim$ 1960]\label{ex:lecam}
    
  Let $X_n = \sum_{i=1}^n X_i$ with
  $X_i\stackrel{\mathrm{ind}}{\sim} \mathrm{Bern}(\theta_i)$ and
  $X_\infty \sim \mathrm{Poi}(\lambda)$ with
  $\lambda=\sum_{i=1}^n \theta_i$. Here and throughout we write
  $a\wedge b = \min(a, b)$ and $a \lor b = \max(a, b)$. Then
  $ { (1 \wedge {\lambda}^{-1}) \sum_{i=1}^n \theta_i^2/32\le
    \mathrm{TV}(X_n, X_{\infty}) \le ({1-e^{-\lambda}}){\lambda}^{-1}
    \sum_{i=1}^n \theta_i^2 } $.
\end{exm}

% \begin{exm}[Ehm's inequality (1991)]

%   Let $F_n = \sum_{i=1}^n X_i$ with
%   $X_i\stackrel{\mathrm{ind}}{\sim} \mathrm{Bern}(\theta_i)$ and
%   $F_\infty \sim \mathrm{Bin}(n, \theta)$ with
%   $\theta= n^{-1} \sum_{i=1}^n\theta_i$.
%     Then 
%   \begin{equation*}
% { \frac{1}{124}\big(\frac{1}{n\theta(1-\theta)} \wedge 1\big)        \sum_{i=1}^n (\theta_i - \theta)^2 \le         \mathrm{TV}(F_n, F_{\infty}) \le \frac{1}{(n+1) \theta(1-\theta)}
%         \sum_{i=1}^n (\theta_i - \theta)^2}
%   \end{equation*}
% \end{exm}

Examples \ref{ex:bebound} and \ref{ex:lecam} illustrate situations
wherein the target law is easy and explicit while the approximating is
unknown and unfathomable. There is also interest for situations
wherein both the target and the approximating distributions are known
explicitly.
\begin{exm}[\cite{duembgen2019bounding}]\label{ex:dumbgetal}
  \begin{itemize}
  \item  
    $\mathrm{TV}(\mathrm{Hyp}(N, L, n), \mathrm{Bin}(n, L/N)) \le (n-1)/N.$
      \item  $\mathrm{TV}( \mathrm{Bin}(n, \lambda/n), \mathrm{Poi}(\lambda)) \le 1-
    \left( 1-\frac{\lceil \lambda \rceil}{n} \right)^{1/2}.$

\item $\mathrm{TV}(\mathrm{Beta}(a, b), \mathrm{Gamma}(a, a+b)) \le 1-
    \left( 1-\frac{a+1}{a+b} \right)^{1/2}$
  % \item
  %   $\mathrm{TV}(\mathrm{Beta}(a, b), \mathrm{Gamma}(a, c)) \ge
  %   \mathrm{TV}(\mathrm{Beta}(a, b), \mathrm{Gamma}(a, a+b-1))$ for
  %   all $c$ and
  %   $$\mathrm{TV}(\mathrm{Beta}(a, b), \mathrm{Gamma}(a,a+b-1))\le 1-
  %   \left( 1-\frac{a}{a+b-1} \right)^{1/2}$$
  \end{itemize}
\end{exm}

There are many ways to prove estimates such as those provided in
Examples \ref{ex:bebound}, \ref{ex:lecam}, and \ref{ex:dumbgetal},
such as Fourier methods, couplings or, whenever possible, direct
analysis of the densities involved.  In this paper we will consider
the well-known \emph{Stein's method}. Our approach builds upon recent
results from \cite{ernst2019first,ernst2019infinite}. In those papers
it is shown that one can associate to any $X_{\infty}$ \emph{two}
linear operators $\mathcal{T}_{\infty}^{\ell}$ and
$\mathcal{L}_\infty^\ell $ such that the ``Stein identities''
\begin{align}
\label{eq:stid1} 
& \mathrm{Cov}[f(X_\infty), g(X_\infty)] = \mathbb{E}
                \left[-\mathcal{L}_\infty^\ell f(X_\infty)
                \Delta^{-\ell}g(X_\infty) \right]\\
\label{eq:stid2}
&   \mathbb{E} \left[ \big(\mathcal{T}_{\infty}^\ell f(X_\infty) \big)g(X_\infty) \right] =   -  \mathbb{E} \left[ f(X_\infty) \Delta^{-\ell}g(X_\infty) \right]          
\end{align}
are valid for all sufficiently regular functions $f, g$ (here 
$\Delta^{-\ell}$ is a generalized differential operator, see Section
\ref{sec:formalism-1} for explicit expressions).
\begin{exm}
  Take $X_{\infty}$ standard Gaussian with density
  $\varphi(x) = (2\pi)^{-1} e^{-x^2/2}$. Then $\ell=0$,
  $\mathcal{T}_{\infty}^0f(x) = f'(x) - xf(x)$,
  $\mathcal{L}_{\infty}f(x) = e^{x^2/2} \int_{-\infty}^x (f(u) -
  \mathbb{E}[f(X_{\infty})]) \mathrm{d}u$, so that \eqref{eq:stid1} and
  \eqref{eq:stid2} read as
  \begin{align*}
& \mathrm{Cov}[f(X_\infty), g(X_\infty)] = \mathbb{E}
                \left[ \left( e^{X_{\infty}^2/2} \int_{-\infty}^{X_{\infty}}
                   (\mathbb{E}[f(X_{\infty})] - f(u)) e^{-u^2/2} \mathrm{d}u  \right)
                g'(X_\infty) \right]\\
&   \mathbb{E} \left[ \big(f'(X_\infty)-X_{\infty} f(X_\infty) \big)g(X_\infty) \right] =   -  \mathbb{E} \left[ f(X_\infty) g'(X_\infty) \right]         
  \end{align*}
  which hold for all $f \in L^1(\varphi)$ and absolutely continuous
  functions $g$.  Both identities are a straightforward consequence of
  Fubini's theorem.
\end{exm}
If, in \eqref{eq:stid1} or \eqref{eq:stid2}, we take expectations with
respect to $X_n$ rather than $X_{\infty}$, absence of equality in
either identities for some functions $f, g$ indicates absence of
equality between the laws of $X_n$ and $X_{\infty}$.  Stein's method
consists in transforming this observation into estimates on relevant
probability distances between the laws of $X_n$ and $X_{\infty}$.
More precisely, the method advocates to fix $f$ in \eqref{eq:stid1} or
\eqref{eq:stid2} some ``well chosen'' function (e.g.  $f(x) = 1$, but
this is not always ideal) and use the numbers
\begin{align*}
  & \mathcal{S}_A(X_n, X_{\infty}, \mathcal{G}) := 
\sup_{g \in \mathcal{G}} \left|  \mathrm{Cov}[f(X_n), g(X_n)]+ 
 \mathbb{E}  \left[ \big(\mathcal{L}_\infty^\ell f(X_n)\big)
  \Delta^{-\ell}g(X_n) \right] \right|\\
  & \mathcal{S}_B(X_n, X_{\infty}, \mathcal{G}) := 
    \sup_{g \in \mathcal{G}} \left| \mathbb{E} \left[
        \big(\mathcal{T}_{\infty}^\ell f(X_n) \big)g(X_n) + f(X_n)
        \Delta^{-\ell} g(X_n) \right]\right| 
\end{align*}
 (with $\mathcal{G}$
    ``some class of functions'' to be determined) to quantify the
    difference between the laws of $X_n$ and $X_{\infty}$.
    \begin{exm}
      If $X_{\infty}$ is standard normal, fixing $f(x) = x$ in
      \eqref{eq:stid1} (or $f (x) = 1$ in \eqref{eq:stid2}) leads to the
      discrepancy measure
      $ \sup_{g \in \mathcal{G}} \left| \mathbb{E}[g'(X_n)-X_n g(X_n)]
      \right|$ which, in light of Stein's characterization of the
      normal distributon, is 0 if and only if $X_n$ is itself Gaussian
      -- at least when $\mathcal{G}$ is a sufficiently large class of
      test functions. Other choices of $f$ are possible, see
      \cite{goldstein2005distributional}.
    \end{exm}
    Before diving into the study of the numbers
    $ \mathcal{S}_{\bullet}(X_n, X_{\infty}, \mathcal{G}) $, it is
    first necessary to argue as to why such numbers indeed metrize
    convergence in distribution in terms of relevant metrics. To this
    end, it suffices to notice that discrepancies
    $ \mathcal{S}_{\bullet}(X_n, X_{\infty}, \mathcal{G}) $ contain
    (at least formally) %many known metrics, and particularly
    any distance that can be represented as an \emph{Integral
      Probability Metric} (IPM):
   \begin{equation}\label{eq:defIPM}
     \mathcal{D}_{\mathcal{H}}(X_n, X_{\infty}) = \sup_{h\in \mathcal{H}}
     |\mathbb{E}h(X_n)- \mathbb{E}h(X_{\infty})|.
   \end{equation}
To see why this holds true, fix $f = \eta$ in \eqref{eq:stid1} or $f = c$
in \eqref{eq:stid2} (the difference in notation is cosmetic but will help
at a later stage) and consider the \emph{Stein equations}
    \begin{align}
     \label{eq:stek1}
    & (\eta(x)-\mathbb{E}\eta(X_{\infty}))g_h(x)+
   \big(\mathcal{L}_\infty^\ell \eta(x)\big) \Delta^{-\ell}g_h(x) =h(x) -
   \mathbb{E}h(X_{\infty})\\
   &\mathcal{T}_{\infty}^\ell c(x) g_h^*(x) + c(x) \Delta^{-\ell} g_h^*(x) = h(x) - \mathbb{E}h(X_{\infty})\label{eq:stek2}
   \end{align}
   for all $x \in \mathcal{S}(p_{\infty})$.  Lemma 2.11 in
   \cite{ernst2019first} guarantees that if $\mathcal{H}$ is
   reasonable, then for any well-chosen $\eta$ or $c$, to every
   $h\in \mathcal{H}$ we can associate (uniquely) a function $g_h$ or
   $g_h^*$ such that either \eqref{eq:stek1} or \eqref{eq:stek2} holds
   at all $x$ in the support of the law of $X_{\infty}$. Let
   $\mathcal{G}_{\mathcal{H}} = \left\{ g_h \, | \, h \in \mathcal{H}
   \right\}$ and $\mathcal{G}^*_{\mathcal{H}}= \left\{ g_h^{\star} \, | \, h \in \mathcal{H}
   \right\}$ be 
   the collection of all these solutions.  Then simple computations
   show that
   $$ \mathcal{D}_{\mathcal{H}}(X_n, X_{\infty}) = \mathcal{S}_{A}(X_n,
   X_{\infty}, \mathcal{G}_{\mathcal{H}}) = \mathcal{S}_{B}(X_n,
   X_{\infty}, \mathcal{G}^*_{\mathcal{H}}).$$ In other words, under
   non-stated regularity conditions which basically require that all
   quantities be defined, the IPMs \eqref{eq:defIPM} can be
   interpreted as specific instances of Stein's discrepancies
   $\mathcal{S}_{\bullet}$.
 \begin{exm}
   Still in the case where $X_{\infty}$ is standard Gaussian, fix
   $ \eta= \mathrm{Id}$ the identity function in \eqref{eq:stek1} (or,
   equivalently, $c=1$ in \eqref{eq:stek2}) and consider the Stein
   equation
     \begin{equation}\label{eq:classgau}
       g'(x) - xg(x) = h(x) - \mathbb{E}h(X_{\infty})
     \end{equation}
     over $x \in \mathbb{R}$. For each $h \in L^1(X_{\infty})$ there
     exists a unique bounded solution given by
     $g_h(x) = e^{x^2/2} \int_{-\infty}^x (h(u) -
     \mathbb{E}h(X_{\infty})) e^{u^2/2} \mathrm{d}u$ (we recognize the
     operator $-\mathcal{L}_{\infty}^0$ from the previous example), so
     that
     \begin{align*}
\mathcal{D}_{\mathcal{H}}(X_n, X_{\infty}) = \sup_{h \in \mathcal{H}}
       \left|  \mathbb{E}[g_h'(X_n) - X_n g_h(X_n)] \right|
     \end{align*}
     and all IPMs with Gaussian target are indeed Stein discrepancies.
   \end{exm}
   Many classical metrics can be represented as IPMs, most notably for
   us the Kolmogorov, Total Variation and Wasserstein distances  
   with  respective classes 
   \begin{align*}
& \mathcal{H}_{\mathrm{Kol}} = \left\{ h(x) = \mathbb{I}[x \in
                    (-\infty, z]]  \mbox{ such that } z \in \mathbb{R} \right\}     \\
& \mathcal{H}_{\mathrm{TV}} = \left\{ h(x) = \mathbb{I}[x \in
                   B]   \mbox{ such that }  B \in \mathcal{B}(\mathbb{R}) \right\}     \\         
& \mathcal{H}_{\mathrm{Wass}} = \mathrm{Lip}(1) = \left\{ h(x)   \mbox{ such that } |h(x) - h(y)| \le |x-y| \mbox{ for all } x, y \in \mathbb{R}\right\}                                  
 \end{align*}
% ; other classical metrics or pseudo metrics
   % are intrinsically connected to
   % $ \mathcal{S}_{\bullet}(X_n, X_{\infty}, \mathcal{G}) $; for
   % instance Fisher information and relative entropy
   % \cite{nourdin2014integration}, or Wasserstein-$p$
   % (\cite{bonis2015rates}).
%    \begin{rmk}
%   A direct application of Lemma \ref{eq:defIPM} induces the
%   well-known relation between distances (see for instance the review
%   \cite{gibbs2002choosing}) :
%    \begin{equation}
%    \mathrm{Kol}(X_n,X_\infty) \leq \mathrm{TV}(X_n,X_\infty) \leq \frac{1}{d_{\min}} \mathrm{Wass}(X_n,X_\infty) \label{eq:orderdist}
%    \end{equation}
% where $d_{\min} = \min_{x \ne y} d(x, y)$.
% \end{rmk}   
 To summarize what has just been written, the heuristic behind our
 version of Stein's method for a metric of the form \eqref{eq:defIPM}
 is to tackle the problem of bounding an IPM by contemplating the
 identities
\begin{align*}
 \mathcal{D}_{\mathcal{H}}(X_n, X_{\infty}) & =
\sup_{h \in
    \mathcal{H}} \left|  \mathbb{E} \left[   
                                              (\eta(x)-\mathbb{E}\eta(X_{\infty}))g_h(x)+
   \big(\mathcal{L}_\infty^\ell \eta(x)\big) \Delta^{-\ell}g_h(x)  \right]
                                              \right|\\
&   = \sup_{h \in
    \mathcal{H}} \left|  \mathbb{E} \left[   \mathcal{T}_\infty^{\ell}
      c(X_n) g_h(X_n) + c(X_n) \Delta^{-\ell} g_h(X_n)  \right]
                                              \right|
\end{align*}
where $ g_h(x)$ is solution to either \eqref{eq:stek1} (first case) or
\eqref{eq:stek2} (second case).  It remains of course to be able to
choose $\eta$ or $c$ in such a way that the resulting expressions are
tractable \emph{and} the corresponding solutions $g_h$  are well
behaved.

It is now extremely well documented that, for many classical targets
(particularly the normal and Poisson), this approach is powerful
because there are many handles for dealing with the quantities
$\mathcal{S}_{\bullet}$, be it via exchangeable pairs, zero- and size
  bias, Malliavin-Stein, etc.  We refer the reader to
  \cite{barbour1992poisson}, \cite{chen2010normal} and
  \cite{nourdin2012normal} (among many other possible references) for
  an in-depth overview of a broad variety of applications around the
  Gaussian and Poisson cases.
  In this paper, we adopt the abstract formalism developed in
  \cite{ernst2019first,ernst2019infinite} to provide a new point of
  view on the properties of the solutions to equations
  \eqref{eq:stek1} and \eqref{eq:stek2}. Our results are of two main
  types.
   \begin{itemize}
   \item 
   The first, developed in Section \ref{sec:examples}, is of
   a classical nature within the theory on Stein's method, and
   summarized in Proposition \ref{prop:boundsonsol}:  we provide
   explicit uniform and non-uniform bounds on the solutions to Stein
   equations and on their derivatives. In all the examples we have
   considered, our bounds are easily computed and competitive with
   existing bounds (whenever there are competitors available).  For
   instance, applying our bounds to the Gaussian case leads (see
   Example \ref{ex:norm}) to the fact that the solutions to equation
   \eqref{eq:classgau} satisfy
     \begin{align*}
    & |g(x)|   \le \min \left( \kappa_1   \frac{\Phi(x)(1- \Phi(x))}{\varphi(x)}
      , \kappa_2  \right) \le  \min\left(\kappa_1
      \frac{1}{2}\sqrt{\frac{\pi}{2}}, \kappa_{2}\right)\\ 
    & |g'(x)| \le  \kappa_1  \left(1 + |x|\frac{\Phi(x)(1-
      \Phi(x))}{\varphi(x)} \right) \le 2 \kappa_1 \\ 
    & |g'(x)| \le 2 \kappa_2 \min\left(|x|, \frac{\int_{-\infty}^x \Phi(u)
      \mathrm{d}u\int_{x}^\infty(1- \Phi(u)) \mathrm{d}u}{\varphi(x)}\right)
      \le 2\kappa_2  \min \left( \sqrt{\frac{2}{\pi}}, |x|\right)
     \end{align*}
     where $\Phi$ is the standard normal cdf,
     $\kappa_1 \le 2\|h\|_{\infty}$ and
     $\kappa_2 \le \|h'\|_{\infty}$. In the body of the article we
     also compute the bounds the Poisson (Example \ref{ex:pois}) and
     the exponential (Example \ref{ex:exp}).  Other targets are
     covered in the supplementary material to this article.

   \item Our second main result is developed in Section
     \ref{sec:discr}, where
     we %provide something of a new version of Stein's method.  We
     propose probabilistic representations of differences between
     expectations which allow to dispense with the need to bound
     solutions to Stein equations. As applications we provide new
     representations for (and bounds on) the Kolmogorov, Total
     Variation and Wasserstein distances whenever the target and the
     approximating random variables are continuous w.r.t.\ the same
     dominating measure. For instance in the case of a Gaussian target
     we obtain (see Example~\ref{sec:stand-norm-targ}) that if
     $X_n \sim p_n$ has support an interval in $\R$ and score function
     $\rho_n(x)$ then
   \begin{align*}
 \mathrm{Kol}(X_n,X_\infty)& = \sup_z \left| \mathbb{E}\left[ (X_n +
    \rho_n(X_n)) \frac{\Phi(X_n\wedge z)\bar \Phi(X_n\vee z)}{\varphi(X_n)} \right]
                             \right|    \\
    & \le  
      \mathbb{E}\left[ |X_n+\rho_n(X_n)| \frac{\Phi(X_n)\bar \Phi(X_n)}{\varphi(X_n)}
    \right] \\
    & \le \frac{1}{2}\sqrt{\frac{\pi}{2}} \mathbb{E}\left[ |X_n+\rho_n(X_n)| 
    \right],
    \end{align*}
    and also provide bounds on Total Variation and Wasserstein
    distances.  We also compare, whenever possible, with other
    available bounds.  Our results appear to be competitive with or
    improve on the current literature on the topic.
   \end{itemize}
    
  The structure of the paper is as follows. We begin by recalling the
  formalism of Stein's method in Section \ref{sec:formalism-1}.
  %and introduce new generalized Stein discrepancies (Section \ref{sec:steins-dens-appr}). 
  We discuss the properties of solutions to Stein equations in Section
  \ref{sec:repr}, and provide explicit uniform and non uniform bounds
  in Section \ref{sec:examples}. In Section \ref{sec:discr} we provide
  new representations for and bounds on the IPMs between densities
  sharing a common dominating measure, and we apply these in several
  examples. Most proofs are either omitted or delayed to the
  Appendix. Many more computations are made available in the
  supplementary material.

  \section{Stein operators, equations and solutions}
  \label{sec:formalism}

\subsection{Formalism}
\label{sec:formalism-1}
We start by recalling the formalism introduced in
\cite{ernst2019first}. Let $\mathcal{X}{\in \mathcal{B}(\R)}$ and
equip it with some $\sigma$-algebra $\mathcal{A}$ and $\sigma$-finite
measure $\mu$. Let $X$ be a random variable on $\mathcal{X}$, with
{induced probability measure $\mathbb{P}^X$} which is absolutely
continuous with respect to $\mu$; we denote {by} $p$ the corresponding
probability density function (pdf or pmf), and its support by
$\mathcal{S}(p) = \left\{ x \in \mathcal{X} : p(x)>0\right\}$. We also
let $P$ be the cdp of $p$, and $\bar P = 1-P$ its survival function.
As usual, $L^1(p)$ is the collection of all real valued functions $f$
such that $\mathbb{E}|f(X)| < \infty$.
%We sometimes call the expectation under $p$ the $p$-mean. 
Although we could in principle
keep the discussion to come very general, in order to make the paper
more concrete and readable we shall often restrict our attention to
distributions satisfying the following Assumption.

\

\noindent \textbf{Assumption A.} The measure $\mu$ is either the
counting measure on $\mathcal{X} = \Z$ or the Lebesgue measure on
$\mathcal{X} = \R$. If $\mu$ is the counting measure then there exist
$a {<} b \in \Z \cup \left\{-\infty, \infty \right\}$ such that
$\mathcal{S}(p) = [a, b]\cap \Z$.  If $\mu$ is the Lebesgue measure
then there exist $a, b \in \R \cup \left\{-\infty, \infty \right\}$
such that ${\mathcal{S}(p)}^{\mathrm{o}} = ]a, b[$ and
$\overline{\mathcal{S}(p)} = [a, b]$.  Moreover, the measure $\mu$ is
not point mass.

\

Let $\ell \in \left\{ -1, 0, 1 \right\}$; we assume this throughout
the paper and do not recall it. In the sequel we shall restrict our
attention to the following three derivative-type operators:
\begin{eqnarray*}
  \Delta^{\ell}f(x) &= \left\{ 
\begin{array}{l l } 
  f'(x), &  \mbox{ if }\ell = 0;\\
  \frac{1}{\ell} (f(x+\ell) - f(x)) & \mbox{ if } \ell \in \{-1, +1\}, \\
 \end{array} \right. 
\end{eqnarray*} 
   % Let  $$\Delta^{\ell}f(x) = (f(x+\ell) - f(x))/\ell$$ for all
  %  $\ell \in \R$, with the convention that $\Delta^0f(x) = f'(x)$,
with $f'(x)$ the weak derivative defined Lebesgue almost everywhere,
$\Delta^{+1} ({\equiv} \Delta^+)$ the classical forward difference and
$\Delta^{-1} ({\equiv} \Delta^-)$ the classical backward
difference. Whenever $\ell=0$ we take $\mu$ as the Lebesgue measure
and speak of the \emph{continuous case}; whenever
$\ell \in \left\{-1, 1 \right\}$ we take $\mu$ as the counting measure
and speak of the \emph{discrete case}. There are two choices of
derivatives in the discrete case, only one in the continuous
case. {We let $\mathrm{dom}(\Delta^{\ell})$ denote the collection
  of functions $f : \R \to \R$ such that $\Delta^{\ell}f(x)$ exists
  and is finite $\mu$-almost surely. In the case $\ell=0$, this
  corresponds to all absolutely continuous functions; in the case
  $\ell = \pm1$ the domain is the collection of all functions on
  $\Z$.}  Finally, throughout the paper, we will use the notation
$a_\ell = \mathbb{I}[\ell=1]$ and $b_\ell = \mathbb{I}[\ell=-1]$.

\begin{defn}[Canonical Stein operators] \label{def:candiscop} Let
  $X \sim p$. The canonical ($\ell$-)Stein operator is
  \begin{equation*}
%    \label{eq:cansteinop}
    \mathcal{T}_p^\ell f(x) = \frac{\Delta^\ell (f(x) p(x))}{p(x)}
  \end{equation*}
  with the convention that $\mathcal{T}_p^\ell f(x)=0$ for all
  $x \notin \mathcal{S}(p)$.   The {canonical
    pseudo-inverse ($\ell$-)Stein operator} is, for $h \in L^1(p)$,  
 \begin{equation}
    \label{eq:caninv}
    \mathcal{L}_p^{\ell}h(x) = \frac{1}{p(x)} \int_a^{x-a_{\ell}} (h(u)  -
    \mathbb{E}[h(X)])p(u) \mu(\mathrm{d}u) = \frac{1}{p(x)}
    \int_{x+b_{\ell}}^{b} (
    \mathbb{E}[h(X)]-h(u))p(u) \mu(\mathrm{d}u)  
\end{equation}
for all $x \in \mathcal{S}(p)$ and 
 %  with $\mathcal{L}^{\ell}_ph(x)$ defined  as 
    $ \mathcal{L}_p^{\ell}h(x) = 0$ for all
    $x \notin \mathcal{S}(p)$.
    If $\ell=1$ (resp., $\ell = -1$) we call the operators forward
    (resp., backward), denoted $\mathcal{T}_p^+$ (resp.,
    $\mathcal{T}_p^-$) and $\mathcal{L}_p^+$ (resp.,
    $\mathcal{L}_p^-$).
\end{defn}

One can check (see \cite{ernst2019first}) the following results.
\begin{thm}[\cite{ernst2019first}]

  Let
  $ \mathcal{F}^{(0)}(p) = \{ f \in L^1 (p) : \mathbb{E}[f(X)]= 0 \}$
  and
  $ \mathcal{F}^{(1)}_\ell (p) = \{ f \in
    \mathrm{dom}(\Delta^\ell) : \Delta^\ell(f p )
    {\mathbb{I}[\mathcal{S}(p)]} \in L^1(\mu)$ and
    $ \int_{\mathcal{S}(p)} \Delta^\ell (fp)(x) \, \mu(\mathrm{d}x) =
    \mathbb{E}\left[ \mathcal{T}_p^{\ell}f(X) \right]= 0 \}$.
  Then $\mathcal{T}_p^{\ell} f \in \mathcal{F}^{(0)}(p)$ for all
  $f \in \mathcal{F}^{(1)}_\ell (p)$ and
  $\mathcal{L}_p^{\ell} h \in \mathcal{F}^{(1)}_\ell (p)$ for all
  $h \in L^1(p)$. Moreover
  $\mathcal{T}_p^{\ell}(\mathcal{L}_p^{\ell}h(x)) = h(x)-\E[h(X)]$ for
  all $x \in \mathcal{S}(p)$ for all $h\in L^1(p)$ and
  $\mathcal{L}_p^{\ell}(\mathcal{T}_p^{\ell}h(x))=h(x)$ on the
  subclass of centred (i.e. $\E[h(X)]=0$) functions in
  $L^1(p) \cap \mathcal{F}_\ell^{(1)}(p)$.
\end{thm}
% In particular 
% \begin{equation}
%   \label{eq:firststeiidentity}
%   \mathbb{E} [ \mathcal{T}_p^{\ell} f(X)] = 0 \mbox{ for all } f \in
%   \mathcal{F}^{(1)}_\ell (p). 
% \end{equation}

Functions of the form $x \mapsto \mathcal{T}_p^{\ell}f(x)$ or
$x \mapsto \mathcal{L}_p^{\ell} h(x)$, for given special choices of
$f, h$, will play a crucial role in the sequel.  Of particular
importance is the choice of the constant function $f(x)=1$, on the one
hand, and the identity function $h(x) = \mathrm{Id}(x)$ on the other
hand. This leads to the next Definition (see \cite{ernst2019first}).
  \begin{defn}
    The score function of $p$ is
    $\rho_p^\ell(x)=\mathcal{T}_p^\ell 1(x) = \Delta^\ell p(x) /p(x)$;
    if $p$ has finite mean then its Stein kernel is
    $\tau_p^\ell(x)= - \mathcal{L}_p^\ell \mathrm{Id}(x)$. 
  \end{defn}
  \begin{exm}[Gaussian target] \label{ex:gau1} Consider a standard
    Gaussian target with density $\varphi(x)\propto e^{-x^2/2}$.  Then
    $\ell=0$. Simple computations show that $\rho_{\varphi}(x) = -x$ and
    $\tau_{\varphi}(x) = 1$.
\end{exm}

\begin{exm}[Exponential target] \label{ex:exp1}
   Consider a rate $\lambda$ exponential target with density
    $p_{\mathrm{exp}}(x) =  \lambda e^{-\lambda x} \mathbb{I}[x \ge
    0]$. Then $\ell = 0$. Simple computations show that
    $\rho_{\mathrm{exp}}(x) = - \lambda \mathbb{I}[x \ge 0]$ and
    $\tau_{\mathrm{exp}}(x) = x/\lambda \mathbb{I}[x \ge 0]$.
\end{exm}

\begin{exm}[Poisson target] \label{ex:poi1} Consider the discrete
  Poisson target density
  $p_{\mathrm{pois}}(x) = e^{-\lambda} \lambda^x / x! \mathbb{I}[x \ge
  0]$. Then, $\ell=-1$ or 1. Simple computations show that
  $\rho_{\mathrm{pois}}^+(x) =\lambda/(x+1)-1 $ and
  $\rho_{\mathrm{pois}}^-(x) =1- x/\lambda$,
  $\tau_{\mathrm{pois}}^+(x) = x$ and
  $\tau_{\mathrm{pois}}^-(x) = \lambda$ in all cases for $x\in\N$, and
  0 elsewhere.
\end{exm}

Stein operators satisfy the product rule
\begin{align*}
\mathcal{T}_p^\ell(f(x)g(x-\ell)) %\label{eq:product}
& = \big( \mathcal{T}_p^\ell f(x)\big) g(x) + f(x) \big(\Delta^{-\ell}g(x)\big).  
\end{align*}
for all $f, g$.  This observation leads to the next definition:

\begin{defn}[Standardizations of the operator]\label{def:standop} 
  Let $\mathrm{dom}(\mathcal{T}_p^{\ell})$ be the collection of
  functions such that $c(\cdot) p(\cdot)$ belongs to
  $\mathrm{dom}(\Delta^{\ell})$.  A \emph{standardization} of the
  canonical operator $\mathcal{T}_p^{\ell}$ is any linear operator of
  the form
  $\mathcal{A}g =
  \mathcal{T}_p^{\ell}\left(c(\cdot)g(\cdot-\ell)\right)$ for some
  $c \in \mathrm{dom}(\mathcal{T}_p^{\ell})$. That is,
    \begin{align} \label{eq:standop} \mathcal{A}g(x) =
      \mathcal{T}_p^\ell c(x) g(x) + c(x) \Delta^{-\ell}g(x). 
    \end{align}
    Given some function $c$, the corresponding \emph{standardized
      Stein class} is the collection $\mathcal{F}(\mathcal{A})$ of
    test functions $g$ such that
    $c(\cdot)g(\cdot-\ell)\in \mathcal{F}_\ell^{(1)}(p)$ and
    $c(\cdot) \Delta^{-\ell}g(\cdot) \in L^1(p)$.
\end{defn}

  By the definitions, it is evident that $\mathbb{E}[\mathcal{A}g(X)]
  = 0$ for all $g \in \mathcal{F}(\mathcal{A})$. Moreover, we have
  \begin{equation}
    \label{eq:stid3}
    \mathbb{E}[\mathcal{A}g(X)]= \mathbb{E}[c(X) \Delta^{-\ell}g(X)] +
    \mathbb{E}[\mathcal{T}_p^\ell c(X) g(X)] =0
  \end{equation}
  for all such $g$. Equation \eqref{eq:stid3} is a \emph{Stein identity};
  such identities have many applications, see
  \cite{ernst2019first,ernst2019infinite}.  Identities \eqref{eq:stid1}
  and \eqref{eq:stid2} can be seen to be of the form \eqref{eq:stid3}; hence
  these are in particular the starting point of Stein's method.

\begin{rmk}
  Another way of writing \eqref{eq:standop} is to insert
  $c= \mathcal{L}_p^\ell \eta$ in \eqref{eq:standop}, for $\eta$ well
  chosen, leading to the alternative definition
  \begin{equation}
    \label{eq:standop2}
    \mathcal{A}g (x)=
      \mathcal{T}_p^{\ell}\left(\mathcal{L}_p^\ell\eta(\cdot)g(\cdot-\ell)\right)
      (x) = \big(\eta(x) - \mathbb{E}[\eta(X)]\big) g(x) +
      \mathcal{L}^{\ell}_p\eta(x) \big(\Delta^{-\ell}g(x)\big)
\end{equation}
which acts on the Stein class $\mathcal{F}(\mathcal{A}_p^{\ell,\eta})$
of functions $g$ such that
$\mathcal{L}_p^\ell\eta(\cdot)g(\cdot-\ell) \in
\mathcal{F}_{\ell}^{(1)}(p)$. Although such operators generally have
very good properties, they do not make for a very good starting point
as we will want to consider coefficients with less regularity than
$ \mathcal{L}^{\ell}_p\eta$.
\end{rmk}

\begin{rmk}
  The most common examples of functions $c$ are $c(x) = 1$ and
  $c(x)= \tau_p^{\ell}(x)$; many other choices are of course
  possible. 
\end{rmk}

\begin{exm}[Gaussian target]\label{ex:gau2}
  Consider a Gaussian target as in Example \ref{ex:gau1}. Taking
  $c(x)=1$ in \eqref{eq:standop} (or $\eta(x) = -x$ in
  \eqref{eq:standop2}) leads to the classical operator
  $\mathcal{A}g(x) = g'(x) - x g(x)$ acting on
  $\mathcal{F}(\mathcal{A})$ the collection of test functions such
  that
  $\int_{-\infty}^{\infty}|(g(x) \varphi(x))' | \mathrm{d}x
  <\infty$ and
  $\lim_{x\to \infty} g(x) \varphi(x) =\lim_{x\to -\infty} g(x)
  \varphi(x)$. This is satisfied by all differentiable functions such
  that $g'\in L^1(\varphi)$, which is the classical class of test
  functions in this case, see e.g.\ \cite[Lemma
  3.1.2]{nourdin2012normal}.  Other choices of functions $c$ are
  possible, leading to other operators for the standard Gaussian.
  \end{exm}

  \begin{exm}[Exponential target]\label{ex:expon2}
   Consider an exponential target as in Example  \ref{ex:exp1}. 
    \begin{itemize}
    \item Taking $c(x)=1$ in \eqref{eq:standop} leads to the
      operator
      $\mathcal{A}_1g(x) = (g'(x) - \lambda g(x)) \mathbb{I}[x \ge
      0]$, acting on $\mathcal{F}(\mathcal{A}_1)$ the collection of
      test functions such that
      $\int_0^{\infty} |(\lambda g(x) e^{- \lambda x})'| \mathrm{d}x < \infty$ and
      $\lim_{x \to \infty}\lambda g(x) e^{-\lambda x} = g(0)$. In particular all
      functions $g$ such that $g(0) = 0$ and
      $g' \in L^1(p_{\mathrm{exp}})$ are in this class.
    \item Taking $\eta(x) = -x$ in \eqref{eq:standop2} (or $c(x) = x/\lambda$
      in \eqref{eq:standop}) leads to the operator
      $\mathcal{A}_2g(x) = (x/\lambda g'(x) -(x-1/\lambda)g(x)) \mathbb{I}[x \ge
        0]$ acting on $\mathcal{F}(\mathcal{A}_2)$ the collection of
      test functions such that
      $\int_0^{\infty} |(\lambda xg(x) e^{- \lambda x})'| \mathrm{d}x < \infty$ and
      $\lim_{x \to \infty} xg(x) e^{- \lambda x} = 0$. In particular all
      functions $g$ such that $xg'(x)$ are in $L^1(p_{\mathrm{exp}})$.
    \end{itemize}
  \end{exm}

  \begin{exm}[Poisson target]\label{ex:poi2}
    Consider a Poisson target as in Example  \ref{ex:poi1}.
    \begin{itemize}
    \item Taking $c(x)=1$ in \eqref{eq:standop} leads to the
      operators
      $\mathcal{A}_1^+g(x) =\big((\lambda/(x+1)-1)g(x) +
      \Delta^-g(x)\big)  \mathbb{I}[x\ge0]$ and 
      $\mathcal{A}_1^-g(x) =\big((1-x/\lambda)g(x) +
      \Delta^+g(x)\big)\mathbb{I}[x\ge0]$ 
      acting respectively
      on $\mathcal{F}(\mathcal{A}_1^+)$ the collection of test functions
      such that
      $\sum_{x=0}^{\infty} |\Delta^+(g(x) p_{\mathrm{pois}}(x))| < \infty$ and
      $\lim_{x \to \infty} g(x) p_{\mathrm{pois}}(x) = g(0) e^{-\lambda}$ (in particular all functions $g$ such that $g(0) = 0$ and
      $\Delta^+g \in L^1(p_{\mathrm{pois}})$ are in this class) and 
      $\mathcal{F}(\mathcal{A}_1^-)$ the collection of test functions
      such that
      $\sum_{x=0}^{\infty} |\Delta^-(g(x) p_{\mathrm{pois}}(x))| < \infty$ and
      $\lim_{x \to \infty} g(x) p_{\mathrm{pois}}(x) = 0$ (in
      particular all functions $g$ such that $\Delta^-g \in
      L^1(p_{\mathrm{pois}})$ are in this class).
    \item Taking $\eta(x) = -x$ in \eqref{eq:standop2} leads to the
      operators
      $\mathcal{A}_2^+g(x) =\big((\lambda-x)g(x) + x\Delta^-g(x)\big)
      \mathbb{I}[x\ge0]$ and
      $\mathcal{A}_2^-g(x) =\big((\lambda-x)g(x) + \lambda
      \Delta^+g(x)\big)\mathbb{I}[x\ge0]$ acting respectively on
      $\mathcal{F}(\mathcal{A}_2^+)$ the collection of test functions
      such that
      $\sum_{x=0}^{\infty} |\Delta^+(x g(x) p_{\mathrm{pois}}(x))| <
      \infty$ and $\lim_{x \to \infty} x g(x) p_{\mathrm{pois}}(x) =0$
      and $\mathcal{F}(\mathcal{A}_2^-)$ the collection of test
      functions such that
      $\sum_{x=0}^{\infty} |\Delta^-(\lambda g(x)
      p_{\mathrm{pois}}(x))| < \infty$ and
      $\lim_{x \to \infty} \lambda g(x) p_{\mathrm{pois}}(x) = 0$.

          \end{itemize}
  \end{exm}

\begin{rmk}
  \label{remk:Gfunctions}
  If $c \in \mathcal{F}_{\ell}^{(1)}(p)$, then
  $\mathcal{\mathcal{F}}(\mathcal{A})$ \emph{always} contains the
  constant functions $g(x) = \alpha \in \R$. For instance in the
  exponential case, $\mathcal{F}(\mathcal{A}_2)$ contains constant
  functions, whereas $\mathcal{F}(\mathcal{A}_1)$ does not.
\end{rmk}

The final ingredient of the theory is a family of equations called
\emph{Stein equations}.
%We write $\bar \eta = \eta - \mathbb{E}[\eta(X)]$.

\begin{defn}[Stein equation]
  
  Let $c \in \mathrm{dom}(\mathcal{T}_p^{\ell})$ be such that
  $c(x)\neq 0$ for all $x \in \mathrm{int}(\mathcal{S}(p))$ the
  interior of the support {(in the discrete case we call
    $\{a+1,\ldots,b-1\}$ the interior)}.  The $c$-Stein equation for
  $p$ is
  \begin{equation}
    \label{eq:stek}
    \mathcal{T}_p^{\ell} c(x) g(x) + c(x) \Delta^{-\ell} g(x) = h(x) -
    \mathbb{E}\big[h(X)\big]  =: \bar{h}(x)
  \end{equation}
  considered at all $x \in \mathcal{S}(p)$.  
\end{defn}

% \begin{rmk}
% Equivalently, one could chose    $h$ a function in
%   $L^1(p)$, the $\eta$-Stein equation for $p$ is 
%   \begin{equation}
%     \label{eq:stek*}
%     \left( \eta(x)-\E[\eta(X)] \right) g(x) + \mathcal{L}_p^\ell \eta(x) \Delta^{-\ell} g(x) = h(x) - \mathbb{E}\big[h(X)\big] % =: \bar{h}(x)
%   \end{equation}
%   considered at all $x \in \mathcal{S}(p)$.  
% \end{rmk}

In \cite[Lemma 2.11]{ernst2019first} we provide conditions under which,
for any $h \in L^1(p)$, there exists a solution
$g \in \mathcal{F}( \mathcal{A})$ to \eqref{eq:stek} and
\eqref{eq:stek1} whose derivative is well defined almost everywhere.
\begin{lem}[Stein solution]\label{lem:steinsolutttt}
The solution to \eqref{eq:stek}  is $g_h^{p, \ell, c}=:g$ defined by
\begin{align}
  \label{eq:solstek1}
 g (x) 
&  = \frac{\mathcal{L}_p^{\ell} h(x+ \ell)}
          {c(x+\ell)}.
\end{align}
with the convention that $g(x) = 0$ for all $x+\ell$
outside of $\mathcal{S}(p)$. This function admits a derivative 
defined almost everywhere as                   
\begin{align}
  \label{eq:derivsolstek1}
  \Delta^{-\ell} g(x) &  =  \frac{ \bar h(x) - \mathcal{T}_p^{\ell}
                        c(x)  g(x)}{c(x)} \\
  & =  \frac{ \bar h(x) c(x+\ell) - \mathcal{T}_p^{\ell}
                        c(x)  \mathcal{L}_p^{\ell} h(x+ \ell)}{c(x)c(x+\ell)}                   
      \label{eq:derivsolstek2}
\end{align}
at all $x \in \mathrm{int}(\mathcal{S}(p))$. Moreover, {in the
  discrete case, if $\mathcal{S}(p)=\N \cap [a,b]$, then
  $\Delta^{-\ell} g(a)=g(a+b_\ell)$ and
  $\Delta^{-\ell} g(b)=-g(b-a_\ell)$.}
\end{lem}

\begin{exm}[Gaussian target]\label{ex:gau3}
  Consider a Gaussian target as in Example \ref{ex:gau2}.  The
  operator leads to the  Stein
  equation
$  g'(x) - x g(x) = h(x) - \mathbb{E}h(X)$
  whose solution in  $\mathcal{F}(\mathcal{A})$ is given by
  \begin{equation}\label{eq:gausssoljenaimarre}
    g(x) = e^{x^2/2} \int_{-\infty}^x ( h(u) - \mathbb{E}h(X))
e^{-u^2/2}  \mathrm{d}u.
\end{equation}
Illustrations are provided for $h(x) = \mathbb{I}[x \le \xi]$
indicator of half lines in Lemma \ref{lem:halflinein} and Figure
\ref{fig:solgau}. 
  \end{exm}
  
  \begin{exm}[Exponential target]\label{ex:expon3}
    Consider an exponential target as in Example \ref{ex:expon2}. 
    The first operator $\mathcal{A}_1$ leads to the Stein equation
$   g_1'(x) -  \lambda g_1(x) = h(x) - \mathbb{E}[h(X)]$
on $[0, \infty)$ whose solution in $\mathcal{F}(\mathcal{A}_1)$ is given
by
\begin{equation}\label{eq:expsollll1}
  g_1(x) = \left( e^{\lambda x} \int_0^x (h(u) - \mathbb{E}h(X))
    e^{-\lambda u}
    \mathrm{d}u\right) \mathbb{I}[x \ge 0]. 
\end{equation}
Illustrations are provided for $h(x) = \mathbb{I}[x \le \xi]$
indicator of half lines in Lemma \ref{lem:halflinein} and Figure
\ref{fig:solexpgau}.

The second operator $\mathcal{A}_2$ leads to the Stein equation
$ x/\lambda g_2'(x) - (x-1/\lambda)g_2(x) = h(x) - \mathbb{E}[h(X)]$ (still restricted to
$[0, \infty)$) whose solution in $\mathcal{F}(\mathcal{A}_2)$ is given
by
\begin{equation}\label{eq:expsollll2}
  g_2(x) = \left( \frac{\lambda}{x} e^{ \lambda x } \int_0^x (h(u) -
    \mathbb{E}h(X))   e^{-\lambda u}
    \mathrm{d}u\right) \mathbb{I}[x \ge 0]. 
\end{equation}
Illustrations are provided for $h(x) = \mathbb{I}[x \le \xi]$
indicator of half lines in Lemma \ref{lem:halflinein} and Figure
\ref{fig:solexp2}. 
\end{exm}

  \begin{exm}[Poisson target]\label{ex:poi3}
    Consider a Poisson target as in Example \ref{ex:poi2}. 
The first operators $\mathcal{A}_1^+$ and $\mathcal{A}_1^-$ leads to the Stein equations $(\lambda/(x+1)-1)g(x)+\Delta^-g(x)=h(x)-\E[h(X)]$ and $(1-x/\lambda)g(x)+\Delta^+g(x)=h(x)-\E[h(X)]$ on positive integers whose solutions in  $\mathcal{F}(\mathcal{A}_1^+)$ and $\mathcal{F}(\mathcal{A}_1^-)$ are given by 
\begin{align*}
  g^+_1(x) 
  %&= \left( e^\lambda x! \lambda^{-x} \sum_{j=0}^x (h(u) - \mathbb{E}h(X)) e^{-\lambda} \lambda^u (u!)^{-1} \right) \mathbb{I}[x \ge 0] \\
  &= \left( \frac{1}{p_{\mathrm{pois}}(x+1)} \sum_{j=0}^x (h(j) - \mathbb{E}h(X)) p_{\mathrm{pois}}(j)) \right) \mathbb{I}[x \ge 0] ,\\  
  g^-_1(x) &=  \left( \frac{1}{p_{\mathrm{pois}}(x-1)} \sum_{j=0}^{x-1} (h(j) - \mathbb{E}h(X)) p_{\mathrm{pois}}(j)) \right) \mathbb{I}[x > 0] .
\end{align*}  
  Illustrations are provided for the point mass $h(x) = \mathbb{I}[x = \xi]$ in Lemma \ref{lem:pointmass} and Figure \ref{fig:solpoi}.  
  
  The other operators $\mathcal{A}_2^+$ and $\mathcal{A}_2^-$ leads to the Stein equations $(\lambda-x)g(x)+x\Delta^-g(x)=h(x)-\E[h(X)]$ and $(\lambda-x)g(x)+\lambda\Delta^+g(x)=h(x)-\E[h(X)]$ on positive integers whose solutions in  $\mathcal{F}(\mathcal{A}_2^+)$ and $\mathcal{F}(\mathcal{A}_2^-)$ are given by 
\begin{align}
  g^+_2(x) 
  &= \left( \frac{1}{(x+1) p_{\mathrm{pois}}(x+1)} \sum_{j=0}^x (h(j) - \mathbb{E}h(X)) p_{\mathrm{pois}}(j)) \right) \mathbb{I}[x \ge 0], \label{eq:jnm1}\\  
  g^-_2(x) &=  \left( \frac{1}{\lambda p_{\mathrm{pois}}(x-1)}
           \sum_{j=0}^{x-1} (h(j) - \mathbb{E}h(X))
           p_{\mathrm{pois}}(j)) \right) \mathbb{I}[x> 0]. \label{eq:jnm2}
\end{align}  
Illustrations are provided for the point mass $h(x) = \mathbb{I}[x = \xi]$ in Lemma \ref{lem:pointmass}.   
  \end{exm}

  In the sequel we shall focus on four different classes of test
  functions $\mathcal{H}$: (i) Lipschitz, (ii) indicators of Borel
  sets, (iii) indicators of half-lines, and (iv) Dirac deltas. As
  mentioned in the Introduction, these choices correspond in the
  Steinian approach to some of the more classical integral probability
  metrics, namely the Wasserstein distance (case (i)), the total
  variation distance (cases (ii) and (iv), and the Kolmogorov
  distance, case (iii). There is, however, in principle no need to
  restrict only to this choice of classes of test functions.

\subsection{The  solutions to Stein equations}
\label{sec:repr}

We study the solutions $g_h$ and their derivatives
$\Delta^{-\ell} g_h$ from Lemma \ref{lem:steinsolutttt}.

\begin{lem}[Lower half-line indicators, $\ell = 0$]\label{lem:halflinein}
  Let $\ell=0$ (i.e.\ $p$ is absolutely continuous w.r.t.\ the
  Lebesgue measure). If $h(x) = \mathbb{I}[x \le \xi]$, the Stein
  equation \eqref{eq:stek} for $p$ is
\begin{equation*} %\label{eq:eq-indicatorHL}
\mathcal{T}_p^0c(x) g(x) + c(x) g'(x) = \mathbb{I}[x\le \xi] - P(\xi).
\end{equation*}
The solutions \eqref{eq:solstek1} are 
\begin{equation}
  \label{eq:sol-indicatorHL}
  g (x) = \frac{1}{c(x)}\frac{P(\xi \wedge x)\bar P(\xi \lor x)}{ p(x)}
\end{equation}
still with the convention that the functions are set to 0 outside
support of $p$. The derivatives \eqref{eq:derivsolstek1} of these
solutions are
\begin{equation}
  \label{eq:derivsol-indicatorHL}
  g'(x) = \frac{\mathbb{I}[x \le \xi] - P(\xi)}{c(x)} - \frac{ \mathcal{T}_p^0c(x)}{c^2(x)}  \frac{P(\xi \wedge
    x)\bar P(\xi \lor x)}{ p(x)}. 
\end{equation}
\end{lem}

\begin{figure}
  \centering
    \includegraphics[width=0.4\textwidth]{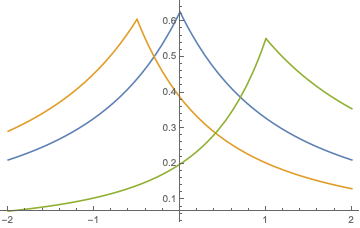}
    \includegraphics[width=0.4\textwidth]{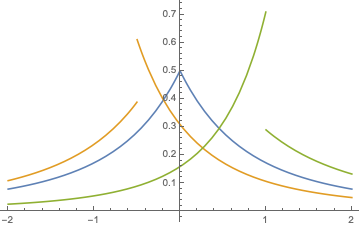}
  \caption{\small Solution \eqref{eq:sol-indicatorHL} (left plot) and
    absolute value of its derivative \eqref{eq:derivsol-indicatorHL}
(right plot) for Gaussian target with $c(x) = 1$ and, in both
plots,  $\xi=-0.5$
(orange curves), $\xi=0$ (blue curves) and $\xi = 1$ (green curves)}
  \label{fig:solgau}
\end{figure}

\begin{figure}
  \centering 
  \includegraphics[width=0.4\textwidth]{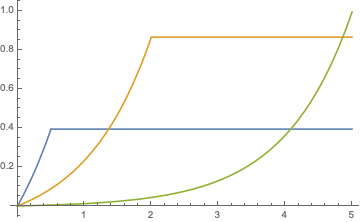}
    \includegraphics[width=0.4\textwidth]{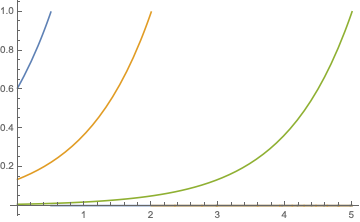}
  \caption{\small  Solution \eqref{eq:sol-indicatorHL} (left plot) and
    absolute value of its derivative \eqref{eq:derivsol-indicatorHL} 
    (right plot) for exponential target with $c(x) = 1$ and, in both
plots,  $\xi=0.5$
(blue curves), $\xi=2$ (orange curves) and $\xi = 5$ (green curves) }
  \label{fig:solexpgau}
\end{figure}

\begin{figure}
  \centering
  \includegraphics[width=0.4\textwidth]{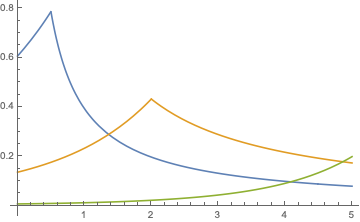}
    \includegraphics[width=0.4\textwidth]{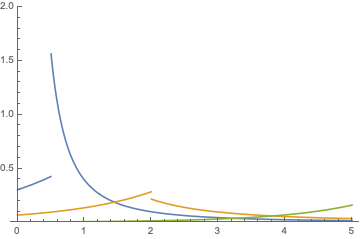}
    \caption{\small Solution \eqref{eq:sol-indicatorHL} (left plot)
      and and absolute value of its derivative
      \eqref{eq:derivsol-indicatorHL} (right plot) for exponential
      target with $c(x) = x$ and, in both plots, $\xi=.5$ (blue
      curves), $\xi=2$ (orange curves) and $\xi = 5$ (green curves)}
  \label{fig:solexp2}
\end{figure}

\begin{lem}[Point mass, $\ell = \pm 1$] \label{lem:pointmass} Let
  $\ell=\pm1$ (i.e.\ $p$ is absolutely continuous w.r.t.\ the counting
  measure).  Let $P$ be the cdf of $p$ and
  $h(x)=\mathbb{I}[x=\xi]$. The Stein equation \eqref{eq:stek} for $p$
is 
\begin{equation}\label{eq:eq-indicator1}
 \mathcal{T}_p^\ell c(x)  g(x) + c(x) \Delta^{-\ell}g(x) = \mathbb{I}[x=\xi] - p(\xi)
\end{equation}
and the solutions \eqref{eq:solstek1} are given by
\begin{equation}  \label{eq:gpointmass}
g^\ell_\xi(x) % =
% \frac{p(\xi)}{c(x+\ell)p(x+\ell)}\left(P(x-b_\ell)-\mathbb{I}[x\geq
%   \xi+b_\ell]\right) 
%= \begin{cases}
%\frac{p(\xi)}{c(x+\ell)p(x+\ell)} P(x-b_\ell)>0 & \text{ if }  x < \xi+b_\ell
%\\
%-\frac{p(\xi)}{c(x+\ell)p(x+\ell)}(1-P(x-b_\ell))<0 & \text{ if } x \geq \xi+b_\ell,
%\end{cases}
= \frac{p(\xi)}{c(x+\ell)p(x+\ell)}\left(\mathbb{I}[x\geq \xi+b_\ell] - P(x-b_\ell)\right)
\end{equation}
If, moreover, $ c =  \tau^{\ell}_p$ then the derivatives
\eqref{eq:derivsolstek1} satisfy
\begin{equation}\label{eq:dersol-indicator}
  \Delta^{-} g^+_\xi(x) = \Delta^{+} g^-_\xi(x) =
  \frac{\mathbb{I}[x=\xi]-p(\xi)}{\tau_p^+(x)} + \frac{p(\xi)(\mathbb{I}[x\geq \xi]-P(x))}{p(x)}\left(\frac{1}{\tau_p^-(x)}-\frac{1}{\tau_p^+(x)}\right)
\end{equation}
\end{lem}

\begin{figure}
  \centering
  \includegraphics[width=0.4\textwidth]{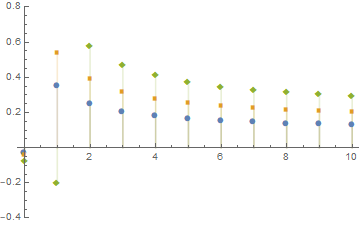}
  \includegraphics[width=0.4\textwidth]{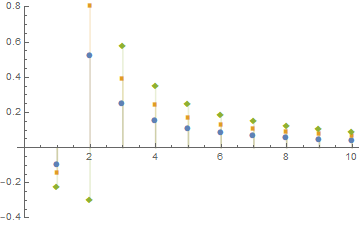}
    
  \includegraphics[width=0.4\textwidth]{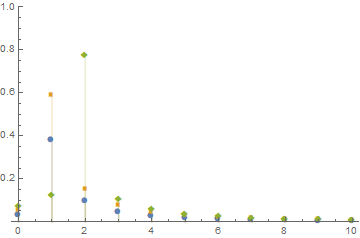}
  \includegraphics[width=0.4\textwidth]{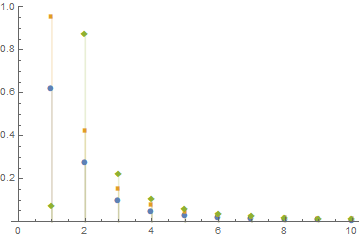}
  \caption{\small Solutions \eqref{eq:gpointmass} (upper panels) and and
    absolute value of their derivatives \eqref{eq:dersol-indicator}
    (lower panels) for Poisson target of parameter 3 with $c(x)=1$,  $\ell = 1$ (left plot) and $\ell =-1$ (right plot) and, in all plots, $\xi=0.5$
(blue curves), $\xi=1$ (orange curves) and $\xi = 2$ (green curves)}
  \label{fig:solpoi}
\end{figure}

\begin{rmk}\label{rmk:borel}
  The result of point mass can easily be extended to any Borel
  set. Following the proof of \cite[Lemma 1.1.1]{barbour1992poisson},
  for any $A\subset \mathcal{S}(p)$, the Stein equation
  \eqref{eq:stek} for $p$ is
\begin{equation*}
\mathcal{T}_p^{\ell}c(x) g(x) + c(x) \Delta^{-\ell}g(x) = \mathbb{I}_A(x) - \E[\mathbb{I}_A(X)]
\end{equation*}
and the solutions \eqref{eq:solstek1} are now given by
\begin{align*}  
g^\ell_A(x) &=
\frac{1}{c(x+\ell)p(x+\ell)}\left(
\sum_{\xi\in A} p(\xi)\mathbb{I} [x\geq \xi+ b_\ell] - P(x-b_\ell)\E[\mathbb{I}_A(X)]\right) 
 = \sum_{\xi\in A} g_\xi (x)
\end{align*}
if $g_\xi$ is the solution of Stein equation \eqref{eq:eq-indicator1} for the point mass function $h_\xi(x)=\mathbb{I}[x=\xi]$. 
% Therefore, $|g_A^\ell(x)| \leq \frac{P(x-b_\ell)(1-P(x-b_\ell))}{p(x)\tau_p^{-\ell}(x)}$ which will correspond to the bound $B_{1,p,\mathrm{Id}}^\ell$ defined in Proposition \ref{prop:boundsonsol}. 
\end{rmk}

Lemmas~\ref{lem:halflinein} and \ref{lem:pointmass} are facilitated by
the explicit nature of the test functions.  In order to be able to
deal with unspecified functions $h$, we first recall a result proved
in \cite{ernst2019first}, wherein it is shown that the inverse
operator \eqref{eq:caninv} admits several probabilistic
representations. Throughout the section, all results are stated with
the implicit assumption that all functions exist and that the various
expectations are defined.

\begin{lem}
  We introduce the following notations: generalized indicator
  functions
\begin{equation*}
  \chi^{\ell}(x, y) = \mathbb{I}[x + a_{\ell}\le y ] \mbox{ and }\chi^\ell(u,x,v)=\chi^\ell(u,x)\chi^{-\ell}(x,v)=\mathbb{I}[u+a_\ell
\leq x\leq v-b_\ell],
\end{equation*}
 the symmetric positive kernel
\begin{equation*}
%  \label{eq:kernel}
  \tilde{K}_p^{\ell}(x, y) = \frac{P(x\wedge y-a_{\ell})
    \bar{P}(x\lor y-a_{\ell})}{p(x)p(y)}.
\end{equation*}
Then, for all functions $h \in L^1(p)$, we have
\begin{align}
-\mathcal{L}_p^\ell h(x) 
%  \label{eq:rep_caninv2}
&  = \frac{-1}{p(x)}\mathbb{E} \left[
   \chi^{\ell}( X, x)\big(h(X)- \mathbb{E}[h(X)]\big)\right] \nonumber \\ 
    &    = \frac{-1}{p(x)}\mathbb{E} \left[
   \big(\chi^{\ell}( X, x) - \mathbb{E}[\chi^{\ell}( X, x)]\big)
      \big(h(X)- \mathbb{E}[h(X)]\big)\right] \nonumber%\label{eq:rep_caninv3} 
      \\
  &  = \frac{1}{p(x)} \mathbb{E}[(h(X_2) - h(X_1)) \chi^{\ell}(X_1, x,
    X_2)] \nonumber\\
  & = \mathbb{E} \left[\tilde{K}_p^{\ell}(x,
  X) \Delta^{-\ell}h(X)  \right]\label{eq:rep_caninv}
\end{align}
with $X_1, X_2$ independent copies of $X$. 
\end{lem}
The next useful lemma is easily proved along the same lines as the
previous one.
\begin{lem}\label{lem:Rp}
Define 
  \begin{equation*}
%    \label{eq:Rp}
R_p^{\ell}(x, y) = \chi^{-\ell}(y, x) \frac{P(y-a_{\ell})}{p(y)}  -
\chi^{\ell}(x, y) \frac{\bar P(y-a_{\ell})}{p(y)} = \frac{1}{p(y)}
\big(P(y - a_{\ell}) - \chi^{\ell}(x, y)\big).  
\end{equation*}
Then 
  \begin{equation}
    \label{eq:rep_h}
    \bar{h}(x) =  \mathbb{E}[R_p^{\ell}(x, X) \Delta^{-\ell}h(X)].
  \end{equation}
\end{lem}
\begin{rmk}\label{rmk:expK}
  It is easy to show that
  $ \mathbb{E} \big[ \tilde{K}_p^{\ell}(x, X) \big] =
  \tau_p^{\ell}(x) $ (the Stein kernel of $p$), and
  $\mathbb{E}\big[R_p^{\ell}(x, X)\big] = x- \mathbb{E}[X].$
\end{rmk}

With these notations in hand, the following result holds.
\begin{lem}[Representation formulae]\label{prop:deltagrep}
 
  The solutions \eqref{eq:solstek1}  can be
  written:
\begin{align}
 g (x)   & =
 % \frac{\mathbb{E}\big[(h(X_2)-h(X_1))  \chi^{\ell}(X_1, x+\ell, X_2)\big]}{\mathbb{E}\big[(\eta(X_2)-\eta(X_1))
 %                             \chi^{\ell}(X_1, x+\ell, X_2)\big]}
 %           \bigg(\mbox{or }
          -\frac{\mathbb{E}\big[(h(X_2)-h(X_1))  \chi^{\ell}(X_1, x+\ell,
            X_2)\big]}{p(x+\ell) c(x+\ell)} \label{eq:sol1}\\ 
 & =%  \frac{\mathbb{E}\left[ \tilde{K}_p^{\ell}(X, x+\ell)
    % \Delta^{-\ell}h(X) \right]}{\mathbb{E}\left[ \tilde{K}_p^{\ell}(X,
    % x+\ell) \Delta^{-\ell}\eta(X) \right]} \bigg(\mbox{or } 
-   \frac{\mathbb{E}\left[ \tilde{K}_p^{\ell}(x+\ell,X)
            \Delta^{-\ell}h(X) \right]}{c(x+\ell)} \label{eq:sol2}
\end{align}
The derivatives \eqref{eq:derivsolstek1} can be written:
\begin{align}
  \label{eq:derivivsolrep1}
  \Delta^{-\ell}g(x) &=  \frac{\bar{h}(x) }{c(x)}+
                       \frac{\mathcal{T}_p^{\ell}c(x)}{c(x)} \frac{
                       \mathbb{E}[(h(X_2)-h(X_1))  \chi^{\ell}(X_1, x+\ell, 
                       X_2)]}{c(x+\ell)p(x+\ell)}  \\
                     &   \label{eq:derivivsolrep2}
                       = \frac{\mathbb{E}\Big[ \big(R_p^{\ell}(x, X) c(x+\ell) +
                       \mathcal{T}_p^{\ell}c(x) \tilde{K}_p^{\ell}(x+\ell,X)
                       \big) \Delta^{-\ell}h(X) \Big]}{c(x) c(x+\ell)}.
\end{align}
If, moreover, $c \in \mathcal{F}_{\ell}^{(1)}(p)$ then, setting
$\bar\eta(x) = \mathcal{T}_p^{\ell} c(x)$, the derivatives
\eqref{eq:derivsolstek1} can further be simplified as:
\begin{align}  \label{eq:dersol1}
  \Delta^{-\ell}g(x)
& = \frac{\mathbb{E} \left[ \left(  \bar \eta(x) \big(h(X_2)-h(X_1)\big) -
                       \bar{h}(x) \big(\eta(X_2) - \eta(X_1)\big)
                       \right)\chi^{\ell}(X_1, x+\ell,
                       X_2)\right]}{p(x+\ell) \mathcal{L}_p^{\ell} \eta(x)\mathcal{L}_p^{\ell} \eta(x+\ell)} \\
  &=\frac{1}{p(x+\ell)\mathcal{L}_p^{\ell} \eta(x)\mathcal{L}_p^{\ell} \eta(x+\ell)}  \nonumber \\
&\qquad \times \Bigg(\E\left[\Delta^{-\ell} h(X) \frac{\bar{P}(X-a_\ell)}{p(X)}\chi^{\ell}(x,X)\right] \E\left[\Delta^{-\ell} \eta(X) \frac{P(X-a_\ell)}{p(X)}\chi^{-\ell}(X,x)\right] 
\nonumber \\
&\quad - \E\left[\Delta^{-\ell} h(X)
               \frac{P(X-a_\ell)}{p(X)}\chi^{-\ell}(X,x)\right]
               \E\left[\Delta^{-\ell} \eta(X)
               \frac{\bar{P}(X-a_\ell)}{p(X)}\chi^{\ell}(x,X)\right]
               \Bigg) 
               \label{eq:dersol2}.
\end{align}
\end{lem}
 
% Among all the possible choices of functions $c$ in Proposition
% \ref{prop:deltagrep}, we single out the choices $c(x) = 1$ and
% $c(x) = \tau_p^{\ell}(x)$. As mentioned previously, these lead to the
% solutions to the classical Stein equations treated in most of the
% literature.
% \begin{exm}
%   Example for Gaussian, Poisson and exponential with $c(x)=1$ and
%   $c(x) = \tau_p^{\ell}(x)$.  
% \end{exm}

\subsection{Stein factors}
\label{sec:examples}

% As mentioned in the introduction, Stein factors are bounds on the
% solutions to Stein equations.  In this section, we first use the
% formalism developed in Section \ref{sec:formalism} to propose factors
% under weak assumptions on $p$. We then illustrate on various examples,
% both discrete and continuous. As we shall see, all our bounds are
% optimal and/or improve on the available bounds from the literature.

We start with the discrete case, by following arguments in
\cite{ehm1991binomial,barbour1992poisson,erhardsson2005steins} to
obtain the following result.

\begin{lem}[Discrete case, point mass]\label{lem:pointmass2}
  Let $\ell=\pm1$.  Consider $g_{\xi}^{\ell}$ the solution to the
  Stein equation
  \begin{equation}\label{eq:eq-indicator}
\tau_p^{\ell}(x) \Delta^{-\ell}g(x)  - (x - \mathbb{E}[X])  g(x) = \mathbb{I}[x=\xi] - p(\xi)
\end{equation}
If the ratio $\frac{P(x-1)}{\tau_p^+(x)p(x)}$ is non decreasing for
$x\leq \xi$ and the ratio $\frac{1-P(x-1)}{\tau_p^+(x)p(x)}$ is non
increasing for $x> \xi$ then
%  If the ratio $\frac{P(x-b_\ell)}{\tau_p^{-\ell}(x)p(x)}$ is non decreasing for
%  $x< \xi+b_\ell$ and the ratio $\frac{1-P(x-b_\ell)}{\tau_p^{-\ell}(x)p(x)}$ is non increasing for $x \geq \xi+b_\ell$ then
%the solution for point mass equation \eqref{eq:eq-indicator} reaches his supremum at $\xi$ or $\xi+1$, i.e.,
\begin{align}\label{eq:supg-indicator}
\|g_\xi^{\ell}\|_\infty \leq \max\left\{\frac{P(\xi-1)}{\tau_p^+(\xi)}, \frac{1-P(\xi)}{\tau_p^-(\xi)}\right\},
\end{align} 
and
\begin{align}
\|\Delta g_\xi^{\ell}\|_\infty  %&= - \Delta^+g^-(\xi) 
&= \frac{P(\xi-1)}{\tau_p^+(\xi)} + \frac{1-P(\xi)}{\tau_p^-(\xi)} 
%\nonumber\\
%&= \frac{P(\xi-1)}{\tau_p^+(\xi)} + \frac{1-P(\xi)}{\tau_p^+(\xi)+\E[X]-\xi} 
\leq \begin{cases}
\frac{1-p(\xi)}{\tau_p^+(\xi)} & \text{ if } \xi\leq \E[X] \\
\frac{1-p(\xi)}{\tau_p^-(\xi)} & \text{ if } \xi\geq \E[X] \\
\end{cases}  \label{eq:supder-indicator}\\
&\leq \frac{1-p(\xi)}{\min\{\tau_p^+(\xi),\tau_p^-(\xi)\}}
\nonumber
%\label{eq:supder-indicator2}
%\leq \max\left\{\frac{1}{\tau_p^+(\xi)},\frac{1}{\tau_p^-(\xi)}\right\}.
\end{align}
More generally, for any Borel set $A$, 
\begin{align}\label{eq:supg-borel}
\|g_A^{\ell}\|_\infty \leq 
\left( \sum_{j\in A}p(j) \right) \sup_{\xi\in A}\left\{ \frac{1}{\tau_p^+(\xi)p(\xi)}, \frac{1}{\tau_p^-(\xi)p(\xi)}\right\}
\end{align} 
and
\begin{align}
\|\Delta g_A^{\ell}\|_\infty
%& \leq \sup_{\xi \in A} \sup_x |\Delta g_\xi(x)| 
\leq \sup_{\xi \in A} \left(\frac{P(\xi-1)}{\tau_p^+(\xi)} + \frac{1-P(\xi)}{\tau_p^-(\xi)}\right) 
=: \sup_{\xi \in A} B_{p}(\xi)
\label{eq:der-borel}
\end{align}
\end{lem}

\medskip For general $h$, representations \eqref{eq:sol1} to
\eqref{eq:dersol2} lead to the following bounds.

  \begin{prop} 
    \label{prop:boundsonsol} 
 Let $\kappa_1(h) = \sup_{y\in\mathcal{S}(p)}h(y) -
      \inf_{y\in\mathcal{S}(p)}h(y)$ and $\kappa_2(h) =
 \sup_{y\in\mathcal{S}(p)}|\Delta^{-\ell}h(y)|$.     Let $g$ be the function defined by \eqref{eq:solstek1}. Suppose
    that $c>0$ on the interior of the support of $p$.  Then 
\begin{enumerate}
\item \label{item:1} If $h$ is bounded then
   \begin{equation}\label{eq:bound1}
    \left| g(x) \right| \le   \kappa_1(h) \frac{P(x-b_{\ell})
      \bar{P}(x-b_{\ell})}{p(x+\ell)} \frac{1}{c(x+\ell) } 
  \end{equation}
  and
  \begin{align}
  \label{eq:nextbound}
  \left| \Delta^{-\ell}g(x) \right| \le \kappa_1(h)\frac{1}{c(x)} \left(1
    + \frac{|\mathcal{T}_p^{\ell}c(x)|}{c(x+\ell)}    \frac{P(x-b_{\ell})
      \bar{P}(x-b_{\ell})}{p(x+\ell)}     \right).
\end{align}
\item \label{item:2} If $\Delta^{-\ell}h$ exists and is bounded then
  \begin{align}\label{eq:bounddelta1}
    \left| g(x) \right|%  &\le 
   %  \|  \Delta^{-\ell}h \|_{\infty} \frac{1}{p(x+\ell)}{\mathbb{E}\left[
   %     \frac{ P(X \wedge x+\ell - a_{\ell}) \bar P(X \lor x+\ell + b_{\ell})}{p(X)}
   % \right]} \frac{1}{{-\mathcal{L}_p^\ell \eta(x+\ell)}  }   \nonumber\\
   &\leq
   \kappa_2(h)
     \frac{\tau_p^\ell(x+\ell)}{{c(x+\ell)}  }         
\end{align}
 and
  \begin{align}
    \label{eq:next9}
  \left| \Delta^{-\ell}g(x) \right| \le
    \kappa_2(h) \left(  \frac{|x -\mathbb{E}[X]|}{c(x)} +
    \frac{|\mathcal{T}_p^{\ell}c(x)|}{c(x)}
    \frac{\tau_p^{\ell}(x+\ell)}{c(x+\ell)} \right). 
  \end{align}
\end{enumerate}
 If, moreover, $c \in \mathcal{F}_{\ell}^{(1)}(p)$ is of the form
 $c = - \mathcal{L}_p^{\ell} \eta$, then the following also hold true. 
 \begin{enumerate}
 \item[3.] 

If  $h$ satisfies
  $|h(x) - h(y)| \le k |\eta(x) - \eta(y)|$ then
\begin{equation}
  \label{eq:boundgLip}
  \| g\|_{\infty} \le k.
\end{equation}
\item[4.] \label{item:4} If $h$ is bounded then
  \begin{align}
 | \Delta^{-\ell} g(x) |  
%& \le \left(\max_{y\in\mathcal{S}(p)}h(y) - \min_{y\in\mathcal{S}(p)}h(y)\right)   \\
%    & \qquad 
%     \left( |\bar \eta(x)| \frac{P(x+\ell-a_{\ell}) \bar{P}(x+\ell +
%      b_{\ell})}{p(x+\ell)}\frac{1}{\big(-\mathcal{L}_p^\ell
%      \eta(x+\ell)\big)^2  } + \frac{1}{{-\mathcal{L}_p^\ell \eta(x+\ell)}  }  
%\right)  \\
& \le \kappa_1(h)   \frac{1}{-\mathcal{L}_p^\ell \eta(x)}
     \left(1+\frac{|\bar \eta(x)| }{ -\mathcal{L}_p^\ell
      \eta(x+\ell)   }  \frac{P(x-b_{\ell}) \bar{P}(x +
      a_{\ell})}{p(x+\ell)}  
\right).                                                                                                   \label{eq:bounddelta4}     
    \end{align}
\item[5.] \label{item:5} If $\Delta^{-\ell}h$ exists and is bounded then
 \begin{align}
\nonumber |    \Delta^{-\ell} g(x) | &  \le  \kappa_2(h)
 \frac{1}{p(x+\ell)\big(-\mathcal{L}_p^\ell\eta(x)\big)\big(-\mathcal{L}_p^\ell\eta(x+\ell)\big)}  \nonumber \\
&\qquad \times
                                                                                                                  \Bigg(\E\left[\frac{\bar{P}(X+b_\ell)}{p(X)}\chi^{\ell}(x,X)\right] \E\left[\Delta^{-\ell} \eta(X) \frac{P(X-a_\ell)}{p(X)}\chi^{-\ell}(X,x)\right]
\nonumber \\
&\qquad \quad  + \E\left[\frac{P(X-a_\ell)}{p(X)}\chi^{-\ell}(X,x)\right]
               \E\left[\Delta^{-\ell} \eta(X)
               \frac{\bar{P}(X+b_\ell)}{p(X)}\chi^{\ell}(x,X)\right]
               \Bigg)\label{eq:bounddelta5}
                 \end{align}

 \end{enumerate}
\end{prop}
In order to
lighten the notations,  in the sequel we write $\kappa_j$ for
$\kappa_j(h)$, $j=1, 2$.

\begin{rmk}
  We remark that for $\ell = 0$ (the continuous case), the non uniform
  bounds in \eqref{eq:bounddelta1} and \eqref{eq:bounddelta4} are
  \emph{exactly} the optimal bounds for all Lipschitz-continuous
  functions $h$ among all bounds involving the factor
  $\kappa_2(h)=\| h'\|_{\infty}$, as demonstrated in \cite[Proposition
  3.13]{dobler2015stein}.  Taking $\ell = 0$ and $c(x) =1$ leads to
  (improvements of) the bounds discussed in
  \cite{chatterjee2011nonnormal} (see their Lemma 4.1).
\end{rmk}

\begin{rmk}\label{rem:scalesteiq} There exist many papers with bounds
  on Stein factors. There is often a difference in scaling between our
  Stein equation and the one used in those papers, that is we use some
  function $\eta$ and the literature rather uses $r \eta$ for some
  scalar factor $r \neq 0$. Such scaling obviously has an effect on
  the bounds, which have to be divided by powers of $|r|$ according to
  the occurrences of $\eta$ in their expressions. 
\end{rmk}
\begin{rmk}
  An important
  reference on Stein factors is \cite{dobler2018gamma} who consider
  the case of a gamma target. We do not recover their results exactly,
  because in that paper the equations are extended to the real
  line. See also \cite{dobler2012stein} (i.e.\ the arXiv version of
  \cite{dobler2015stein}) for an in depth first study of the problem
  of extending Stein equations outside the support of the target.
\end{rmk}

%In order to not overburden the text, we defer the statement of
%additional inequalities to the Appendix (see Appendix
%\ref{sec:some-more-ineq}). We conclude the section with examples.

\begin{exm}[Standard normal distribution] \label{ex:norm} Continuing
  Example \ref{ex:gau3}, we consider $g$ the solution to 
  \begin{equation*}
    g'(x)-xg(x) = h(x) - \mathbb{E}[h(X)]
  \end{equation*}
  given in \eqref{eq:gausssoljenaimarre}. Applying Proposition
  \ref{prop:boundsonsol}, the following holds:
  \begin{align*}
    & |g(x)|   \le \min \left( \kappa_1   \frac{\Phi(x)(1- \Phi(x))}{\varphi(x)}
      , \kappa_2  \right) \le  \min\left(\kappa_1
      \frac{1}{2}\sqrt{\frac{\pi}{2}}, \kappa_{2}\right)\\ 
    & |g'(x)| \le  \kappa_1  \left(1 + |x|\frac{\Phi(x)(1-
      \Phi(x))}{\varphi(x)} \right) \le 2 \kappa_1 \\ 
    & |g'(x)| \le 2 \kappa_2 \min\left(|x|, \frac{\int_{-\infty}^x \Phi(u)
      \mathrm{d}u\int_{x}^\infty(1- \Phi(u)) \mathrm{d}u}{\varphi(x)}\right)
      \le 2\kappa_2  \min \left( \sqrt{\frac{2}{\pi}}, |x|\right)
  \end{align*}
  To our own surprise, the first bound (both the uniform and the
  non-uniform one) appears to be a strict improvement on the known
  bound in this case, from e.g.\ \cite[Lemma 2.4]{chen2010normal} or
  \cite[Theorem 3.3.1]{nourdin2012normal}.  Each of the uniform bounds
  are equivalent to the known bound in this case; it is not clear to
  us whether the non uniform bounds are known (though, once again, we
  stress that the bounds involving $\kappa_2$ are in some sense
  available in \cite{dobler2015stein}).
\end{exm}

\begin{exm}[Exponential distribution]\label{ex:exp}
Continuing Example \ref{ex:expon3}, we consider the two different
situations. First, $g_1$ is solution to 
\begin{equation*}
 g_1'(x) -   \lambda g_1(x)  = h(x)- \mathbb{E}[h(X)] 
\end{equation*}
over the positive real line, given by \eqref{eq:expsollll1}. Applying Proposition
  \ref{prop:boundsonsol} (with $c(x) = 1$ and
  $\tau_{\mathrm{exp}}^0(x) = \lambda x$), the following holds: 
\begin{align*}
  & |g_1(x)| \le \frac{1}{\lambda} \min\left(\kappa_1( 1-e^{-\lambda x}), \kappa_2
    x\right) \\
  & |g_1'(x)| \le \min\left(\kappa_1 \left( 2-e^{-\lambda x} \right), \kappa_2 (|x-\lambda|+x)\right)
\end{align*}
Note that only items \ref{item:1} and \ref{item:2} apply because
$c(x) = 1 \notin \mathcal{F}_{\ell}^{(1)}( \mathrm{exp})$.  Second,
$g_2$ is solution to
\begin{equation*}
 \frac{x}{\lambda} g_{2}'(x) -      \left(x- \frac{1}{\lambda}\right)g_{2}(x)
 = h(x) - \mathbb{E}[h(X)]
  \end{equation*}
  over the positive real line, given  by \eqref{eq:expsollll2}. Here
  all the items of Proposition   \ref{prop:boundsonsol} apply  (with
  $c(x) = x/\lambda$), yielding 
\begin{align*}
  & |g_2(x)| \le  \min\left(\kappa_1  \frac{1-e^{-\lambda x}}{x}, \kappa_2
    \right) \\
  & |g_2'(x)| \le  \kappa_1 \frac{\lambda}{x}  \left( 1 + \bigg|x-\frac{1}{\lambda}\bigg|
    \frac{1-e^{-\lambda x}}{x} \right) \\
    & |g_2'(x)| \le   2\kappa_2 \min\left(\left| \lambda - \frac{1}{x}
      \right|,  \frac{1}{x} \left( 1 - \frac{1-e^{-\lambda x}}{\lambda x} \right)\right)
\end{align*}
The first bound is uniformly smaller than the bound $1/x$ of
\cite{chatterjee2006exponential} (bound for $\lambda=1$);  the other
bounds are of same order as \cite{chatterjee2006exponential}. 
\end{exm}

\begin{exm}[Poisson distribution] \label{ex:pois} We continue Example
  \ref{ex:poi3}. We consider the  solutions  $g^+$ and $g^-$ to 
  \begin{align*}
    &   x \Delta^-g^+(x) - \left(x-\lambda\right)g^+(x)   = h(x) - \mathbb{E}[h(X)] \\
&  \lambda \Delta^+g^-(x) -  \left(x-\lambda\right)g^-(x)  = h(x) - \mathbb{E}[h(X)] 
    \end{align*}
    given in \eqref{eq:jnm1} and \eqref{eq:jnm2}, respectively. Recall
    that $g^-$ is the classical solution to the usual equation for the
    Poisson; also $g^+(x) = g^-(x+1)$ and
    $\Delta^+g^-(x) = \Delta^-g^+(x)$.  Applying
    Proposition~\ref{prop:boundsonsol} (with $\ell = -1$ and
    $c(x) = \lambda$ or $\ell = 1$ and $c(x) = x$), the following
    holds (recall $\bar{P}(x) = 1-P(x)$):
    \begin{align}
      &     |g^-(x)| \le \min \left(
        \kappa_1\frac{P(x-1)\bar P(x-1))}{\lambda p(x-1)},
        \kappa_2   \right) \label{eq:boundgm}\\
      &  | \Delta^+g^-(x) | \le  \kappa_1 \min \left(
        \frac{1}{\lambda} +
        \frac{\left|x-\lambda\right|}{\lambda^2}
        \frac{P(x-1) \bar P(x-1)}{p(x-1)}, \frac{1}{x}+\frac{|x-\lambda|}{x(x+1)}
        \frac{P(x)\bar P(x)}{p(x+1)} \right)\label{eq:delboundgm1}\\
      &  | \Delta^+g^-(x) | \le 2 \kappa_2   \min \left(
        \frac{ |x-\lambda|}{\lambda} ,  \frac{ |x-\lambda|}{x},     \frac{\sum_{j=0}^{x-1} P(j)
         \sum_{j=x}^\infty \bar P(j)}{\lambda x p(x)}   \right) \label{eq:delboundgm2}
    \end{align}
    (we only give the bounds in terms of $g^-$; those for $g^+$ follow
    trivially).  One can see, as illustrated in Figure
    \ref{fig:poisBa}, that the non uniform bound in \eqref{eq:boundgm}
    is strictly smaller than $\min( 1, \sqrt{{2}/{(e \lambda)}})$
    which thus yields an improvement on the classical bound, e.g.\ in
    \cite[Theorem 2.3]{erhardsson2005steins}; the constant bound -- in
    terms of $\kappa_2$ -- is already available in \cite[Remark
    1.1.6]{barbour1992poisson} (proof in \cite{barbour2006stein}).
    The bound \eqref{eq:delboundgm1} is of similar order to the
    classical $(1-e^{-\lambda})/\lambda$ (see Figure
    \ref{fig:poisBb}), but does not improve everywhere.  Finally the
    bound \eqref{eq:delboundgm2} strictly improves on the bound
    $\min(1, 8/(3\sqrt{2e\lambda}))$ from \cite{barbour1992poisson},
    as illustrated in Figure \ref{fig:poisBc} for $\lambda=10$.

\begin{figure}[h]
\begin{center}
\begin{subfigure}{.3\textwidth}
  \centering
  \includegraphics[width=1\textwidth]{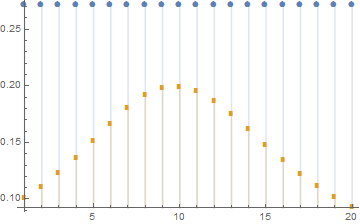}
  \caption{}
  \label{fig:poisBa}
\end{subfigure}%
\quad 
\begin{subfigure}{.3\textwidth}
  \centering
\includegraphics[width=1\textwidth]{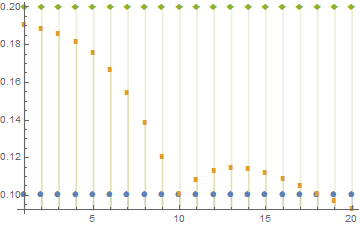}
  \caption{}
  \label{fig:poisBb}
\end{subfigure}%
\quad
\begin{subfigure}{.3\textwidth}
  \centering
\includegraphics[width=1\textwidth]{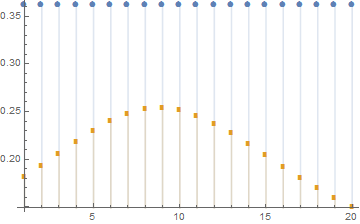}
  \caption{}
  \label{fig:poisBc}
\end{subfigure}%
\end{center}
\caption{\small Figure \ref{fig:poisBa} gives the non-uniform bound
  \eqref{eq:boundgm} (orange curve) as well as the classical bound
  $\min(1, \sqrt{2/ \lambda})$ (blue curve). Figure \ref{fig:poisBb}
  gives the non uniform bound \eqref{eq:delboundgm1} (orange curve),
  the bound $(1-e^{-\lambda})/\lambda$ (blue curve) and $2/\lambda$
  (green curve). Figure \ref{fig:poisBc} gives the non uniform bound
  \eqref{eq:delboundgm2} (orange curve) and the bound $\min(1,
  8/(3\sqrt{2e\lambda}))$ (blue curve).  
  All cases correspond to the Poisson
  distribution of parameter $\lambda=10$.}
\label{fig:poisB}
\end{figure}    

Lemma \ref{lem:pointmass2} also applies to this case, because the
Poisson distribution satisfies the conditions (monotonicity of the two
ratios for any $\xi \in \mathcal{S}(p)$).  Therefore, the bound
\eqref{eq:supg-indicator} on the solution of equation
\eqref{eq:eq-indicator} becomes:
\begin{align}\label{eq:pointmass1}
\| g_\xi\|_\infty \leq \max\left\{\frac{P(\xi-1)}{\xi}, \frac{\bar
  P(\xi-1)}{\lambda}\right\}, 
\end{align} 
as illustrated in Figure \ref{fig:poisa}.  Moreover, the bound
\eqref{eq:supder-indicator} becomes
\begin{align}\label{eq:pointmass2}
||\Delta^+ g_\xi||_\infty &
= \frac{P(\xi-1)}{\xi} + \frac{1-P(\xi)}{\lambda}
%\leq 
%\begin{cases}
%\frac{1-e^{-\lambda} \lambda^\xi / \xi!}{\xi} & \text{ if } \xi\leq \lambda \\
%\frac{1-e^{-\lambda} \lambda^\xi / \xi!}{\lambda} & \text{ if } \xi\geq  \lambda \\
%\end{cases} 
%\leq \begin{cases}
%\frac{1-e^{-\lambda}}{\xi} & \text{ if } \xi\leq \lambda \\
%\frac{1}{\lambda} & \text{ if } \xi\geq  \lambda \\
%\end{cases} \\
\leq \min\left\{\frac{1}{\xi},\frac{1-e^{-\lambda}}{\lambda}\right\} .
\end{align}
For any Borel set $A\subset \mathcal{S}(p)$, the solution is bounded by \eqref{eq:supg-borel}
\begin{align*}
\|g_A\|_\infty \leq 
\left( \sum_{j\in A}p(j) \right) \sup_{\xi\in A}\left\{ \frac{1}{\xi p(\xi)}, \frac{1}{\lambda p(\xi)}\right\}
\end{align*}
and the bound \eqref{eq:der-borel} gives
\begin{align*}
||\Delta g_A||_\infty &\leq 
\sup_{x \in A} \left(\frac{P(x-1)}{x} + \frac{1-P(x)}{\lambda}\right) 
\leq \frac{1-e^{-\lambda}}{\lambda}
\end{align*}
which is the bound given in \cite[Lemma 1.1.1]{barbour1992poisson}.

\begin{figure}[h]
  \begin{center}
    \begin{subfigure}{.4\textwidth}
  \centering
\includegraphics[width=1\textwidth]{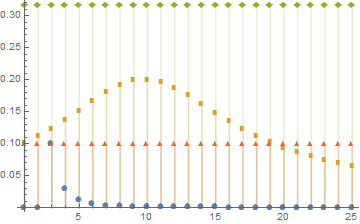}
  \caption{}
  \label{fig:poisa}
\end{subfigure}%
\quad
    \begin{subfigure}{.4\textwidth}
  \centering
\includegraphics[width=1\textwidth]{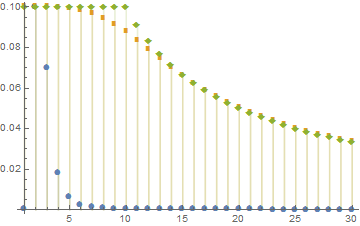}
  \caption{}
  \label{fig:poisb}
\end{subfigure}%
\end{center}
\caption{\small Figure \ref{fig:poisa} gives the numerical exact value
  of the function $|g_{\xi}|$ (blue curve), the bound
  \eqref{eq:boundgm} (orange curve), the bound \eqref{eq:pointmass1} (red
  curve) and $\min(1, 1/|\sqrt{\lambda})$ (green curve). Figure
  \ref{fig:poisb} gives the numerical exact value of the function
  $|\Delta^+ g_{\xi}|$ (blue curve), the first bound in \eqref{eq:pointmass2}
  (orange curve) and the second one (green curve).  All cases
  correspond to the Poisson distribution with parameter $\lambda =
  10$ at value $\xi=2$.} 
\label{fig:pois}
\end{figure}
\end{exm}

More examples are provided in the supplementary material, namely
uniform and non uniform Stein factors for the beta, gamma, $\chi^2$,
Student, binomial and negative binomial distributions.

\section{Bounds on IPMs and comparison of generators}
\label{sec:discr}

As described in the introduction, one of the purposes of the material
of Section \ref{sec:formalism} is to provide quantitative bounds on a
distance between an approximating distribution $X_n$, say, and a
target distribution, $X_{\infty}$. Straightforward manipulation of the
definitions lead to the following very general abstract
results.

  \begin{thm}[Stein discrepancies] \label{sec:steins-dens-appr-1}
    Let $X_n \sim p_n$ be some random variable and let $X_{\infty}$
    have canonical Stein operators
    $\mathcal{T}_{\infty}^{\ell_{\infty}}$ and
    $\mathcal{L}_{\infty}^{\ell_{\infty}}$ for some
    $\ell_{\infty} \in \left\{ -1, 0, 1 \right\}$. Then, for all
    $\eta \in L^1(p_n)$ and all $c_1 \in
    \mathrm{dom}(\mathcal{T}_\infty^{\ell})$ and $h \in
    L^1(p_{\infty})\cap     L^1(p_n)$ we have
    \begin{align}
     &  \mathbb{E}h(X_n) - \mathbb{E}h(X_{\infty}) \nonumber \\
       & \label{eq:IPM1ex}  =  
    \mathbb{E} \left[(\eta_1(X_n) - \mathbb{E}[\eta_1(X_{\infty})])
         \frac{\mathcal{L}_{\infty}^{\ell_{\infty}}h(X_n)}{\mathcal{L}_{\infty}^{\ell_{\infty}} 
         \eta(X_n)}\right] 
    + \mathbb{E}\left[\mathcal{L}_\infty^{\ell_{\infty}} \eta_1(X_n)
      \Delta^{-\ell_{\infty}} \left(\frac{\mathcal{L}_\infty^{\ell_{\infty}}
                     h(\cdot+\ell_{\infty})}{\mathcal{L}_\infty^{\ell_{\infty}}
                     \eta(\cdot+\ell_{\infty})} \right)(X_n)\right]\\
        &  \label{eq:IPM2ex}  =
 \mathbb{E}\left[\big(\mathcal{T}_{\infty}^{\ell_{\infty}} c_1(X_n)
      \big)\frac{\mathcal{L}_{\infty}^{\ell_{\infty}}h(X_n)}{c_1(X_n)}\right] +
          \mathbb{E}\left[c_1(X_n)
          \Delta^{-\ell_{\infty}}\left(\frac{\mathcal{L}_\infty^{\ell_{\infty}} 
                     h(\cdot+\ell_{\infty})}{c_1(\cdot+\ell_{\infty})} \right)(X_n) \right].
    \end{align}
  In particular the IPMs \eqref{eq:defIPM} can be written as suprema
    of either of the above.
  \end{thm}

  The generality of the expressions in \eqref{eq:IPM1ex} and
  \eqref{eq:IPM2ex} (we stress that there is basically full freedom of
  choice in the functions $\eta, c_1$ and $h$!) ensure that all first
  order Stein equations from the literature can easily be rewritten
  particularizations of these expressions. Moreover, the dependence on
  the test functions $\eta, c_1$ and $h$ is made explicit which
  therefore permits further simplifications in line with the results
  from Section \ref{sec:examples}.  It still remains, of course, to
  show that our abstract formulations actually provide some benefits.
  This we now demonstrate by concentrating on comparison of random
  variables $X_n$ and $X_{\infty}$ under the additional assumption
  that both have an accessible Stein operators. For convenience we
  also impose $\ell_n = \ell_{\infty} = \ell$ with the added
  assumption that both the target and the approximating laws are a.c.\
  with respect to the same dominating measure.  This assumption
  provides many simplifications but is in no way necessary, see Remark
  \ref{rmk:compadisc} and Example \ref{ex:binotonormsu} in the
  supplementary material.

  The first step is to
  associate to $X_n$ its Stein operators $\mathcal{T}_n^{\ell}$ and
  $\mathcal{L}_n^{\ell}$. Then we can withdraw 0 in identities
  \eqref{eq:stid1} and \eqref{eq:stid2} to obtain
  \begin{align}
&\mathbb{E}h(X_n) - \mathbb{E} h(X_{\infty})  \nonumber\\
    &  =  
    \mathbb{E} [(\eta_1(X_n)-\eta_2(X_n))g_h(X_n)] 
    + \mathbb{E}\left[(\mathcal{L}_\infty^\ell \eta_1(X_n)-\mathcal{L}_n^\ell\eta_2(X_n)\big) \Delta^{-\ell} g_h(X_n)\right]+ \kappa_{\eta_2}^\ell(h) \label{eq:ehehv1}\\
        &    =
 \mathbb{E}[\big(\mathcal{T}_{\infty}^\ell c_1(X_n)-\mathcal{T}_n^\ell c_2(X_n)
      \big)g_h^*(X_n)] + \mathbb{E}[\big(c_1(X_n) - c_2(X_n)\big) \Delta^{-\ell}g_h^*(X_n) ]
      + \kappa^{*\ell}_{c_2}(h)  \label{eq:ehehv2}  
  \end{align}
  with
  \begin{align*}     %\label{eq:jnmajnmajnma}
&    \kappa_{\eta_2}^\ell(h):=\mathbb{E} \left[
    \mathcal{T}_n^{\ell}(\mathcal{L}_n^\ell\eta_2(\cdot)g_h(\cdot-\ell))(X_n)
  \right]  +
      (\mathbb{E}[\eta_2(X_n)]-\mathbb{E}[\eta_1(X_\infty)])
      \mathbb{E}[g_h(X_n)] \\
 & \kappa^{*\ell}_{c_2}(h):=\mathbb{E} \left[
    \mathcal{T}_n^{\ell}(c_2(\cdot)g_h^*(\cdot-\ell))(X_n) \right] %\label{eq:jnmajnmajnma2}
\end{align*}
and where the choice of $c_1, c_2, \eta_1$ and $\eta_2$ are left free
up to validation of easily verified technical conditions.  If
$\mathcal{F}(\mathcal{A}_n^{\ell,\eta_2})$ contains $g_h$ then
$ \kappa_{\eta_2}^\ell(h)=0$. Similarly, if
$\mathcal{F}^*(\mathcal{A}_n^{\ell,c_2})$ contains $g_h^*$, then
$\kappa^{*\ell}_{c_2}(h) =0$. In all cases, if the approximation
problem is reasonable, these remainder terms should be small.
Particularizing to the choice $c_1 = c_2 = 1$ and
$\eta_1 = \eta_2 = - \mathrm{Id}$ (again, this is arbitrary and
alternative options are available, see Examples
\ref{sec:frechet-distribution} and \ref{ex:maxfre} in the
supplementary material), we obtain one of the main results of the
paper.
\begin{thm}\label{prop:EhEh}
  Suppose that $X_n \sim p_n$ and $X_{\infty} \sim p_\infty$ are
  absolutely continuous w.r.t.\ the same dominating measure. For all
  $h \in L^1(p_{\infty})\cap L^1(p_n)$ we have
\begin{align}
\mathbb{E}h(X_n) - \mathbb{E}h(X_{\infty}) 
&= \mathbb{E}\left[ \left( \rho^\ell_\infty(X_n) - \rho^\ell_n(X_n) \right) \mathcal{L}_\infty^\ell h(X_n+\ell)\right]+ \kappa_1^{*\ell}(h) \label{eq:EhEh1} 
\end{align}
with
\begin{equation*}%\label{eq:14}
  \kappa_1^{\star \ell}(h)  = \mathbb{E} \left[
    \mathcal{T}_n^{\ell} \mathcal{L}_{\infty}^{\ell}h(X_n) \right].
\end{equation*}
Furthermore, if $\eta=\mathrm{Id}\in L^1(p_\infty)$, setting  
$\mu_n = \mathbb{E}[X_n]$ and $\mu_{\infty} = \mathbb{E}[X_{\infty}]$
we get 
\begin{align}
\mathbb{E}h(X_n) - \mathbb{E}h(X_{\infty}) 
&= \mathbb{E}\left[\big(\tau_n^\ell(X_n)-\tau_\infty^\ell(X_n) \big)
                                              \Delta^{-\ell}\left(\frac{-\mathcal{L}_\infty^\ell
                                              h(\cdot+\ell)}{\tau_\infty^\ell(\cdot+\ell)}
                                              \right)(X_n) \right]+ 
                                              \kappa_\mathrm{Id}^\ell(h)  \label{eq:EhEh2}
\end{align}
with
\begin{equation*} %\label{eq:15}
\kappa_\mathrm{Id}^\ell(h) = \mathbb{E} \left[
    \mathcal{T}_n^{\ell}\left(\frac{\tau_n^{\ell}(\cdot)}{\tau_\infty^{\ell}(\cdot)}
    \mathcal{L}_{\infty}^{\ell}h(\cdot)\right)(X_n)  \right] +
                                              (\mu_n - \mu_{\infty})
                                              \mathbb{E} \left[ \frac{-\mathcal{L}_\infty^\ell
                                              h(X_n+\ell)}{\tau_\infty^\ell(X_n+\ell)}
                                              \right]. 
\end{equation*}
\end{thm}   
Clearly expressions such as those in Theorem
\ref{sec:steins-dens-appr-1} and \ref{prop:EhEh} will only be useful
if the different functions involved are tractable. In the next section
and in the supplementary material we show that this is the case for
many important examples.  We now specialize Theorem \ref{prop:EhEh} to
various situations of interest, that is for Kolmogorov, Total
Variation and Wasserstein metrics; in particular, setting
$A_{n}^{\infty} =\left\{ x \, | \, p_n(x) \ge p_{\infty}(x) \right\}$
and
$h_{\mathrm{TV}}(x) = \mathbb{I}_{A_{n}^{\infty}}(x) -
\mathbb{I}_{(A_{n}^{\infty})^c}(x) = 2 \mathbb{I}_{A_{n}^{\infty}}(x)
-1$, we reap
    \begin{align*} %\label{eq:tvdirectstein}
      \mathrm{TV}(X_n, X_{\infty})  &= \sup_B \left| {P}_n( B) - {P}_{\infty}(X_B) \right|  = \frac{1}{2}\int \left| p_n(x) -
                                      p_{\infty}(x) \right|
                                      \mu(\mathrm{d}x) \nonumber\\
      %      
      % &= \mathbb{P}(X_n \in A) - \mathbb{P}(X_{\infty} \in A)  \\
% &       = \E\left[\mathbb{I}_A(X_n) - \mathbb{I}_A(X_{\infty}) \right] 
% %      =\frac{1}{2} \int_{-\infty}^{\infty} |p_n(x) - p_{\infty}(x)| \mu(\mathrm{d}x) 
&        = \frac{1}{2}(\mathbb{E} h_{\mathrm{TV}}(X_n) -
                                                                                           \mathbb{E}h_{\mathrm{TV}}(X_{\infty}))
= \mathbb{E}[ \mathbb{I}_{A_{n}^{\infty}}(X_n)]-\mathbb{E}[ \mathbb{I}_{A_{n}^{\infty}}(X_\infty)]
    \end{align*}
    (here and throughout we write
    $P(B) = \mathbb{E}[ \mathbb{I}_B(X)]$ if $X$ has cdf $P$).
    Although the set $A_{n}^{\infty}$ is intractable, this last
    rewriting allows to avoid having a supremum in our Stein
    discrepancy (we work with a single indicator function) and thus
    leads to improved bounds.

\begin{rmk}\label{rmk:compadisc}
  It is immediate to extend the scope of Theorem~\ref{prop:EhEh} to
  the comparison of \emph{any} arbitrary distributions without
  requiring that they share a common dominating measure. Such has
  already been attempted successfully in \cite{goldstein2013stein} and
  our notations would allow to perform similar operations in full
  generality. We present an outline of such a ``general'' bound as
  well as two simple applications (one towards extreme value
  distributions and one towards normal approximation) at the end of
  Section \ref{sec:bounds-ipms-1} of the suppelementary material.
\end{rmk}

\begin{cor}[Identity \eqref{eq:EhEh1}, score functions and $\ell = 0$]
  \label{cor:kolbound}
  Suppose that the laws of $X_n$ and $X_{\infty}$ are absolutely
  continuous with respect to the Lebesgue measure with densities $p_n$
  and $p_{\infty}$, respectively. Let $\mathcal{S}_n$ (resp.,
  $S_{\infty}$) be the support of $p_n$ (resp., $p_{\infty}$); also
  let 
  $b_n = \sup \mathcal{S}_n$ and
  $a_n = \inf \mathcal{S}_n $ (resp.,   $b_\infty = \sup \mathcal{S}_\infty$ and
  $a_\infty = \inf \mathcal{S}_\infty $).  Finally, let $\rho_n(x)$ and
  $\rho_\infty(x)$ be the scores and $\tau_n(x)$ and
  $\tau_{\infty}(x)$ be the Stein kernels of $p_n$ and $p_{\infty}$.
  \begin{enumerate}
  \item   The Kolmogorov distance
  between the random variables $X_n$ and $X_\infty$ is
% \begin{align}
% &\mathrm{Kol}(X_n,X_\infty)\\
%  &= \sup_z \left| \mathbb{E}\left[ (\rho^\ell_\infty(X_n) -
%    \rho^\ell_n(X_n)) \frac{P_\infty((X_n-b_\ell)\wedge z)\bar P_\infty((X_n-b_\ell))\vee z)}{p_\infty(X_n+\ell)}
%       \mathbb{I}[(X_n+\ell)\in\mathcal{S}(p_\infty)] \right]
%     + \kappa_1^{*\ell}(h_z) \right|     
%     \label{eq:corkol}
% \end{align}
\begin{align}
  \mathrm{Kol}(X_n,X_\infty)& = \sup_z \left| \mathbb{E}\left[ (\rho_\infty(X_n) -
    \rho_n(X_n)) \frac{P_\infty(X_n\wedge z)\bar P_\infty(X_n\vee
                              z)}{p_\infty(X_n)} \mathbb{I}_{\mathcal{S}_{\infty}}(X_n) \right]
    + \kappa_1^{\star}(z) \right|     
                              \label{eq:corkol}\\
  & \le  \mathbb{E}\left[ |\rho_\infty(X_n) -
    \rho_n(X_n)| \frac{P_\infty(X_n)\bar P_\infty(X_n)}{p_\infty(X_n)}\mathbb{I}_{\mathcal{S}_{\infty}}(X_n)
    \right] + \sup_z \kappa_1^{\star}(z)
\end{align}
where 
\begin{align*}
  \kappa_1^{\star}(z) = \lim_{x \nearrow b_n\wedge b_{\infty}}\frac{p_n(x)}{p_\infty(x)} P_\infty(x\wedge z)\bar P_\infty(x\vee
  z)- \lim_{x \searrow a_n\vee a_{\infty}}\frac{p_n(x)}{p_\infty(x)}P_\infty(x\wedge z)\bar P_\infty(x\vee
  z).
\end{align*}

\item  The Total
  Variation distance between $X_n$ and $X_\infty$ is
    \begin{align}
      &      \mathrm{TV}(X_n, X_{\infty})  \nonumber \\
      & 
        = \mathbb{E}\left[ \left( {\rho_\infty(X_n) - \rho_n(X_n)}
        \right)      \frac{ {P}_\infty(A_{n}^{\infty}\cap (-\infty,
      X_n])-{P}_\infty (A_{n}^{\infty})
                                                P_{\infty}(X_n)}{p_\infty(X_n)}\mathbb{I}_{\mathcal{S}_{\infty}}(X_n)\right] 
        + \kappa_1^{\star}(\mathbb{I}_{A_{n}^{\infty}}) \label{eq:tv1} \\
      &       \leq \mathbb{E}\left[ \left|{\rho_\infty(X_n) -
        \rho_n(X_n)} \right| \frac{P_\infty(X_n)\bar P_\infty(X_n)}{p_\infty(X_n)} \mathbb{I}_{\mathcal{S}_{\infty}}(X_n)\right]
        + \kappa_1^{\star}(\mathbb{I}_{A_{n}^{\infty}})
    \label{eq:tvbound1}
    \end{align}
    where
    $A_{n}^{\infty} =\left\{ x \, | \, p_n(x) \ge p_{\infty}(x)
    \right\}$, $X_1, X_2 \stackrel{\mathrm{iid}}{\sim} p_\infty$, and
    \begin{align*}
\kappa_1^{\star}(\mathbb{I}_{A_{n}^{\infty}}) & = \lim_{x \nearrow
                                                b_n\wedge b_{\infty}}\frac{p_n(x)}{p_\infty(x)}
      \left( {P}_\infty(A_{n}^{\infty}\cap (-\infty,
      x])-{P}_\infty (A_{n}^{\infty})
                                                P_{\infty}(x)\right)
                                                \nonumber \\
      & \quad  - \lim_{x \searrow a_n \vee a_{\infty}}\frac{p_n(x)}{p_\infty(x)}
      \left( {P}_\infty(A_{n}^{\infty}\cap (-\infty,
      x])-{P}_\infty (A_{n}^{\infty})  P_{\infty}(x)\right).
    \end{align*}

  \item The Wasserstein distance between $X_n$ and $X_\infty$ is
    \begin{align}
      \mathrm{Wass}(X_n, X_{\infty}) &  
%      = \sup_{h \in \mathrm{Lip}(1)}  
%      \left| \mathbb{E}\left[ \left(\frac{\Delta^\ell p_\infty(X_n)}{p_\infty(X_n)} - \frac{\Delta^\ell p_n(X_n)}{p_n(X_n)} \right) \Delta^{-\ell}h(X_\infty) \tilde{K}_\infty(X_\infty, X_n+\ell)\right]
%                                       \right|\\
% & \leq \mathbb{E}\left[  \left|\frac{\Delta^\ell p_\infty(X_n)}{p_\infty(X_n)} - \frac{\Delta^\ell p_n(X_n)}{p_n(X_n)} \right|  \tilde{K}_\infty(X_\infty, X_n+\ell)\right] \\
% & = \mathbb{E}\left[  \left|\frac{\Delta^\ell p_\infty(X_n)}{p_\infty(X_n)} - \frac{\Delta^\ell p_n(X_n)}{p_n(X_n)} \right| \tau_\infty^\ell(X_n+\ell) \right]
      = \sup_{h \in \mathrm{Lip}(1)}  
      \left| \mathbb{E}\left[
    \left({\rho_n(X_n)-\rho_\infty(X_n)} \right)
    h'(X_\infty) \tilde{K}_\infty(X_\infty,
    X_n)\mathbb{I}_{\mathcal{S}_{\infty}}(X_n)\right] + \kappa_1^{\star}(h)
                                       \right| \label{eq:corwass1}\\
 % & \leq \mathbb{E}\left[  \left|\ma{\rho_n(X_n)-\rho_\infty(X_n)} \right|  \tilde{K}_\infty(X_\infty, X_n)\right] \\
 & \le  \mathbb{E}\left[  \left|{\rho_n(X_n)-\rho_\infty(X_n)} \right|
   \tau_\infty(X_n) \mathbb{I}_{\mathcal{S}_{\infty}}(X_n) \right]+ \sup_{h \in \mathrm{Lip}(1)}\kappa_1^{\star}(h) \label{eq:corwass1bis}
    \end{align}
    where
    \begin{align*}
\kappa_1^{\star}(h)  &  = \lim_{x \searrow a_n \vee a_{\infty}}
      \frac{p_n(x)}{p_{\infty}(x)}  \int_{a_{\infty}}^{b_{\infty}} h'(u)
       {P_{\infty}(x \wedge u) \bar P_{\infty}(x \vee u) } \mathrm{d}u
                       \nonumber 
      \\
      & \quad  -  \lim_{x \nearrow b_n \wedge b_{\infty}} \frac{p_n(x)}{p_{\infty}(x)}  \int_{a_{\infty}}^{b_{\infty}}h'(u)
       {P_{\infty}(x \wedge u) \bar P_{\infty}(x \vee u) } \mathrm{d}u
    \end{align*}
  \end{enumerate}
\end{cor}

% \begin{rmk}\label{rmk:whoseinf}
%   In the assumptions we imposed that the support of $p_n$ be a subset
%   of that of $p_{\infty}$; this is for notational convenience only. If
%   this is not the case, one must simply take care of restricting the
%   integrals to the correct support. An illustration shall be given in
%   Section \ref{sec:beta-vs-gamma}.
% \end{rmk}

\begin{rmk}[Distances between nested distributions]
 
  Inspired by \cite{ley2017distances} we know that it is of interest
  to consider situations where
  $\mathcal{S}(p_n) \subseteq \mathcal{S}(p_{\infty})$. Then, setting
  $\pi_0(x) = p_n(x)/p_{\infty}(x)$ and
  $\rho_0^{\ell}(x) = \Delta^{\ell}p_0(x)/p_0(x)$ we get
  \begin{equation*}
    \rho_n(x) - \rho_{\infty}(x) = \frac{p_{\infty}(x+\ell)}{p_{\infty}(x)}
  \rho_0^{\ell}(x)
  \end{equation*}
 for all $x \in \mathcal{S}_n(x)$. If $\ell = 0$ then 
  $\frac{ p_{\infty}(X_n+\ell)}{p_{\infty}(X_n)} = 1$.
\end{rmk}

\begin{cor}[Identity \eqref{eq:EhEh1}, score functions, $\ell =
  \pm1$] \label{cor:prec} Suppose that the laws of $X_n$ and
  $X_{\infty}$ are discrete with mass functions $p_n$ and
  $p_{\infty}$, respectively. Let
  $b_n = \sup \mathcal{S}(p_n) \le b_{\infty} = \sup
  \mathcal{S}(p_\infty) $ and
  $a_n = \inf \mathcal{S}(p_n) \ge a_\infty = \inf
  \mathcal{S}(p_\infty)$.  Finally, let $\rho_n^{\ell}(x)$ and
  $\rho_\infty^{\ell}(x)$ be the scores and $\tau_n^{\ell}(x)$ and
  $\tau_{\infty}^{\ell}(x)$ be the Stein kernels of $p_n$ and
  $p_{\infty}$. The following results hold true.
  \begin{align*}
       &      \mathrm{TV}(X_n, X_{\infty})  \nonumber \\
      & 
        = \mathbb{E}\left[ \left( {\rho_\infty^{\ell}(X_n) - \rho_n^{\ell}(X_n)}
        \right)  \mathbb{I}_{\mathcal{S}_{\infty}}(X_n+\ell) 
        \frac{ {P}_\infty(A_{n}^{\infty}\cap (-\infty,
      X_n-b_{\ell}])-{P}_\infty (A_{n}^{\infty})
                                                P_{\infty}(X_n-b_{\ell})}{p_\infty(X_n+\ell)}
        \right] 
        + \kappa_1^{\star\ell}(\mathbb{I}_{A_{n}^{\infty}})
        % \label{eq:tv2} 
        \\
      &       \leq \mathbb{E}\left[ \left|{\rho_\infty^{\ell}(X_n) -
        \rho_n^{\ell}(X_n)} \right| \frac{P_\infty(X_n-b_{\ell})\bar
        P_\infty(X_n-b_{\ell})}{p_\infty(X_n+\ell)}\mathbb{I}_{\mathcal{S}_{\infty}}(X_n+\ell) \right] 
        + \kappa_1^{\star\ell}(\mathbb{I}_{A_{n}^{\infty}})
  %  \label{eq:tvbounddisc2}
  \end{align*}
  with
\begin{align*}
  \kappa_1^{\star+}(\mathbb{I}_{A_{n}^{\infty}}) & = - \lim_{x
                                                   \searrow a_n \vee a_{\infty}}\frac{p_n(x)}{p_\infty(x)}
                                                      \left( {P}_\infty(A_{n}^{\infty}\cap (-\infty,
                                                      x-1])-{P}_\infty
                                                      (A_{n}^{\infty})
                                                      P_{\infty}(x-1)\right)\\
    \kappa_1^{\star-}(\mathbb{I}_{A_{n}^{\infty}}) & = \lim_{x
                                                     \nearrow
                                                     b_n\wedge b_{\infty}}\frac{p_n(x)}{p_\infty(x)}
                                                      \left( {P}_\infty(A_{n}^{\infty}\cap (-\infty,
                                                      x])-{P}_\infty (A_{n}^{\infty})
                                                      P_{\infty}(x)\right)
                                                    \end{align*}

\end{cor}
 
It is not hard to obtain bounds on Total Variation, Kolmogorov and
Wasserstein by starting from identity \eqref{eq:EhEh2} through Stein
kernels. It is also well-documented that such bounds are, in many
cases, useful; we refer e.g.\ to Nourdin and Peccati's important
Malliavin Stein method (\cite{nourdin2012normal}) for applications of
the corresponding bounds in the standard normal case. However, in our
applications we have not found situations where such bounds perform
better than the corresponding ones from the above corollaries. Since
we found it quite cumbersome to obtain the complete statements and we
believe that such results may one day serve the community, we relegate
their statement to Appendix \ref{sec:some-more-ineq-1}.

\begin{exm}[Standard normal target]
  \label{sec:stand-norm-targ}
Let
  $X_{\infty} \sim \mathcal{N}(0, 1)$ and consider the notation of
  example \ref{ex:norm}. The classical Stein discrepancy between any random
  variable  $X_n$ and $X_{\infty}$ in this case is
  \begin{equation}
    \label{eq:distZ}
    \sup_{h \in \mathcal{H}} \left| \mathbb{E} \left[ g_h'(X_n) - X_n g_h(X_n)   \right] \right|
  \end{equation}
  with $g_h = \mathcal{L}_{\infty}h$ the unique bounded solution to
  the Stein equation
  $ g_h'(x) - x g(x) = h(x) - \mathbb{E}h(X_{\infty})$.  Applications
  of \eqref{eq:distZ} are extremely well documented. To illustrate the
  power of our approach, let $X_n$ be a continuous real random
  variable.  By Corollaries \ref{cor:kolbound} and \ref{cor:kolbound2}
  the following bounds hold.
 
\begin{itemize}
\item  \emph{Kolmogorov distance}

  Direct computations from \eqref{eq:corkol} yield 
\begin{align*}
 \mathrm{Kol}(X_n,X_\infty)& = \sup_z \left| \mathbb{E}\left[ (X_n +
    \rho_n(X_n)) \frac{\Phi(X_n\wedge z)\bar \Phi(X_n\vee z)}{\varphi(X_n)} \right]
                             -
                             \kappa_1^{\star}(z) \right|    \\
    & \le  
      \mathbb{E}\left[ |X_n+\rho_n(X_n)| \frac{\Phi(X_n)\bar \Phi(X_n)}{\varphi(X_n)}
    \right] + \sup_z \kappa_1^{\star}(z) \\
    & \le \frac{1}{2}\sqrt{\frac{\pi}{2}} \mathbb{E}\left[ |X_n+\rho_n(X_n)| 
    \right] + \sup_z \kappa_1^{\star}(z) 
    \end{align*}
and, from \eqref{eq:corkol2}, 
    \begin{align*}
 \mathrm{Kol}(X_n,X_\infty)   
&     = \sup_z \left| \mathbb{E}\Bigg[ (\tau_n(X_n) - 1)
\Bigg( \Phi(z) - \mathbb{I}[X_n \le z]+ X_n
 \frac{\Phi(X_n \wedge z)  \bar \Phi(X_n\vee z)}{\varphi(X_n)} \Bigg) \Bigg]   
    + \kappa_{\mathrm{Id}}(z) \right|   \\
   & \le  
   \mathbb{E}\left[ \left| \tau_n(X_n)-1 \right|
\left( 1 + |X_n|
    \frac{\Phi(X_n)\bar \Phi(X_n)}{\varphi(X_n)} \right)
    \right] + \sup_z |\kappa_{\mathrm{Id}}(z)|  \\
       & \le  
   2\mathbb{E}\left[ \left| \tau_n(X_n)-1 \right|
    \right] + \sup_z |\kappa_{\mathrm{Id}}(z)|
\end{align*}     

For instance, if $X_n\sim t_n$ is Student with $n$ degrees of liberty,
then $\kappa_1^{\star}(z) = \kappa _{\mathrm{Id}}(z) =0$ for all $z$,
$\rho_n=-(1 + n) x/(n + x^2)$ and $\tau_n(x)=(x^2+n)/(n-1)$ (see e.g.\
Table 3 in the supplementary material to \cite{ernst2019first}) we
obtain
\begin{align}
  \mathrm{Kol}(X_n, X_{\infty}) 
   & \le  
      \mathbb{E}\left[ |X_n| \left|\frac{X_n^{2}-1}{X_n^2+n}   \right| \frac{\Phi(X_n)\bar \Phi(X_n)}{\varphi(X_n)}
    \right] \label{eq:kolstgauss}\\
    & \le \frac{1}{2}\sqrt{\frac{\pi}{2}} \mathbb{E}\left[ |X_n|\left|\frac{X_n^{2}-1}{X_n^2+n}   \right|
      \right]  \le \frac{2/\sqrt{e} -1/2}{n-1} \approx
      \frac{0.7130}{n-1} \nonumber 
\end{align}
(we use
${\Phi(x)\big(1-\Phi(x)\big)}/{\phi(x)} \le
{\Phi(0)\big(1-\Phi(0)\big)}/{\phi(0)} = 1/2\sqrt{\pi/2} \approx
0.626$) and
\begin{align*}
  \mathrm{Kol}(X_n, X_{\infty}) 
   & \le  
      \mathbb{E}\left[  \frac{X_n^2+1}{n-1}  
\left( 1 + |X_n|
    \frac{\Phi(X_n)\bar \Phi(X_n)}{\varphi(X_n)} \right)
    \right]  \le 2 \mathbb{E}\left[ \frac{X_n^2+1}{n-1}
      \right]   = \frac{2}{n-2}.
\end{align*}
Both our bounds improve e.g.\ on \cite[Example 1,
p1614]{cacoullos1994variational} but do (of course) not improve on the
optimal bound of Pinelis \cite[Theorem 1.2]{pinelis2015exact} which is
of order $0.158/n$.

\item \emph{Total variation distance.}

  Our upper bounds \eqref{eq:tv1} and \eqref{eq:cortv2} on Total
  Variation distance are the same as those for the Kolmogorov distance
  reported above. We can compare these bounds directly with
  \cite[Lemma 9]{duembgen2019bounding} who obtain the elegant bound
  $\mathrm{TV}(X_n, X_{\infty}) \le 2/n$ in this case.  Our rough
  upper bounds are not competitive. We could also use known results on
  Mill's ratio (such as e.g.\ in \cite[Theorem 2.3]{baricz2008mills}'s
  bound
  $ \frac{\Phi(x)\big(1-\Phi(x)\big)}{\phi(x)} \le
  \frac{4}{\sqrt{x^2+8}+3|x|}$) to hope for more explicit
  results. This does not, however, seem to lead easily to more
  explicit bounds and we'd rather not focus on this issue at the time
  being. Hence we content ourselves with numerical evaluations of
  \eqref{eq:kolstgauss} which in this case show that our non uniform
  bound is a (slight) improvement on \cite[Lemma
  9]{duembgen2019bounding}, see Figure \ref{fig:sfig1}. It would of
  course be interesting to obtain a formal proof of this result.

%obviously,

\item \emph{Wasserstein distance.}

  Direct computations from \eqref{eq:corwass1} yield 
 \begin{align*}
   \mathrm{Wass}(X_n, X_{\infty}) &  
                                    = \sup_{h \in \mathrm{Lip}(1)}  
                                    \left| \mathbb{E}\left[ \left({\rho_n(X_n) +} X_n \right)
                                    h'(X_\infty)
                                    \tilde{K}_\varphi(X_\infty,
                                    X_n)\right] + \kappa_1^{\star}(h)\right|\\
                                  & \leq \mathbb{E}\left[  \left|
                                    {\rho_n(X_n)+}X_n \right| \right]
                                    + \sup_{h \in \mathrm{Lip}(1) }
                                    |\kappa_1^{\star}(h)|. 
    \end{align*}  
In the particular case of Student $t$ vs standard normal, we
    obtain 
    \begin{align*}
\mathrm{Wass}(X_n, X_{\infty})  \le
      \E\left[\left|X_n\frac{1-X_n^2}{n+X_n^2}\right| \right]  \le
      \frac{3}{\sqrt{2\pi}} \frac{1}{\sqrt{n-1}}. 
    \end{align*}
    The bounds obtained from \eqref{eq:corwas2} are of the same order
    and not reported here. 
\end{itemize}
 \begin{figure}
   \centering
   \begin{subfigure}{.4\textwidth}
  \centering
\includegraphics[width=1\linewidth]{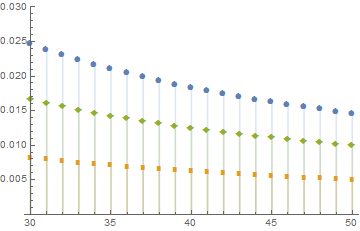}
  \caption{}
\label{fig:sfig1}
\end{subfigure}%
\quad 
\begin{subfigure}{.4\textwidth}
  \centering
  \includegraphics[width=1\linewidth]{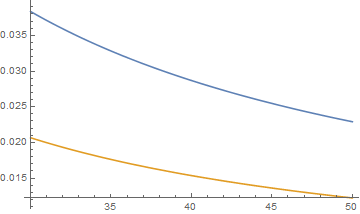}
  \caption{}
  \label{fig:sfig2}
\end{subfigure}
\caption{ \small{ Figure \ref{fig:sfig1} reports bounds on the
    total variation distance between $t_n$ and $\mathcal{N}(0, 1)$ for
    $n \in [30, 50]$: $2/n$ (green curve), our bound
    $({2/\sqrt{e} -1/2})/({n-1})$ (blue curve) and numerical
    evaluation of bound \eqref{eq:kolstgauss} (orange curve). Figure
    \ref{fig:sfig2} provides our upper bound on the Wasserstein
    distance (blue curve) as well as the exact value of the
    Wasserstein distance (computed with the formula
    $\mathrm{Wass}(X_n, X_{\infty}) = \int_{-\infty}^{\infty} \left|
      P_n(z) - P_{\infty}(z) \right| \mathrm{d}z$) for the same model
    and range of $n$.
 \label{fig:figbound}  }
}
\end{figure} 
\end{exm}

\begin{exm}[{Beta vs gamma}]
 \label{sec:beta-vs-gamma}

 Let $X_B \sim \mathrm{Beta}(\alpha, \beta)$ with density $p_B(x) =
 x^{\alpha-1} (1-x)^{\beta-1}/B(\alpha, \beta) \mathbb{I}_{[0, 1]}(x)$
 and cdf $P_{B}$; also let  $X_G \sim
 \Gamma(r, s)$ with density $p_{G}(x)  = x^{r-1} s^{r}
 e^{-s x}/\Gamma(r) \mathbb{I}_{[0, \infty)}(x)$ and cdf $P_{G}$. Simple computations
 yield (see also Table 3 in \cite{ernst2019first}) the scores and
 Stein kernels: 
 \begin{align*}
   &\rho_{B}(x) = \frac{1-\alpha+ x(\alpha +
   \beta-2)}{x(x-1)} \mbox{ and } \tau_{B}(x) =
   \frac{x(1-x)}{\alpha+\beta}\\
   & \rho_{G}(x) = \frac{r-1}{x} - s\mbox{ and }
     \tau_{G} = \frac{x}{s}.
 \end{align*}
 In order to facilitate 
 comparison with \cite{duembgen2019bounding}, we consider the
 same parameter settings as in that paper, namely $r = \alpha$ and
 $\beta>1$. Then
 \begin{align*}
   \rho_{B}(x)-\rho_{G}(x)=
   s + \frac{\beta-1}{x-1} \mbox{ and }
   \tau_{B}(x) - \tau_{G}(x) = x \left( \frac{1-x}{\alpha+\beta} -\frac{1}{s} \right).
 \end{align*}
We apply Corollary \ref{eq:EhEh1} to obtain 
\begin{align}\label{eq:tvbg}
 \mathrm{TV}(X_B, X_G)  
&  \leq \mathbb{E}\left[ \left|  s + \frac{\beta-1}{X_B-1}  \right|
                           \frac{P_G(X_B)\bar P_G(X_B)}{p_G(X_B)} 
                           \right] 
\end{align}
(here we use $\Gamma(\alpha, s)$ as target, i.e.\  $X_B = X_n$ and
   $X_G = X_{\infty}$; 
   $\kappa_1^{\star}(\mathbb{I}_{A_n^{\infty}}) = 0$) and
   \begin{align}\label{eq:tvgb2}
 \mathrm{TV}(X_G, X_B)  
&  \leq \mathbb{E}\left[ \left|  s + \frac{\beta-1}{X_G-1}  \right|
                           \frac{P_B(X_G)\bar P_B(X_G)}{p_B(X_G)}
                           \mathbb{I}[X_G \in [0, 1]]
                           \right] 
   \end{align}
   (here we use $\mathrm{Beta}(\alpha, \beta)$ as target, i.e.\
   $X_B= X_{\infty}$ and $X_G = X_n$;
   $\kappa_1^{\star}(\mathbb{I}_{A_n^{\infty}}) = 0$).  Numerical
   evaluations show that our bounds seem to outperform those
   \cite{duembgen2019bounding} (see Figure \ref{fig:betagamma}). More
   effort needs to be put in the study of the behavior of the ratio
   $P_{\infty}(x)\bar P_{\infty}(x) / p_{\infty}(x)$.  We do not
   report the corresponding bounds on the total variation distance
   that can be obtained from Corollary \ref{eq:EhEh2}; we do not
   either compute the bounds on Kolmogorov or Wasserstein distance.
   
 \begin{figure}[h]
\centering\begin{subfigure}{.4\textwidth}
  \centering
\includegraphics[width=1\linewidth]{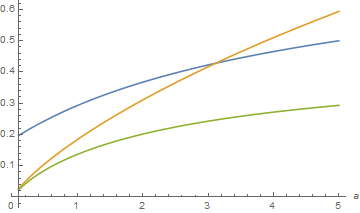}
\caption{}
 \label{fig:betagammaa}
\end{subfigure}%
\quad 
\begin{subfigure}{.4\textwidth}
  \centering
  \includegraphics[width=1\linewidth]{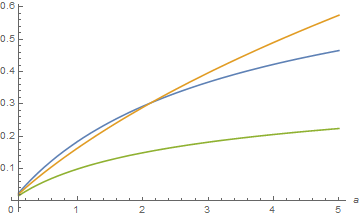}
\caption{}
 \label{fig:betagammab}
\end{subfigure}
\caption{ \small{Bounds on
    $\mathrm{TV}(X_B,X_G)$ obtained through  \eqref{eq:tvbg}
    (orange curve),   \eqref{eq:tvgb2}
    (green curve) and   \cite{duembgen2019bounding}
    (blue curve), with $X_B \sim \mathrm{Beta}(a, 3)$ vs $X_G \sim
    \Gamma(a, a+3)$ (Figure \ref{fig:betagammaa})  and $X_B \sim \mathrm{Beta}(a, 3)$ vs $X_G \sim
    \Gamma(a, a+2)$  (Figure \ref{fig:betagammab}).
 \label{fig:betagamma}}
}
\end{figure} 
 
\end{exm}

\begin{exm}[{Poisson target}]
\label{sec:poisson-target}

Let
  $X_{\infty} \sim \mathrm{Pois}(\lambda)$ and consider the notation of
  example \ref{ex:pois}. The classical Stein discrepancy between any random
  variable  $X_n$ and $X_{\infty}$ in this case is
  \begin{equation}
    \label{eq:12RR}
    \sup_{h \in \mathcal{H}} \left| \mathbb{E} \left[ \lambda g_h(X_n+1) - X_n g_h(X_n)   \right] \right|
  \end{equation}
  with $g_h(x) = \mathcal{L}_{\infty}^-h(x-1)$ the unique bounded
  solution to the Stein equation
  $\lambda g_h(x+1) - x g(x) = h(x) - \mathbb{E}h(X_{\infty})$.
  Applications of \eqref{eq:12RR} are extremely well documented. To
  illustrate the power of our approach, let $X_n$ be a discrete real
  random variable with values in $\mathbb{N}$.  By Corollaries 
  \ref{cor:prec} and  \ref{cor:prec2}, we get that $
  \mathrm{TV}(X_n, X_{\infty}) $ is bounded from above by  the
  following four quantities:  
  \begin{align*}
B_1(\lambda, X_n)  &  = \mathbb{E}\left[ \left|{\frac{\lambda}{X_n+1} - 1 -
        \rho_n^{+}(X_n)} \right| \frac{P_\infty(X_n)\bar
        P_\infty(X_n)}{p_\infty(X_n+1)} \right] 
        + \kappa_1^{\star+}(\mathbb{I}_{A_{n}^{\infty}}) \\
  B_2(\lambda, X_n)     &   = \mathbb{E}\left[ \left|{1-\frac{X_n}{\lambda}  -
        \rho_n^{-}(X_n)} \right| \frac{P_\infty(X_n-1)\bar
        P_\infty(X_n-1)}{p_\infty(X_n-1)}\mathbb{I}[X_n>0] \right] 
        + \kappa_1^{\star-}(\mathbb{I}_{A_{n}^{\infty}}) \\
B_3(\lambda, X_n)    & =  \mathbb{E}\left[ \left| \frac{\tau_n^{+}(X_n)}{X_n}-1
    \right|
\left( 1 + \frac{|X_n
                              -\lambda|}{X_n+1}
    \frac{P_\infty(X_n)\bar P_\infty(X_n)}{p_\infty(X_n+1)} \right)
      \right] +
      \kappa_{\mathrm{Id}}^{+}(\mathbb{I}_{A_{n}^{\infty}})\\
      B_4(\lambda, X_n)  & =  \mathbb{E}\left[ \left| \frac{\tau_n^{-}(X_n)}{\lambda}-1
    \right|
\left( 1 + \frac{|X_n
                              -\lambda|}{\lambda}
    \frac{P_\infty(X_n-1)\bar P_\infty(X_n-1)}{p_\infty(X_n-1)} \right)\mathbb{I}[X_n>0]
    \right] + \kappa_{\mathrm{Id}}^{-}(\mathbb{I}_{A_{n}^{\infty}})             
  \end{align*}

  We illustrate the bounds on some easy examples.
  \begin{exm}[Poisson vs Poisson]
    If $X_n \sim \mathrm{Pois}(\lambda_n)$ then
    $\kappa_1^{\star+}(\mathbb{I}_{A_{n}^{\infty}}) =
    \kappa_1^{\star-}(\mathbb{I}_{A_{n}^{\infty}}) = 0$  so that
    \begin{align*}     
      B_1(\lambda, \lambda_n) & = |\lambda - \lambda_n| \mathbb{E} \left[
            \frac{1}{X_n+1}   \frac{P_\infty(X_n)\bar
        P_\infty(X_n)}{p_\infty(X_n+1)}\right] \le |\lambda -
            \lambda_n| \frac{{\lambda}}{\lambda_n} \\
      B_2(\lambda, \lambda_n) & = \left| \frac{1}{\lambda} - \frac{1}{\lambda_n}
            \right|
            \mathbb{E}\left[ X_n\frac{P_\infty(X_n-1)\bar
        P_\infty(X_n-1)}{p_\infty(X_n-1)}\mathbb{I}[X_n>0] \right]
               \le |\lambda - \lambda_n|.  
    \end{align*}
    Similar arguments apply for $B_3$ and $B_4$ yielding similar
    results that are not reported here (although it is interesting to
    note that the first term in $B_3$ cancels out, and the only non
    zero term arises through non equality of the means).  \end{exm}

\begin{exm}[Poisson vs binomial]
  If $X_n \sim \mathrm{Bin}(n, \theta)$ and $\lambda = n \theta$ then
  $\kappa_1^{\star+}(\mathbb{I}_{A_{n}^{\infty}}) = 0 $ and 
  $\kappa_1^{\star-}(\mathbb{I}_{A_{n}^{\infty}}) \le \sqrt{2\pi}
  n^{1/2} e^{-n(1-\theta)}$ which is negligible for all values of
  $\theta \in (0, 1)$. Moreover 
  \begin{align*}
    \rho_n^+(x) = \frac{\theta}{1-\theta} \frac{n-x}{x+1}-1 \mbox{ and }
    \rho_n^-(x) =  1- \frac{1-\theta}{\theta} \frac{x}{n-x+1}
  \end{align*}
  so that
  \begin{align*}
    B_1(\lambda, n, \theta) & = \mathbb{E} \left[  \frac{\theta}{1-\theta} \frac{|X_n-n
    \theta|}{X_n+1}  \frac{P_\infty(X_n)\bar
    P_\infty(X_n)}{p_\infty(X_n+1)} \ \right]  \\
 B_2(\lambda, n, \theta) & = \mathbb{E} \left[ X_n
       \frac{|X_n-1-n\theta|}{n\theta(n-X_n+1)}
       \frac{P_\infty(X_n-1)\bar 
    P_\infty(X_n-1)}{p_\infty(X_n-1)} \mathbb{I}[X_n>0] \right] + \kappa_1^{\star-}(\mathbb{I}_{A_{n}^{\infty}}) \
  \end{align*}
  We can also exchange the roles of $p_n$ and $p_{\infty}$ and compute
  the same bounds with respect to the Poisson target. Numerical
  evaluations are reported in Figure \ref{fig:binpoi}.

\begin{figure}
  \centering
\begin{subfigure}{.4\textwidth}
  \centering
  \includegraphics[width=1\textwidth]{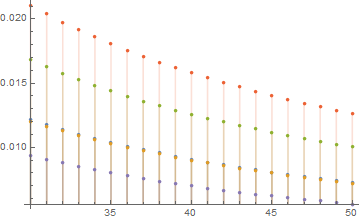}
  \caption{}
  \label{fig:poisBinlambda1}
\end{subfigure}%
\quad
\begin{subfigure}{.4\textwidth}
  \centering
    \includegraphics[width=1\textwidth]{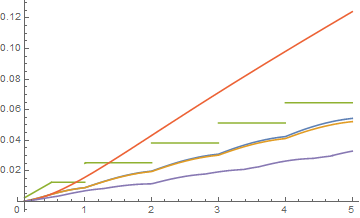}
  \caption{}
  \label{fig:DpoisBinn40}
\end{subfigure}%
\caption{\small Exact value of
  $\mathrm{TV}(\mathrm{Bin}(n, \lambda/n), \mathrm{Pois}(\lambda))$
  (purple curve), bound $B_1(\lambda, n, \lambda/n)$ (blue curve), the
  same bound when the roles of $X_n$ and $X_{\infty}$ are reversed
  (orange curve), the bound
  $\min(\lambda/n, 1-\sqrt{1-\lceil{\lambda}\rceil/n})$ from
  \cite{duembgen2019bounding} (green curve) and Chen's classical bound
  $\lambda(1- e^{-\lambda})/n$ from \cite{chen1975poisson} (red
  curve). Left plot for $\lambda = 1$ and $n \in [30, 50]$; right plot
  for $n=40$ and $\lambda \in (0, 5)$}
  \label{fig:binpoi}
\end{figure}

\end{exm}
\end{exm}
More examples and applications are detailed in the supplementary
material. 

\section*{Acknowledgements}

The research of YS was partially supported by the Fonds de la
Recherche Scientifique -- FNRS under Grant no F.4539.16.  ME
acknowledges partial funding via a Welcome Grant of the Universit\'e
de Li\`ege. YS thanks C\'eline Esser for many fruitful discussions. ME
and YS thank Robert Gaunt for sharing with us a nice problem (detailed
in the Appendix, Section \ref{sec:bounds-ipms-1}) which our technology
applied to, see \cite{gaunt2019lap}.

Most of all, Marie and Yvik thank Gesine Reinert whose ideas are at
the basis of references \cite{ernst2019first,ernst2019infinite}
without which this paper could not have been written.

\bibliographystyle{abbrv}

\appendix

\section{Some more proofs}
\label{sec:proofs}

 \begin{proof}[Proof of Lemma \ref{lem:Rp}]  Introduce
   $ \Phi^{\ell}_p(u, x, v) = {\chi^{\ell}(u, x)\chi^{-\ell}(x,
     v)}/{p(x)}$ for all $x \in \mathcal{S}(p)$ and 0 elsewhere, which
   allows to perform ``probabilistic integration'' as follows: if
   $f \in \mathrm{dom}(\Delta^{-\ell})$ is such that
   $(\Delta^{-\ell}f)$ is integrable on
   $[x_1, x_2] \cap \mathcal{S}(p)$ then
 \begin{equation}
 \label{eq:probaint}
 f(x_2)-f(x_1) =  \mathbb{E}\left[   \Phi^{\ell}_p(x_1, X, x_2)
   \Delta^{-\ell}f(X)\right] % =
 % \begin{cases}
 %   \int_{x_1}^{x_2}f'(u) \mathrm{d}u  & (\ell = 0) \\
 %   \sum_{j=x_1}^{x_2-1} \Delta^+f(j)  & (\ell = -1) \\
 %   \sum_{j=x_1+1}^{x_2} \Delta^-f(j)  & (\ell = 1) 
 % \end{cases}
 \end{equation}
 for all $x_1 < x_2 \in \mathcal{S}(p)$. We can use this function to
 obtain 
     \begin{align*}
     \bar{h}(x) & = 
     \E \left[(h(x)-h(X))(\chi^\ell(X,x)+\chi^{-\ell}(x,X)) \right] \\
     &= \E\left[ \Delta^{-\ell} h(X_2) \E\left[ \Phi_p^\ell(X,X_2,x)-\Phi_p^\ell (x,X_2,X)  | X_2\right] \right]
 %    \mathbb{E} \left[ \Delta^{-\ell} h(X) \left(                 
 %                 \sum_{y=a}^{x} \Phi_p^{\ell}(y, X, x) p(y)
 %                 -\sum_{y=x}^b \Phi_p^{\ell}(x, X, y) p(y)  \right) \right]
     \end{align*}
     (we use the fact that $\chi^\ell(x,y)+\chi^{-\ell}(y,x) = 1+\mathbb{I}[\ell=0]\mathbb{I}[x=y]$) and it only
     remains to reorganize the integrand to obtain the claim. To this
     end we note how, by definition,
 %    \begin{align*}       
 %            &      \sum_{y=a}^{x} \Phi_p^{\ell}(y, X, x) p(y)
 %      -\sum_{y=x}^b \Phi_p^{\ell}(x, X, y) p(y) \\  & =
 %                 \frac{\chi^{-\ell}(X, x)}{p(X)}\sum_{y=a}^{x}
 %                                                      \chi^{\ell}(y,
 %                                                      X) p(y)
 %-  \frac{\chi^{\ell}(x, X)}{p(X)}\sum_{y=x}^{b}
 %                                                      \chi^{-\ell}(X,
 %                                                        y) p(y)  \\
 %            & =    \chi^{-\ell}(X, x) \frac{P(X - a_{\ell})}{p(X)} 
 %-  \chi^{\ell}(x, X)\frac{\bar P(x+b_{\ell}) }{p(X)}
 %    \end{align*}
     \begin{align*}       
    \E\left[ \Phi_p^\ell(X,y,x)-\Phi_p^\ell (x,y,X)  \right] 
 & =\frac{\chi^{-\ell}(y,x)}{p(y)} \E[\chi^\ell(X,y)] - \frac{\chi^{\ell}(x,y)}{p(y)}\E[\chi^{-\ell}(y,X)]   \\
  & =   \chi^{-\ell}(y, x) \frac{P(y - a_{\ell})}{p(y)} 
 -  \chi^{\ell}(x, y)\frac{\bar P(y+b_{\ell}) }{p(y)}
     \end{align*}
     where the first identity is immediate by definition of $\Phi_p^\ell$ and the 
      last identity follows from the definition of the
     generalized indicator $\chi^{\ell}$. 
%     The identities \eqref{eq:rep_caninv2} are direct from Definition \ref{def:can_inverse}.
  \end{proof}   
 
\begin{proof}[Proof of Lemma \ref{prop:deltagrep}]
    The expressions \eqref{eq:sol1} and \eqref{eq:sol2} of the solution $g$
    are direct from the definition of $\mathcal{L}_p^\ell$ and its representation \eqref{eq:rep_caninv}. 
    The first expression \eqref{eq:dersol1} of the derivative is direct from the expression \eqref{eq:derivsolstek2}. For the second claim, we shall first prove the following results:
  \begin{align}
  \label{eq:deltag1}
  \Delta^{-\ell}g(x) & = \frac{\mathbb{E} \left[\tilde{K}_p^{\ell}(X_1,
  x+\ell) R_p^{\ell}(x, X_2) \bigg(  \Delta^{-\ell}\eta(X_2)
  \Delta^{-\ell}h(X_1) - \Delta^{-\ell} h(X_2)
  \Delta^{-\ell} \eta(X_1)  \bigg) \right]}{\big(-\mathcal{L}_p^\ell
\eta(x)\big)\big(-\mathcal{L}_p^\ell \eta(x+\ell)\big)}  \\
& =   \frac{\mathbb{E} \left[\bigg(\tilde{K}_p^{\ell}(X_1,
  x+\ell)  R_p^{\ell}(x, X_2)- R_p^{\ell}(x, X_1)\tilde{K}_p^{\ell}(X_2,
  x+\ell) \bigg) \Delta^{-\ell}h(X_1) \Delta^{-\ell}\eta(X_2) \right]}{\big(-\mathcal{L}_p^\ell
                         \eta(x)\big)\big(-\mathcal{L}_p^\ell
                         \eta(x+\ell)\big)}  \label{eq:deltag2}
\end{align}
  We first prove \eqref{eq:deltag1}. Starting from \eqref{eq:derivsolstek1} and   applying
  repeatedly \eqref{eq:rep_caninv} then  \eqref{eq:rep_h} (once to $h$ and once to $\eta$) we obtain
\begin{align*}
   \Delta^{-\ell} g(x)&  =\frac{\mathbb{E} \left[\tilde{K}_p^{\ell}(X_1,
  x+\ell) \bigg(  \bar\eta(x) \Delta^{-\ell}h(X_1)\big) - \bar h(x)
  \Delta^{-\ell} \eta(X_1)  \bigg) \right]}{\big(-\mathcal{L}_p^\ell
                        \eta(x)\big)\big(-\mathcal{L}_p^\ell \eta(x+\ell)\big)}  \\
  & = \frac{\mathbb{E} \left[\tilde{K}_p^{\ell}(X_1,
  x+\ell) R_p^{\ell}(x, X_2) \bigg(  \Delta^{-\ell}\eta(X_2)
  \Delta^{-\ell}h(X_1)\big) - \Delta^{-\ell} h(X_2)
  \Delta^{-\ell} \eta(X_1)  \bigg) \right]}{\big(-\mathcal{L}_p^\ell
\eta(x)\big)\big(-\mathcal{L}_p^\ell \eta(x+\ell)\big)}.
\end{align*}
We now prove \eqref{eq:deltag2}. By similar arguments as above, this
follows from
\begin{align*}
   \Delta^{-\ell} g(x) & 
%   =\frac{\mathbb{E} \left[\tilde{K}_p^{\ell}(X_1,
%  x+\ell)   \bar\eta(x) \Delta^{-\ell}h(X_1) \right] - \bar h(x)
%\big(-\mathcal{L}_p^\ell \eta(x+\ell)\big)}{\big(-\mathcal{L}_p^\ell
%                         \eta(x+\ell)\big)\big(-\mathcal{L}_p^\ell
%                         \eta(x+\ell)\big)} \\
=\frac{\mathbb{E} \left[\tilde{K}_p^{\ell}(X_1,
  x+\ell)   \bar\eta(x) \Delta^{-\ell}h(X_1) \right] - 
\big(-\mathcal{L}_p^\ell \eta(x+\ell)\big) \E\left[R_p^\ell(x,X_1)\Delta^{-\ell}h(X_1)\right]
}{\big(-\mathcal{L}_p^\ell
                         \eta(x+\ell)\big)\big(-\mathcal{L}_p^\ell
                         \eta(x+\ell)\big)} \\
  & = \frac{\mathbb{E} \left[\bigg(\tilde{K}_p^{\ell}(X_1,
  x+\ell)   \bar\eta(x) - R_p^{\ell}(x, X_1)\big(-\mathcal{L}_p^\ell
                         \eta(x+\ell)\big) \bigg) \Delta^{-\ell}h(X_1) \right]}{\big(-\mathcal{L}_p^\ell
                         \eta(x)\big)\big(-\mathcal{L}_p^\ell
    \eta(x+\ell)\big)}\\
& = \frac{\mathbb{E} \left[\bigg(\tilde{K}_p^{\ell}(X_1,
  x+\ell)  R_p^{\ell}(x, X_2)- R_p^{\ell}(x, X_1)\tilde{K}_p^{\ell}(X_2,
  x+\ell) \bigg) \Delta^{-\ell}h(X_1)\Delta^{-\ell}\eta(X_2) \right]}{\big(-\mathcal{L}_p^\ell
                         \eta(x)\big)\big(-\mathcal{L}_p^\ell
                         \eta(x+\ell)\big)}  .
\end{align*}

To conclude, we decompose the above expectation into four
parts with: $X_i<x+a_\ell$ and/or $X_i\geq x+a_\ell$, for $i=1,2$
(i.e., using either $\chi^{-\ell}(X_i,x)$ or
$\chi^{\ell}(x,X_i)$). Therefore, by considering separately
$\ell\in\{0,-1,1\}$, we can easily verify that
\begin{align*}
\tilde{K}_p^\ell(y,x+\ell) = \begin{cases} 
\dfrac{P(y-a_\ell)\bar{P}(x+a_\ell)}{p(y)p(x+\ell)} \mbox{ if } y<x+a_\ell \\[10pt]
\dfrac{P(x-b_\ell)\bar{P}(y+b_\ell)}{p(y)p(x+\ell)} \mbox{ if } y\geq x+a_\ell\end{cases} 
\quad \mbox{ and }\quad
R_p^\ell(x,y) = \begin{cases} 
\dfrac{P(y-a_\ell)}{p(y)} \mbox{ if } y<x+a_\ell \\[10pt]
\dfrac{-\bar{P}(y+b_\ell)}{p(y)} \mbox{ if } y\geq x+a_\ell\end{cases} 
\end{align*}
Basic manipulations then give
\begin{align*}
& \Delta^{-\ell} g(x) \big(-\mathcal{L}_p^\ell\eta(x)\big)\big(-\mathcal{L}_p^\ell\eta(x+\ell)\big)  \\
&= 
\frac{\bar{P}(x+a_\ell)+P(x-b_\ell)}{p(x+\ell)} \Bigg( 
\E\left[\Delta^{-\ell} h(X_1) \frac{\bar{P}(X_1+b_\ell)}{p(X_1)}\chi^{\ell}(x,X_1)\right] \E\left[\Delta^{-\ell} \eta(X_2) \frac{P(X_2-a_\ell)}{p(X_2)}\chi^{-\ell}(X_2,x)\right] \\
&\qquad\qquad\qquad \qquad\qquad- \E\left[\Delta^{-\ell} h(X_1) \frac{P(X_1-a_\ell)}{p(X_1)}\chi^{-\ell}(X_1,x)\right] \E\left[\Delta^{-\ell} \eta(X_2) \frac{\bar{P}(X_2+b_\ell)}{p(X_2)}\chi^{\ell}(x,X_2)\right] \Bigg)
\end{align*}
which leads to the claim as $\bar{P}(x+a_\ell)+P(x-b_\ell)=1$ and $\ell=a_\ell-b_\ell$.
\end{proof}

\begin{proof}[Proof of Lemma \ref{lem:pointmass2}]
The condition implies that $g^-$ is non decreasing and non negative over $\mathcal{S}(p) \cap (-\infty, \xi]$ and non decreasing and non positive over $\mathcal{S}(p) \cap (\xi, \infty)$. Therefore, 
the absolute value of the solution for point mass equation \eqref{eq:eq-indicator} reaches his supremum at $\xi$ or $\xi+1$, which gives the bound \eqref{eq:supg-indicator}. Moreover, the supremum of the difference is observed between $\xi$ and $\xi+1$. Using the explicit expression \eqref{eq:gpointmass} and the relation $\tau_p^\ell(x+\ell)p(x+\ell)=\tau_p^{-\ell}(x)p(x)$, we have
\begin{align*}
\sup_x |\Delta g(x)| &= g^-(\xi) - g^-(\xi+1) =
\frac{P(\xi-1)}{\tau_p^+(\xi)} + \frac{(1-P(\xi))p(\xi)}{\tau_p^+(\xi+1)p(\xi+1)} \\
&= \frac{P(\xi-1)}{\tau_p^+(\xi)} + \frac{1-P(\xi)}{\tau_p^-(\xi)}.
\end{align*}
Furthermore, as $x-\E[X] = \tau_p^+(x)-\tau_p^-(x)$, we have $\tau^-_p(\xi) \geq \tau^+_p(\xi)$ if $\xi \leq \E[X]$ (resp. $\tau^-_p(\xi) \leq \tau^+_p(\xi)$ if $\xi \geq \E[X]$). Therefore, the supremum is bounded by $\frac{P(\xi-1)+1-P(\xi)}{\tau_p^+(\xi)}=\frac{1-p(\xi)}{\tau_p^+(\xi)}$ if $\xi\leq \E[X]$ and otherwise by $\frac{1-p(\xi)}{\tau_p^-(\xi)}$. 

By remark \ref{rmk:borel}, the solution $g_A^\ell(x)$ is explicit and defined by $g_\xi$ for $\xi\in A$. The sign of $g_\xi$ changes according to the relative position of $\xi$ and $x$. Then, combined with the hypotheses, the maximal value of $|g_A^-(x)|$ is either observed at $x=\min_{\xi\in A}\{\xi\}=:\xi_{1}$ or $x=\max_{\xi\in A}\{\xi\}+1=:\xi_{2}+1$. Then, 
\begin{align*}
\sup_x |g^-_A(x)| 
&= \max\left\{ \frac{P(\xi_1-1)}{p(\xi_1)\tau_p^+(\xi_1)}\sum_{j\in A}p(j), \frac{1-P(\xi_2)}{p(\xi_2)\tau_p^-(\xi_2)}\sum_{j\in A}p(j) \right\} \\
&\leq \left( \sum_{j\in A}p(j) \right) \sup_{\xi\in A}\left\{ \frac{1}{\tau_p^+(\xi)p(\xi)}, \frac{1}{\tau_p^-(\xi)p(\xi)}\right\}.
\end{align*} 
Finally, due to the monotonicity of each $g_\xi(x)$ function, the maximal difference $|\Delta g_A(x)|$ is bounded by the supremum of $|\Delta g_\xi(x)|$ for $\xi \in A$, which is enough to conclude. 
\end{proof}

\medskip
 \begin{proof}[Proof of Theorem \ref{prop:EhEh}]
   First take $c_1(x)= c_2(x) = 1 $ in \eqref{eq:ehehv2}.
   Without any further assumptions on $h$, the solution $g_h^*$ of \eqref{eq:stek2} with $c(x)=1$ can be represented as
    $$ g_h^*(x) = \frac{\mathcal{L}_\infty^\ell h(x+\ell)}{c_1(x+\ell)}
     = \mathcal{L}_\infty^\ell h(x+\ell)
%    = \frac{1}{p_\infty(x+\ell)} \mathbb{E}[(h(X_1) - h(X_2))
%    \mathbb{I}[X_1+a_\ell \le x+\ell \le X_2-b_\ell]] 
%    = \frac{1}{p_\infty(x+\ell)} \mathbb{E}[(h(X_1) - h(X_2))
%    \mathbb{I}[X_1+b_\ell \le x \le X_2-a_\ell]] 
$$ 
%    with $X_1, X_2$ independent copies of $X_{\infty}$. 
Hence, we obtain \eqref{eq:EhEh1}.
    
    \medskip
    Next take $\eta_1 = \eta_2 = \mathrm{Id}$ in \eqref{eq:ehehv1}. Then,
    $-\mathcal{L}_{\infty}^\ell\eta_1(x) = \tau_\infty^\ell(x)$ and
    $-\mathcal{L}_n^\ell\eta_2(x) = \tau_n^\ell(x)$, the Stein kernels of $p_\infty$ and $p_n$. 
    Without any further assumptions on $h$, the solution $g_h(x)$ of \eqref{eq:stek1} with $\eta=\mathrm{Id}$ can be represented as
    \begin{align*}
    g_h(x) = \frac{-\mathcal{L}_\infty^\ell h(x+\ell)}{\tau_\infty(x+\ell)}
%    = \frac{1}{\tau_\infty(x+\ell) p_\infty(x+\ell)} \mathbb{E}[(h(X_2) - h(X_1))
%    \mathbb{I}[X_1+b_\ell \le x \le X_2-a_\ell]]  
    \end{align*}    
%    with $X_1, X_2$ independent copies of $X_{\infty}$
%    and $\Delta^{-\ell}g_h$ can be represented as
%    \begin{align*}
%     \Delta^{-\ell}g_h(x) 
%   & = \frac{x-\E[X_\infty]}{\tau^\ell_\infty(x)} g_h(x) - \frac{h(x) - \mathbb{E}h(X_{\infty})}{\tau_\infty^\ell(x)} \\
%   & = \frac{x-\E[X_\infty]}{\tau^\ell_\infty(x)\tau^\ell_\infty(x+\ell)p_\infty(x+\ell)} \mathbb{E}[(h(X_2) - h(X_1))
%    \mathbb{I}[X_1+b_\ell \le x \le X_2-a_\ell]]  - \frac{h(x) - \mathbb{E}h(X_{\infty})}{\tau_\infty^\ell(x)}.
%    \end{align*}
%for $x\in \mathrm{int}(\mathcal{S}(p))$.
    Hence we get \eqref{eq:EhEh2}.    
    
%    The final statement uses an alternative representation of $g_h^*$ by considering \eqref{eq:rep_caninv}, i.e.
%$$ g_h^*(x) = \mathcal{L}_\infty^\ell h(x+\ell) = -\mathbb{E}[\Delta^{-\ell}h(X_\infty) \tilde{K}_\infty(X_\infty, x+\ell)].$$ Hence, working as above, from \eqref{eq:EhEh1} we obtain the equality \eqref{eq:EhEhb3}. 
 \end{proof}

 \section{Some more inequalities}
 \label{sec:some-more-ineq-1}
 
\begin{cor}[Identity \eqref{eq:EhEh2}, Stein kernels and $\ell = 0$]
  \label{cor:kolbound2}
  Under the same assumptions and with exactly the same notations as in
  Corollary \ref{cor:kolbound}, the following results hold true.
  \begin{enumerate}
  \item   The Kolmogorov distance
  between the random variables $X_n$ and $X_\infty$ is
\begin{align}
  &   \mathrm{Kol}(X_n,X_\infty) \nonumber\\
&     = \sup_z \Bigg| \mathbb{E}\Bigg[ \frac{\tau_n(X_n) -
                              \tau_\infty(X_n)}{\tau_{\infty}(X_n)}\mathbb{I}_{\mathcal{S}_{\infty}}(X_n) 
                                     \times  \label{eq:corkol2} \\
  & \qquad  \Bigg( 
{P_{\infty}(z) - \mathbb{I}[X_n \le z]}+
                              \frac{X_n
                              -\mathbb{E}[X_{\infty}]}{\tau_{\infty}(X_n)
                              } \frac{P_{\infty}(X_n
                              \wedge z)  \bar P_{\infty}(X_n\vee
    z)}{p_{\infty}(X_n)} \Bigg) \Bigg]   
    + \kappa_{\mathrm{Id}}(z) \Bigg|     \nonumber \\
  & \le  \mathbb{E}\left[ \left| \frac{\tau_n(X_n)}{\tau_{\infty}(X_n)}-1
    \right|
\left( 1 + \frac{|X_n
                              -\mathbb{E}[X_{\infty}]|}{\tau_{\infty}(X_n)}
    \frac{P_\infty(X_n)\bar P_\infty(X_n)}{p_\infty(X_n)} \right)
 \mathbb{I}_{\mathcal{S}_{\infty}}(X_n)   \right] + \sup_z |\kappa_{\mathrm{Id}}(z)|
\end{align}
where 
\begin{align*}
  &  \kappa_{\mathrm{Id}}(z)= (\mu_n - \mu_{\infty})  \mathbb{E} \left[ \frac{P_{\infty}(X_n
                              \wedge z)  \bar P_{\infty}(X_n\vee
    z)}{\tau_{\infty}(X_n)p_{\infty}(X_n)}  \right] \\
  &  \quad +  \lim_{x \nearrow
  b_n \wedge b_{\infty}}\frac{\tau_n(x)}{\tau_\infty(x)}  \frac{p_n(x)}{p_\infty(x)}P_\infty(x\wedge z)\bar P_\infty(x\vee
                            z)- \lim_{x \searrow a_n \vee a_{\infty}}\frac{\tau_n(x)}{\tau_\infty(x)} \frac{p_n(x)}{p_\infty(x)}P_\infty(x\wedge z)\bar P_\infty(x\vee
    z).    
\end{align*}

\item  The Total
  Variation distance between $X_n$ and $X_\infty$ is
  \begin{align}
  &    \mathrm{TV}(X_n, X_{\infty})  \nonumber\\
&     =
                                                   \kappa_{\mathrm{Id}}(\mathbb{I}_{A_{n}^{\infty}})     \label{eq:cortv2}
                                                   +\mathbb{E}\Bigg[ \frac{\tau_n(X_n) -
                              \tau_\infty(X_n)}{\tau_{\infty}(X_n)}\mathbb{I}_{\mathcal{S}_{\infty}}(X_n)
                                     \times  \\
  & \qquad   \Bigg( 
{P}_{\infty}(A_n^{\infty})  - \mathbb{I}_{ A_n^{\infty}}(X_n) + 
                              \frac{X_n
    -\mathbb{E}[X_{\infty}]}{\tau_{\infty}(X_n)}
   \frac{ {P}_\infty(A_{n}^{\infty}\cap (-\infty,
      X_n])-{P}_\infty (A_{n}^{\infty})
                                                P_{\infty}(X_n)}{p_\infty(X_n)}  \Bigg) \Bigg] \nonumber  \\
  & \le  \mathbb{E}\left[ \left| \frac{\tau_n(X_n)}{\tau_{\infty}(X_n)}-1
    \right|
\left( 1 + \frac{|X_n
                              -\mathbb{E}[X_{\infty}]|}{\tau_{\infty}(X_n)}
    \frac{P_\infty(X_n)\bar P_\infty(X_n)}{p_\infty(X_n)} \right)\mathbb{I}_{\mathcal{S}_{\infty}}(X_n)
    \right] + \kappa_{\mathrm{Id}}(\mathbb{I}_{A_{n}^{\infty}})
  \end{align}
  with
  \begin{align*}  \kappa_{\mathrm{Id}}(\mathbb{I}_{A_{n}^{\infty}})
    & = \lim_{x \nearrow b_n \wedge b_{\infty}}\frac{\tau_n(x)}{\tau_\infty(x)} \frac{p_n(x)}{p_\infty(x)}
      \left( {P}_\infty(A_{n}^{\infty}\cap (-\infty,
      x])-{P}_\infty (A_{n}^{\infty})
                                                P_{\infty}(x)\right)
                                                \nonumber \\
      & \quad  - \lim_{x \searrow a_n \vee b_n}\frac{\tau_n(x)}{\tau_\infty(x)}
        \frac{p_n(x)}{p_\infty(x)}
      \left( {P}_\infty(A_{n}^{\infty}\cap (-\infty,
        x])-{P}_\infty (A_{n}^{\infty})  P_{\infty}(x)\right)
    \\
   & \quad + (\mu_n-\mu_{\infty}) \mathbb{E} \left[
   \frac{ {P}_\infty(A_{n}^{\infty}\cap (-\infty,
      X_n])-{P}_\infty (A_{n}^{\infty})
                                                P_{\infty}(X_n)}{
     {\tau_{\infty}(X_n)} p_\infty(X_n)}   \right]
    \end{align*}

  \item The Wasserstein distance between $X_n$ and $X_\infty$ is
    \begin{align}
      &    \mathrm{Wass}(X_n, X_{\infty})   =   \sup_{h \in \mathrm{Lip}(1)}  
      \Bigg|  \kappa_{\mathrm{Id}}(h)         \label{eq:corwas2}                    \\
      &  
+  \mathbb{E}\left[
 \frac{\tau_n(X_n) -
                              \tau_\infty(X_n)}{\tau_{\infty}(X_n)}
                                       h'(X_\infty)  \left(
                                       R_{\infty}(X_n, X_{\infty}) +
                                       \frac{X_n -
                                       \mathbb{E}[X_{\infty}]}{\tau_{\infty}(X_n)} 
                                       \tilde{K}_\infty(X_n,
        X_\infty)\right)\mathbb{I}_{\mathcal{S}_{\infty}}(X_n)\right] \Bigg| \nonumber  \\
 % & \leq \mathbb{E}\left[  \left|\ma{\rho_n(X_n)-\rho_\infty(X_n)} \right|  \tilde{K}_\infty(X_\infty, X_n)\right] \\
 & \le 2 \mathbb{E}\left[  \left| \frac{\tau_n(X_n)}{\tau_{\infty}(X_n)}-1
    \right| |X_n - \mathbb{E}[X_{\infty}]|\mathbb{I}_{\mathcal{S}_{\infty}}(X_n)\right] 
   + \sup_{h \in \mathrm{Lip}(1)} \kappa_{\mathrm{Id}}(h)
    \end{align}
    where
    \begin{align*}
 \kappa_{\mathrm{Id}}(h)  &  = \lim_{x \searrow a_n \vee a_{\infty}}
      \frac{\tau_n(x)}{\tau_\infty(x)}\frac{p_n(x)}{p_{\infty}(x)}  \int_{a_{\infty}}^{b_{\infty}} h'(u)
       {P_{\infty}(x \wedge u) \bar P_{\infty}(x \vee u) } \mathrm{d}u
                       \nonumber 
      \\
      & \quad  -  \lim_{x \nearrow b_n \wedge b_{\infty}}\frac{\tau_n(x)}{\tau_\infty(x)} \frac{p_n(x)}{p_{\infty}(x)}  \int_{a_{\infty}}^{b_{\infty}}h'(u)
        {P_{\infty}(x \wedge u) \bar P_{\infty}(x \vee u) }
        \mathrm{d}u\\
      & \quad + (\mu_n - \mu_{\infty}) \mathbb{E}\left[\frac{h'(X_\infty)}{\tau_{\infty}(X_{n})}  \left(
                                       R_{\infty}(X_n, X_{\infty}) +
                                       \frac{X_n -
                                       \mathbb{E}[X_{\infty}]}{\tau_{\infty}(X_n)} 
                                       \tilde{K}_\infty(X_n,
        X_\infty)\right)\right]      
    \end{align*}
  \end{enumerate}
 
\end{cor}

\begin{cor}[Identity \eqref{eq:EhEh2}, Stein kernels, $\ell = \pm1$] \label{cor:prec2}
Under the same assumptions and with exactly the same notations as in
Corollary \ref{cor:prec}, the following results hold true.
    \begin{align*}
  &    \mathrm{TV}(X_n, X_{\infty})  \nonumber\\
&     =
                                                   \kappa_{\mathrm{Id}}^{\ell}(\mathbb{I}_{A_{n}^{\infty}})     %\label{eq:cortdiscv2}
                                                   +\mathbb{E}\Bigg[ \frac{\tau_n^{\ell}(X_n) -
                              \tau_\infty^{\ell}(X_n)}{\tau_{\infty}^{\ell}(X_n)} \mathbb{I}_{\mathcal{S}_{\infty}}(X_n+\ell)
                                     \times  \\
  & \qquad   \Bigg( 
{P}_{\infty}(A_n^{\infty})  - \mathbb{I}_{ A_n^{\infty}}(X_n) + 
                              \frac{X_n
    -\mathbb{E}[X_{\infty}]}{\tau_{\infty}^{\ell}(X_n+\ell)}
   \frac{ {P}_\infty(A_{n}^{\infty}\cap (-\infty,
      X_n-b_{\ell}])-{P}_\infty (A_{n}^{\infty})
                                                P_{\infty}(X_n-b_{\ell})}{p_\infty(X_n+\ell)}  \Bigg) \Bigg] \nonumber  \\
  & \le  \mathbb{E}\left[ \left| \frac{\tau_n^{\ell}(X_n)}{\tau_{\infty}^{\ell}(X_n)}-1
    \right|
\left( 1 + \frac{|X_n
                              -\mathbb{E}[X_{\infty}]|}{\tau_{\infty}^{\ell}(X_n+\ell)}
    \frac{P_\infty(X_n-b_{\ell})\bar P_\infty(X_n-b_{\ell})}{p_\infty(X_n+\ell)} \right)\mathbb{I}_{\mathcal{S}_{\infty}}(X_n+\ell)
    \right] + \kappa_{\mathrm{Id}}^{\ell}(\mathbb{I}_{A_{n}^{\infty}})
  \end{align*}
 with
\begin{align*}
  \kappa_{\mathrm{Id}}^{+}(\mathbb{I}_{A_{n}^{\infty}}) & = - \lim_{x
                                                   \searrow a_n \vee
                                                          a_{\infty}}\frac{\tau_n^{+}
                                                          (x)}{\tau_{\infty}^{+}(x)}\frac{p_n(x)}{p_\infty(x)}
                                                      \left( {P}_\infty(A_{n}^{\infty}\cap (-\infty,
                                                      x-1])-{P}_\infty
                                                      (A_{n}^{\infty})
                                                          P_{\infty}(x-1)\right)\\
  & \quad  + (\mu_\infty-\mu_n) \mathbb{E} \left[
   \frac{ {P}_\infty(A_{n}^{\infty}\cap (-\infty,
      X_n])-{P}_\infty (A_{n}^{\infty})
                                                P_{\infty}(X_n)}{
    {\tau_{\infty}^{+}(X_n+1)} p_\infty(X_n+1)}   \mathbb{I}_{\mathcal{S}_{\infty}}(X_n+1)\right]
 \end{align*}
and 
\begin{align*}    
    \kappa_{\mathrm{Id}}^{-}(\mathbb{I}_{A_{n}^{\infty}}) & = \lim_{x
                                                     \nearrow b_n
                                                     \wedge b_{\infty}}\frac{\tau_n^{-}(x)}{\tau_{\infty}^{-}(x)}\frac{p_n(x)}{p_\infty(x)}
                                                      \left( {P}_\infty(A_{n}^{\infty}\cap (-\infty,
                                                      x])-{P}_\infty (A_{n}^{\infty})
                                                      P_{\infty}(x)\right)
  \\
    & \quad  + (\mu_\infty-\mu_{n}) \mathbb{E} \left[
   \frac{ {P}_\infty(A_{n}^{\infty}\cap (-\infty,
      X_n])-{P}_\infty (A_{n}^{\infty})
                                                P_{\infty}(X_n)}{
    {\tau_{\infty}^{-}(X_n-1)} p_\infty(X_n-1)} \mathbb{I}_{\mathcal{S}_{\infty}}(X_n-1)\right]
                                                    \end{align*}

\end{cor}
 
\section{More examples of Stein equations,  solutions and bounds} 
 \label{sec:stein-equat-solut}

Before proceeding we recall that, for $h : \R \to \R$, we write
$\kappa_1:=\kappa_1(h) = \sup_{y\in\mathcal{S}(p)}h(y) -
\inf_{y\in\mathcal{S}(p)}h(y)$ and
$\kappa_2:=\kappa_2(h) =
\sup_{y\in\mathcal{S}(p)}|\Delta^{-\ell}h(y)|$. We also introduce the
notations (not present in the main text):
 \begin{equation*}
   M_p(x):= \frac{P(x)\bar{P}(x)}{p(x)} \mbox{ and } \tilde{M}_p^{\ell}(x) =
   \frac{\int_{{a+a_\ell}}^{x+\ell}P(u) \mathrm \mu(\mathrm{d}u)
     \int_{{x+\ell}}^{{b-b_\ell}}\bar{P}(u) \mu(\mathrm{d}u)}{p({x+\ell})}
 \end{equation*}
 with the convention that these functions are set to 0 outside the
 support of $p$.

In this section we apply the theory from  Section
\ref{sec:formalism} to various illustrative concrete examples. In all
cases we explicit the bounds from Section \ref{sec:examples}. 
\begin{exm}[{Beta distribution}] \label{ex:beta} 
%This is a continuous distribution, so $\ell= 0$. 
This distribution has pdf and support
  \begin{align*}
    p_{\alpha, \beta}(x) = \frac{x^{\alpha-1}(1-x)^{\beta-1}}{
  B(\alpha,\beta)}, \qquad \mathcal{S}(p_{\alpha, \beta}) = (0,1).
  \end{align*}
The cdf $P_{\alpha, \beta}$
  and survival $\bar{P}_{\alpha, \beta}$ do not bear an explicit
  expression.  Simple computations show that
  \begin{align*}
    \rho_{\alpha, \beta}(x)=\frac{\alpha-1-x(\alpha+\beta-2)}{x(1-x)}
    \mbox{ and } \tau_{\alpha, \beta}(x) =\frac{x(1-x)}{\alpha+\beta}.
  \end{align*}
Taking $c(x)=1$ in \eqref{eq:standop} %  leads to the
%       operator
%       $$\mathcal{A}_1g(x) =
%       \frac{\alpha-1-x(\alpha+\beta-2)}{x(1-x)}g(x)+ g'(x)$$ acting on
%       $\mathcal{F}(\mathcal{A}_1)$ the collection of test functions
%       such that
%       $\int_0^1 |( g(x) x^{\alpha-1}(1-x)^{\beta-1} /
%       B(\alpha,\beta))'| \mathrm{d}x < \infty$ and
% %      $\lim_{x \to 1} g(x) x^{\alpha-1}(1-x)^{\beta-1} /
% %  B(\alpha,\beta) = \lim_{x \to 0} g(x) x^{\alpha-1}(1-x)^{\beta-1} /
% %  B(\alpha,\beta)$. 
% $\lim_{x \to 1} g(x) x^{\alpha-1}(1-x)^{\beta-1} = \lim_{x \to 0} g(x) x^{\alpha-1}(1-x)^{\beta-1}$. 
% This operator 
leads to the Stein
  equation
  \begin{align*}
     \frac{\alpha-1-x(\alpha+\beta-2)}{x(1-x)}g_1(x)+ g'_1(x) = h(x) -
\mathbb{E}h(X)
  \end{align*}
with  conditions 
\begin{align*}
\int_0^{1}|(g_1(x) p_{\alpha, \beta}(x))'| \mathrm{d}x <\infty  \mbox{ and } \lim_{x \to 0} g_1(x)s
p_{\alpha, \beta}(x) =   \lim_{x \to 1}  g(x) p_{\alpha, \beta}(x).  
\end{align*}
The solution 
  \begin{align*}
    g_1(x) = \frac{1}{x^{\alpha-1}(1-x)^{\beta-1}} \int_{0}^x ( h(u) - \mathbb{E}h(X)) u^{\alpha-1}(1-u)^{\beta-1} \mathrm{d}u. 
\end{align*}      
    satisfies
  \begin{align*}
    & |g_1(x)|   \le  \kappa_1   M_{\alpha, \beta}(x); \quad 
  |g_1(x)|   \le  \kappa_2 \frac{x(1-x)}{\alpha+\beta}    \\ 
    & |g'_1(x)| \le \kappa_1  \left(1 +
      \frac{|\alpha-1-x(\alpha+\beta-2)|}{x(1-x)}
      M_{\alpha, \beta}(x)\right) \\
    & |g_1'(x)|   \le  \kappa_2  \left(
      \left|x-\frac{\alpha}{\alpha+\beta} \right|  +\left| 
x\left(   1-\frac{2}{\alpha+\beta}\right)-
      \frac{\alpha-1}{\alpha+\beta}  \right|\right)
  \end{align*}
    Taking $\eta(x) = -x$ in \eqref{eq:standop2} %  leads to the operator
  %   \begin{equation*}
  %   \mathcal{A}_2g(x) = \left(\frac{\alpha}{\alpha+\beta}-x\right)g(x) + \frac{x(1-x)}{\alpha+\beta} g'(x) 
  % \end{equation*}
  %     for $x\in(0,1)$, 
  %      acting on $\mathcal{F}(\mathcal{A}_2)$ the collection of
  %     test functions such that
  %       $\int_0^1 |( g(x) x^{\alpha}(1-x)^{\beta} /
  % ((\alpha+\beta)B(\alpha,\beta)))'| \mathrm{d}x < \infty$ and
  %     $\lim_{x \to 1} g(x) x^{\alpha}(1-x)^{\beta} = \lim_{x \to 0} g(x) x^{\alpha}(1-x)^{\beta}$.       
  %     This operator
      leads to the Stein
  equation,
  \begin{align*}
    \left(\frac{\alpha}{\alpha+\beta}-x\right)g_2(x) +
\frac{x(1-x)}{\alpha+\beta} g'_2(x) = h(x) - \mathbb{E}h(X)
  \end{align*}
with  conditions 
\begin{align*} 
\int_0^{1}|(x(1-x)g_2(x) p_{\alpha, \beta}(x))'| \mathrm{d}x <\infty  \mbox{ and } \lim_{x \to 0} x(1-x)g_2(x)
p_{\alpha, \beta}(x) =   \lim_{x \to 1} x(1-x) g_2(x) p_{\alpha, \beta}(x).  
\end{align*}
The solution
  \begin{align*}
    g_2(x) = \frac{\alpha+\beta}{x^{\alpha}(1-x)^{\beta}} \int_{0}^x ( h(u) - \mathbb{E}h(X)) u^{\alpha-1}(1-u)^{\beta-1} \mathrm{d}u
\end{align*}      
satisfies
\begin{align*}
%  & |g_2(x)|   \le \min \left( \kappa_1   \frac{P_{\alpha, \beta}(x)(1- P_{\alpha, \beta}(x))}{p_{\alpha, \beta}(x)} \frac{\alpha+\beta}{x(1-x)}
%      , \kappa_2 \right)
%       \le   \min\left(\frac{\kappa_1}{x(1-x)}, \kappa_{2} \right)\\ 
    & |g_2(x)|   \le \kappa_1  \frac{\alpha+\beta}{x(1-x)}  M_{\alpha,
      \beta}(x); \quad |g_2(x)|   \le  \kappa_2 \\ 
    & |g'_2(x)| \le  \kappa_1 \frac{\alpha+\beta}{x(1-x)} \left(1 + \left|\frac{\alpha}{\alpha+\beta}-x\right| \frac{1}{x(1-x)} \right) 
    \\ 
    & |g'_2(x)| \le 2 \kappa_2  \frac{\alpha+\beta}{x(1-x)}
      \left|x-\frac{\alpha}{\alpha+\beta} \right| ; \quad  |g'_2(x)| \le 2 \kappa_2  \frac{(\alpha+\beta)^2}{x^2(1-x)^2}
       \tilde{M}_{\alpha, \beta}(x)
    % & \qquad\qquad \le 2 \kappa_2 \frac{(\alpha+\beta)^2P_{\alpha, \beta}(x)(1-P_{\alpha, \beta}(x))}{x(1-x)p_{\alpha, \beta}(x)}
    % \le 2 \kappa_2 \frac{(\alpha+\beta)}{x(1-x)}.        
  \end{align*}

 \medskip
 \noindent   \emph{Literature review:}   The classic equation is
 \begin{align*}
     \left(\alpha- (\alpha+\beta)x\right)g_2(x) +
 x(1-x)g'_2(x) = h(x) - \mathbb{E}h(X)
 \end{align*}
 which is equivalent to our
 second equation, up to multiplication by $\alpha + \beta$.  Bounds
 on solutions to this equation are given in \cite[Proposition
 4.2]{dobler2015stein} and \cite[Lemma 3.2,
 3.4]{goldstein2013stein}. 
%  {\color{red}Both the non-uniform bound $B_{4, \alpha, \beta}^0$ and its uniform
%    version} improve strictly on the bounds from \cite[Proposition 4.2]{dobler2015stein} and \cite[Lemma 3.4]{goldstein2013stein} 
 Obviously, obtaining uniform bounds requires bounding
 $M_{\alpha, \beta}$ and $\tilde{M}_{\alpha, \beta}(x)$; bounds on
 these functions are provided in \cite{dobler2015stein}.
%  Obviously one may prefer a more
%  palatable expression for the upper bound.

% \begin{figure}
%   \centering
%     \includegraphics[width=0.4\textwidth]{gh-beta.png}
%     \includegraphics[width=0.4\textwidth]{Dgh-beta.png}
%   \caption{\small Solution \eqref{eq:sol-beta1} (left plot) and
%     absolute value of its derivative
% (right plot) for Beta target with parameter $(2,1)$, $c(x) = 1$ and $h(x)=\mathbb{I}[x\leq \xi]$. In both
% plots,  $\xi=0.1$ (blue curves), $\xi=0.5$
% (orange curves), and $\xi = 0.75$ (green curves)}
%   \label{fig:solbeta}
% \end{figure}

% \begin{figure}
%   \centering
%     \includegraphics[width=0.4\textwidth]{gh-beta2.png}
%     \includegraphics[width=0.4\textwidth]{Dgh-beta2.png}
%   \caption{\small Solution \eqref{eq:sol-beta2} (left plot) and
%     absolute value of its derivative
% (right plot) for Beta target with parameter $(2,1)$, $c(x) = x(1-x)/(\alpha+\beta)$ and $h(x)=\mathbb{I}[x\leq \xi]$. In both
% plots,  $\xi=0.1$ (blue curves), $\xi=0.5$
% (orange curves), and $\xi = 0.75$ (green curves)}
%   \label{fig:solbeta2}
% \end{figure}

\end{exm}

\begin{exm}[{Gamma distribution}] \label{ex:gamma} This distribution
 has pdf
  \begin{align*}
    p_{r, \lambda}(x) = \frac{\lambda^{r} x^{r-1}e^{- \lambda
        x}}{\Gamma(r)}, \quad \mathcal{S}(p_{r, \lambda}(x)) = (0, \infty). 
  \end{align*}
  The cdf $P_{r, \lambda}$ and survival $\bar{P}_{r, \lambda}$ do not
  bear a general explicit expression. Simple computations show that
  \begin{align*}
    \rho_{r, \lambda}(x)=  \frac{r-1}{x} - \lambda \mbox{ and } \tau_{r, \lambda}(x) =
    \frac{x}{\lambda}. 
  \end{align*}
  Taking $c(x)=1$ in \eqref{eq:standop} leads 
to the Stein
  equation
  \begin{align*}
    \left( \frac{r-1}{x} - \lambda \right)g_1(x)+ g'_1(x) = h(x) - \mathbb{E}h(X)
  \end{align*}
with  conditions 
\begin{align*}
\int_0^{\infty}|(g_1(x) p_{r, \lambda}(x))'| \mathrm{d}x < \infty \mbox{ and } \lim_{x \to 0} g_1(x) p_{r,
  \lambda}(x) =   \lim_{x \to \infty} g_1(x) p_{r, \lambda}(x).  
\end{align*}
The solution
  \begin{align*}
    g_1(x) = \frac{e^{\lambda x}}{x^{r-1}} \int_{0}^x ( h(u) - \mathbb{E}h(X)) u^{r-1}e^{-\lambda u} \mathrm{d}u.
  \end{align*}
  satisfies
  \begin{align*}
    & |g_1(x)|   \le  \kappa_1   M_{r, \lambda}(x) ; \quad |g_1(x)|   \le  \kappa_2 \frac{x}{\lambda}    \\ 
    & |g'_1(x)| \le \kappa_1  \left(1 +
      \left| \frac{r-1}{x} - \lambda \right| M_{r, \lambda}(x) \right); \quad  |g_1'(x)|   \le  \kappa_2  \left(  \left|x-\frac{r}{\lambda}
      \right|  +
      \left|x- \frac{r-1}{\lambda}\right| \right)
  \end{align*}
  Taking $\eta(x) = -x$ in \eqref{eq:standop2} %  leads to the operator
  %   \begin{equation*}
  %   \mathcal{A}_2g(x) = \left(\frac{r}{\lambda}-x\right)g(x) + \frac{x}{\lambda} g'(x) 
  % \end{equation*}
  % for $x\geq 0$, acting on $\mathcal{F}(\mathcal{A}_2)$ the collection
  % of test functions such that
  % $\int_0^\infty |(x/\lambda) g(x) p_{r, \lambda}(x)))'| \mathrm{d}x <
  % \infty$ and
  % $\lim_{x \to 0} (x/\lambda) g(x) p_{r, \lambda}(x) = \lim_{x \to
  %   \infty} (x/\lambda) g(x) p_{r, \lambda}(x)$.  This operator 
  leads
  to the Stein equation
  \begin{align*}
     \left(\frac{r}{\lambda}-x\right)g_2(x) + \frac{x}{\lambda} g'_2(x) = h(x) - \mathbb{E}h(X)
  \end{align*}
with  conditions 
\begin{align*}
\int_0^{\infty}|(xg_2(x) p_{r, \lambda}(x))'| \mathrm{d}x < \infty  \mbox{ and } \lim_{x \to 0}  x g_2(x) p_{r,
  \lambda}(x) =   \lim_{x \to \infty} xg_2(x) p_{r, \lambda}(x).  
\end{align*}
The solution  
  \begin{align*}
    g_2(x) =\frac{\lambda e^{\lambda x}}{x^{r}} \int_{0}^x ( h(u) - \mathbb{E}h(X)) u^{r-1}e^{-\lambda u} \mathrm{d}u
  \end{align*}
  satisfies
  \begin{align*}
    & |g_2(x)|   \le  \kappa_1   M_{\lambda, r}(x) \frac{\lambda}{x};
      \quad |g_2(x)|   \le \kappa_2  \\
    & |g'_2(x)| \le  \kappa_1 \frac{\lambda}{x} \left(1 + \left|x-\frac{r}{\lambda}\right| \frac{\lambda}{x}  M_{r,\lambda}(x) \right)     
    \\ 
    & |g'_2(x)| \le 2  \kappa_2\frac{\lambda}{x}  \left(
      \left|x-\frac{r}{\lambda} \right| \right); \quad |g'_2(x)| \le 2   \kappa_2  \frac{\lambda^2}{x^2}
\tilde{M}_{r,\lambda}(x). 
  \end{align*}

\medskip
\noindent   \emph{Literature review:}  There is interest in the
literature for the particular choices 
  $r = \nu/2$ and $\lambda = 1/2$ (chi-square distribution) and
  $r = 1$ (exponential distribution) with operators
  \begin{align*}
   &  \mathcal{A}g(x) =     x g'(x) - (r - \lambda x) g(x) \\
   &  \mathcal{A}g(x) =     g'(x) - \lambda g(x) \qquad (\mbox{exponential distribution})
  \end{align*}
  Our bounds apply to the $\chi^2$ and exponential as well, although
  in this last case further simplifications follow from the fact that
  \begin{align*}
    M_{\lambda}(x) = \frac{1-e^{-x \lambda}}{\lambda} \approx \frac{1}{\lambda} \mbox{ and }
    \tilde{M}_{\lambda} = \frac{\lambda x - 1 + e^{-x
        \lambda}}{\lambda^3} \approx \frac{x}{\lambda^2}.
  \end{align*}
Comparable bounds  from the literature can be found in
  \cite[Theorem 2.6]{Luk1994}, \cite[Theorem 3.4]{Pi04},
  \cite{gaunt2013rates} or \cite[Theorem 2.2]{gaunt2016chi} and
  \cite[Theorem 2.1]{dobler2018gamma} and
  \cite{chatterjee2011nonnormal}. Our non uniform bounds improve on
  the available ones whenever they are comparable. In particular, the
  first bound from \cite[Theorem 2.1 equation (19)]{dobler2018gamma}
  follows immediately from ours (recall that it is necessary to divide
  by $\lambda$), and the second bound as expressed in their equation
  (21) follows from the fact that $|g'_2(x)| \le 2\lambda $ uniformly
  in $x, r, \lambda$ positive.  It is interesting to note that the
  dependence on $\lambda$ is linear (and hence, in the classic
  parametrization, there is no dependence on $\lambda$ for this upper
  bound).

  \end{exm}
 
\begin{exm}[Student distribution]\label{ex:student}
%This is a continuous distribution, so $\ell=0$. 
This distribution has pdf 
\begin{align*}
  p_{\nu}(x) = \frac{(\nu/(\nu+x^2))^{(1+\nu)/2} \nu^{-1/2}}{
  B(\nu/2,1/2)}, \quad \mathcal{S}(p_\nu)=(-\infty, \infty).
\end{align*}
 The cdf
  $P_{\nu}$  and survival $\bar{P}_{\nu}$  do not bear an explicit
  expression. Simple computations show that
  \begin{align*}
    \rho_{\nu}(x) = - \frac{x(\nu+1)}{x^2+\nu} \mbox{ and }
    \tau_{\nu}(x) =\frac{x^2+\nu}{\nu-1}. 
  \end{align*}
  Taking $c(x)=1$ in \eqref{eq:standop}  leads %  to the operator
  % \begin{align**}
  %   \mathcal{A}_1g(x) =  - \frac{x(\nu+1)}{x^2+\nu} g(x) +  g'(x)
  % \end{align**}
  % acting on $\mathcal{F}(\mathcal{A}_1)$ the collection of test
  % functions such that $\int_0^1 | ( g(x)
  %     (\nu/(\nu+x^2))^{(1+\nu)/2}  )'|
  % \mathrm{d}x < \infty$ and $\lim_{x\to \infty} g(x)
  %     (\nu/(\nu+x^2))^{(1+\nu)/2}  = \lim_{x\to -\infty} g(x)
  %     (\nu/(\nu+x^2))^{(1+\nu)/2}$.  This operator leads
      to the Stein
      equation
      \begin{align*}
   - \frac{x(\nu+1)}{x^2+\nu} g_1(x) +  g_1'(x)  = h(x) - \mathbb{E}h(X)     
      \end{align*}
      with  conditions 
\begin{align*}
\int_{-\infty}^{\infty}|(g_1(x) p_{\nu}(x))'| \mathrm{d}x < \infty  \mbox{ and
} \lim_{x \to -\infty}   g_1(x) p_{\nu}(x) =   \lim_{x \to \infty}  g_1(x) p_{\nu}(x).  
\end{align*}
The solution given by
      \begin{align*}    
        g_1(x) =  (\nu+x^2)^{(1+\nu)/2}   \int_{-\infty}^x
        \left( h(u) -\mathbb{E}h(X) \right) / (\nu+u^2)^{(1+\nu)/2} \mathrm{d}u.
      \end{align*}
satisfies 
       \begin{align*}
    & |g_1(x)|   \le  \kappa_1   M_{\nu}(x) ; \quad |g_1(x)|   \le
      \kappa_2 {\frac{x^2+\nu}{\nu-1}}\\ 
    & |g'_1(x)| \le \kappa_1  \left(1 +
      \left|  \frac{{x}(\nu+1)}{x^2+\nu} \right| M_{\nu}(x) \right);
      \quad  |g_1'(x)|   \le  \kappa_2  \left(
 |x|  + {|x|}  \left|  \frac{\nu+1}{\nu-1} \right|     \right) {= \kappa_2 |x| \frac{2\nu}{\nu-1}}
       \end{align*}
  Taking $\eta(x) = -x$ in \eqref{eq:standop2} leads to %  the operator
%     \begin{equation*}
%     \mathcal{A}_2g(x) = -xg(x) + \frac{x^2+\nu}{\nu-1} g'(x) 
%   \end{equation*}
% acting on $\mathcal{F}(\mathcal{A}_2)$ the collection
%   of test functions such that
%   $\int_0^\infty |(x g(x) p_{\nu}(x))'| \mathrm{d}x <
%   \infty$ and
%   $\lim_{x \to -    \infty} x g(x) p_{\nu}(x) = \lim_{x \to
%     \infty} x g(x) p_{\nu}(x)$.  This operator leads
%   to 
  the Stein equation
  \begin{align*}
    -xg_2(x) + \frac{x^2+\nu}{\nu-1} g_2'(x)   = h(x) - \mathbb{E}h(X)
  \end{align*}
      with  conditions 
{
	\begin{align*}
\int_{-\infty}^{\infty}|((x^2+\nu)g_2(x) p_{\nu}(x))'| \mathrm{d}x < \infty  \mbox{ and
} \lim_{x \to -\infty}   (x^2+\nu) g_2(x) p_{\nu}(x) =   \lim_{x \to \infty}   (x^2+\nu) g_2(x) p_{\nu}(x).  
\end{align*}
}
The solution 
  \begin{align*}
    g_2(x) ={(\nu-1)(\nu+x^2)^{(\nu-1)/2}}   \int_{-\infty}^x
        \left( h(u) -\mathbb{E}h(X) \right) / (\nu+u^2)^{(1+\nu)/2}
        \mathrm{d}u
\end{align*}      
satisfies
\begin{align*}
    & |g_2(x)|   \le  \kappa_1   M_{\nu}(x) \frac{\nu-1}{x^2+\nu} ;
      \quad |g_2(x)|   \le \kappa_2  \\
    & |g'_2(x)| \le  \kappa_1  \frac{\nu-1}{x^2+\nu}  \left(1 + |x|
      \frac{\nu-1}{x^2+\nu} M_{\nu}(x) \right)      
    \\ 
    & |g'_2(x)| \le 2  \kappa_2  |x|\frac{\nu-1}{x^2+\nu}; \quad
      |g'_2(x)| \le 2   \kappa_2   \left(  \frac{\nu-1}{x^2+\nu}  \right)^2
\tilde{M}_{\nu}(x). 
  \end{align*}

\medskip
\noindent   \emph{Literature review:}    An early reference on Stein operators for Student distribution is
\cite{schoutens2001orthogonal} which considers   operator
\begin{align*}
 \mathcal{A}g(x) =  (x^2+\nu) g'(x) -  (\nu-1) x g(x) 
\end{align*}
(see also \cite{ley2017stein}). Our bounds seem to outperform those
from  \cite[More complete report,
p23]{schoutens2001orthogonal}.

\end{exm}
\begin{exm}[{Fr\'echet distribution}]
\label{sec:frechet-distribution}

This distribution has pdf 
\begin{align*}
  p_{\alpha}(x) = \alpha x^{-\alpha-1} e^{-x^{-\alpha}}, \quad \mathcal{S}(p_\alpha)=(0, \infty).
\end{align*}
with cdf and survival 
\begin{align*}
  P_{\alpha}(x) = e^{-x^{-\alpha}} \mbox{ and } \bar{P}_{\alpha}(x) = 1- e^{-x^{-\alpha}} 
\end{align*}
so that
\begin{align*}
  M_{\alpha}(x) = \frac{x^{1+\alpha}(1-e^{-x^{-\alpha}})}{\alpha} 
\end{align*}
but the function $\tilde{M}_{\alpha}$ does not bear an explicit
expression.  Simple computations show that
$ \rho_{\alpha}(x) = {\alpha}{x^{-\alpha-1}} - (1+\alpha){x}^{-1}$ but
the Stein kernel $\tau_{\alpha}$ does not bear an explicit expression.
Hence the different bounds obtained with the choices $c = 1$ or
$c= \tau$ will not lead to explicit results and we do not report them
here -- they remain computable nevertheless. % It is easy to see that an
% alternative route to good bounds is obtained by taking $c(x)=x$ in
% \eqref{eq:standop} to get the Stein equation
%       \begin{align*}
% \alpha(x^{-\alpha}-1) g_1(x) +  x g_1'(x)  = h(x) - \mathbb{E}h(X)     
%       \end{align*}
%       with  conditions 
% \begin{align*}
% \int_0^{\infty}|(x g_1(x) p_{\alpha}(x))'| \mathrm{d}x < \infty  \mbox{ and
% } \lim_{x \to 0}  x  g_1(x) p_{\alpha}(x) =   \lim_{x \to \infty} x  g_1(x) p_{\alpha}(x).  
% \end{align*}
% One can check that $c(x) = -\mathcal{L}_p\eta(x)$ with $\eta(x) =
% \alpha(x^{-\alpha}-1)$. 
% The solution given by
%       \begin{align*}
%         g_1(x) =  x^{\alpha} e^{x^{-\alpha}}  \int_{0}^x
%         \left( h(u) -\mathbb{E}h(X) \right) u^{-\alpha-1} e^{-u^{-\alpha}} \mathrm{d}u.
%       \end{align*}
% satisfies 
%        \begin{align*}
%     & |g_1(x)|   \le  \kappa_1   \frac{x^{\alpha}(1-e^{-x^{-\alpha}})}{\alpha}
%     & |g'_1(x)| \le \kappa_1   \frac{1}{x} \left( 1+ {(1- x^{\alpha})}
% (1-e^{-x^{-\alpha}})\right) 
%        \end{align*}
Another potentially interesting choice is  $c(x)=x^{\alpha+1}$ in
\eqref{eq:standop} to get the Stein equation
      \begin{equation*}
 \alpha g_2(x) +  x^{\alpha+1} g_2'(x)  = h(x) - \mathbb{E}h(X)     
      \end{equation*}
      with  conditions 
{      
\begin{equation*}
\int_0^{\infty}|( g_2(x) x^{\alpha+1})'| \mathrm{d}x < \infty  \mbox{ and
} \lim_{x \to 0}    g_2(x) x^{\alpha+1} =   \lim_{x \to \infty} g_2(x)
x^{\alpha+1}. 
\end{equation*}
}
The solution given by
      \begin{equation*}
        g_2(x) =  e^{x^{-\alpha}}  \int_{0}^x
        \left( h(u) -\mathbb{E}h(X) \right) u^{-\alpha-1} e^{-u^{-\alpha}} \mathrm{d}u.
      \end{equation*}
satisfies 
       \begin{align*}
    & |g_2(x)|   \le  \kappa_1   \frac{1-e^{-x^{-\alpha}}}{\alpha}\\
    & |g'_2(x)| \le \kappa_1   \frac{1}{x^{\alpha+1}} \left( 1+ { 
(1-e^{-x^{-\alpha}})}\right) 
       \end{align*}
       It is likely that other choices of $c$ lead to other
       interesting equations and bounds, but we leave this to ulterior
       investigations. We refer to \cite[Section 2.6]{ley2017stein}. 

     \end{exm}

     \begin{exm}[Rayleigh distribution]
\label{sec:rayl-distr}

This distribution with support $(0, \infty)$ has explicit pdf, cdf and
survival function given by
\begin{equation*}
  p_{\mathrm{r}}(x) = 2x e^{-x^2}, \,  P_{\mathrm{r}}(x) = 1 - e^{-x^2}  \mbox{ and } \bar{P}_{\mathrm{r}}(x)
  = e^{-x^2},
\end{equation*}
respectively. The mean and variance of $p_{\mathrm{r}}$ are
$\sqrt \pi /2$ and $1-\pi/4$, respectively.  Also
\begin{equation*}
  M_{\mathrm{r}}(x) =  \frac{1 - e^{-x^2}}{2x} \mbox{ and } \tilde{M}_{\mathrm{r}}(x) =
  \frac{\int_0^x (1 - e^{-u^2}) \mathrm{d}u \int_x^{\infty} e^{-u^2}
    \mathrm{d}u}{2x e^{-x^2}} = \frac{(x-\sqrt{\pi} \bar{\Phi}(\sqrt{2}x))\sqrt{\pi} \bar{\Phi}(\sqrt{2}x)}{2x e^{-x^2}}.  
\end{equation*}
Simple  computations  show that
\begin{equation*}
  \rho_{\mathrm{r}}(x) = \frac{1}{x} - 2x \mbox{ and }
  \tau_{\mathrm{r}} (x) = \frac{
  2x + 2\sqrt{\pi}e^{x^2}\bar{\Phi}(\sqrt{2}x)-\sqrt{\pi} }{4x}
\end{equation*}
and also
\begin{equation*}
  0 \le  \tau_{\mathrm{r}} (x) \le \frac{1}{2}. 
\end{equation*}
Taking $c(x)=1$ in \eqref{eq:standop}  leads  
      to the Stein
      equation
      \begin{equation*}
\Big(\frac{1}{x} - 2x\Big) g_1(x) +  g_1'(x)  = h(x) - \mathbb{E}h(X)     
      \end{equation*}
      with  conditions 
\begin{equation*}
\int_{-\infty}^{\infty}|(g_1(x) p_{\mathrm{r}}(x))'| \mathrm{d}x < \infty  \mbox{ and
} \lim_{x \to 0}   g_1(x) p_{\mathrm{r}}(x) =   \lim_{x \to \infty}
g_1(x) p_{\mathrm{r}}(x).   
\end{equation*}
The solution given by
      \begin{equation*}
        g_1(x) =  \frac{e^{x^2}}{x }   \int_0^x
        \left( h(u) -\mathbb{E}h(X) \right)u e^{-u^2} \mathrm{d}u.
      \end{equation*}
satisfies 
       \begin{align*}
         & |g_1(x)|   \le  \kappa_1   \frac{1 - e^{-x^2}}{2x} ; \quad |g_1(x)|   \le
           \kappa_2 \frac{\sqrt{\pi} -
           2x - 2\sqrt{\pi}e^{x^2}\bar{\Phi}(\sqrt{2}x)}{4x} \le \frac{\kappa_2}{2}\\ 
         & |g'_1(x)| \le \kappa_1  \left(1 + \tau_r(x)
           M_{\mathrm{r}}(x) \right) \le  \kappa_1  \left(1 +  
          \frac{ M_{\mathrm{r}}(x) }{2}\right)\\
         &   |g_1'(x)|   \le  \kappa_2  \left(
           \Big|x-\frac{\sqrt \pi}{2}\Big|  +  \left|\Big(\frac{1}{x} -
           2x\Big)
           \right| \tau_{\mathrm{r}}(x)    \right) \le
           \kappa_2  \left(
           \Big|x-\frac{\sqrt \pi}{2}\Big|  +  \frac{1}{2}\left|\frac{1}{x} -
           2x
           \right|  \right) 
       \end{align*}
       Taking $\eta(x) = -x$ in \eqref{eq:standop2} leads to non
       explicit equations and bounds which are therefore not
       reproduced here.
\end{exm}

\begin{exm}[Binomial distribution] \label{ex:binom} 
This distribution has pmf 
\begin{align*}
  p_{n, \theta}(x) = \binom{n}{x} \theta^x(1-\theta)^{n-x}, \quad
  \mathcal{S}(p_{n, \theta}) = \left\{ 0, \ldots, n \right\}. 
\end{align*}
The cdf  $P_{n, \theta}$  and
  survival $\bar{P}_{n, \theta}$ do not bear an explicit
  expression. Simple computations show that
  \begin{align*}
 & \rho^-_{n,\theta}(x) = \frac{(n+1)\theta-x}{\theta(n+1-x)}  &  \rho^+_{n,\theta}(x) = \frac{(n+1)\theta-(x+1)}{(x+1)(1-\theta)}
    \\
  &  \tau^-_{n,\theta}(x) = \theta(n-x) &  \tau^+_{n,\theta}(x) = (1-\theta)x 
  \end{align*}
  The Stein equations associated to $\rho^{\pm}_{n,\theta}$ are, on the one hand,
  \begin{align*}
 \frac{(n+1)\theta-x}{\theta(n+1-x)} g_1^{-}(x) +
    \Delta^+g_1^{-}(x) = h(x) - \mathbb{E}h(X)  
  \end{align*}
   with conditions
  \begin{align*}
\sum_{j=0}^n \left| \Delta^-(g_1^-(j+1) p_{n, \theta}(j)) \right| < \infty \mbox{ and }  g_1^-(n+1) p_{n, \theta}(n) = 0
\end{align*}
and solution
\begin{align*}
  g_1^-(x) = \frac{1}{p_{n, \theta}(x-1)} \sum_{j=0}^{x-1} (h(j) -  \mathbb{E}h(X)) p_{n, \theta}(j)
\end{align*}
  and,  on the other hand,  
  \begin{align*}
 \frac{(n+1)\theta-(x+1)}{(x+1)(1-\theta)} g_1^{+}(x) +
 \Delta^-g_1^{+}(x) =     h(x) - \mathbb{E}h(X) 
  \end{align*}
  with conditions
  \begin{align*}
    \sum_{j=0}^n \left| \Delta^+ (g_1^+(j-1) p_{n, \theta}(j)) \right| < \infty
    \mbox{ and }
    g_1^+(-1) p_{n, \theta}(0) = 0 
  \end{align*}
  and solution
  \begin{align*}
  g_1^+(x) &= \frac{1}{p_{n, \theta}(x+1)} \sum_{j=0}^{x} (h(j) -  \mathbb{E}h(X)) p_{n, \theta}(j) \\
  &= \frac{(1-\theta)(x+1)}{\theta(n-x)} g_1^-(x+1).
\end{align*}

These functions satisfy
 \begin{align*}
    & |g_1^-(x)| \le \kappa_1 M_{n,\theta}(x-1) ;\\
    & |g_1^+(x)|   \le  \kappa_1 M_{n, \theta}(x) \frac{(x+1)(1-\theta)}{\theta(n-x)}  ; \\
    & |g_1^-(x)| \le \kappa_2 \theta (n-x+1) ; \\
    & |g_1^+(x)| \le \kappa_2  (1-\theta)(x+1); \\
    & |\Delta^+{g_1^-}(x)| \le  \kappa_1   \left( 1 + \frac{(n+1)\theta-x}{\theta(n-x+1)}M_{n,\theta}(x-1) \right)
    \\ 
    & |\Delta^-{g_1^+}(x)| \le  \kappa_1   \left( 1 + \frac{(n+1)\theta-(x+1)}{\theta(n-x)}M_{n,\theta}(x) \right)
    \\ 
    & |\Delta^+g_1^-(x)| \le \kappa_2 (|x-n\theta| + |x-(n+1)\theta|); 
    \\
    &  |\Delta^- g_1^{+}(x)| \le \kappa_2  (|x-n\theta| + |x+1-(n+1)\theta|); 
 \end{align*}
  The
  Stein equations associated to $\tau^{\pm}_{n,\theta}$  are on the
  one hand 
  \begin{align*}
  (x-n\theta)g_2^{-}(x) - \theta(n-x)\Delta^+g_2^{-}(x) =     h(x) - \mathbb{E}h(X)
  \end{align*}
  with condition
  \begin{align*}
  \sum_{j=0}^n \left| \Delta^- ( \theta(n-j) g_2^-(j+1) p_{n, \theta}(j)) \right| < \infty
  \end{align*}
  and      on the other hand  
  \begin{align*}
(x-n\theta)g_2^{+}(x) -(1-\theta)x \Delta^-g_2^{+}(x) = h(x) - \mathbb{E}h(X) 
  \end{align*}
  with condition
  \begin{align*}
\sum_{j=0}^n \left| \Delta^+((1-\theta) j g_2^+(j-1) p_{n, \theta}(j))
\right| < \infty
\end{align*}
(in both cases the border conditions disappear because of the premultiplying factor). 
These functions satisfy
\begin{align*}
g_2^+(x) = g_2^-(x+1) \mbox{ and } |\Delta^+{g_2^-}(x)| = |\Delta^-{g_2^+}(x)| =: |\Delta{g_2}(x)|.
\end{align*}
Moreover, 
 \begin{align*}
    & |g_2^+(x)| %= |g_2^-(x+1)|
     \le     \frac{\kappa_1 M_{n, \theta}(x)}{\theta(n-x)} \wedge \kappa_2 ; \\
    & |\Delta{g_2}(x)|\le  \left( \frac{ \kappa_1 }{\theta(n-x)}  \left(1 + 
      \frac{|x-n\theta|M_{n,\theta}(x-1) }{\theta(n-x+1)} \right)  \right) \wedge  
      \left( \frac{ \kappa_1 }{(1-\theta)x}  \left(1 + \frac{|x-n\theta|M_{n,\theta}(x)}{\theta(n-x)}  \right)  \right)   
    \\ 
    &|\Delta{g_2}(x)|  \le 2  \kappa_2 |x-n\theta| \left(\frac{1}{\theta(n-x)} \wedge \frac{1}{(1-\theta)x}\right) \\
    & |\Delta{g_2}(x)|  \le 2 \kappa_2 \left( \frac{\tilde{M}_{n,\theta}^-(x)}{\theta^2(n-x)(n-x+1)} \wedge \frac{\tilde{M}_{n,\theta}^+(x)}{(1-\theta)^2 x(x+1)} \right)
  \end{align*} 
  If $h(x) = \mathbb{I}[x = \xi]$ is point mass, we can also use
  Lemma~\ref{lem:pointmass2} because the binomial distribution
  satisfies the conditions (monotonicity of the two ratios for any
  $\xi \in \mathcal{S}(p_{n, \theta})$).  Therefore, the solution of
  equation \eqref{eq:eq-indicator} is also bounded by
  \eqref{eq:supg-indicator}:
\begin{align*}
\| g_\xi\|_\infty \leq \max\left\{\frac{P_{n, \theta}(\xi-1)}{(1-\theta)\xi}, \frac{1-P_{n, \theta}(\xi)}{\theta(n-\xi)}\right\} 
\end{align*}
and the bound \eqref{eq:supder-indicator} becomes
\begin{align}\label{eq:bindmass}
  ||\Delta g_\xi||_\infty &
                            = \frac{P_{n, \theta}(\xi-1)}{(1-\theta)\xi} + \frac{1-P_{n, \theta}(\xi)}{\theta(n-\xi)}
                            \leq
                            \min\left\{\frac{1}{\xi(1-\theta)},\frac{1}{\theta(n-\xi)}\right\}.                               
\end{align}

\medskip
\noindent \emph{Literature review:}
The classic equation for Binomial target is
  \begin{equation*}
        (1-\theta)x g(x) - \theta(n-x) g(x+1)=     h(x) - \mathbb{E}h(X).
      \end{equation*}
      The bound \eqref{eq:bindmass} is of the same order as the
      corresponding bound in \cite[Example
      2.11]{eichelsbacher2008stein}. Moreover, it outperforms the
      uniform bound from \cite[Lemma 1]{ehm1991binomial}.  Our
      non-uniform bound is smaller than the uniform bound in
      \cite{barbour1992poisson} but the expression is not well
      readable.  By \cite[Theorem 1]{diaconis1991closed}, the Mills
      ratio for the binomial distribution satisfies
  $$\frac{x}{n} \leq \frac{1-P(x-1)}{p(x)} \leq
  \frac{x(1-\theta)}{x-n\theta}$$ for $x>n\theta$. Therefore, we easily
  deduce more readable bounds for the ratio
\begin{equation*}
M_{n, \theta}(x) \leq \frac{x(1-\theta)}{x-n\theta} \vee
\frac{(n-x)\theta}{n\theta-x}. 
\end{equation*}
 This could be inserted into the previous bounds to increase their
 readability. 
\end{exm}

\begin{exm}[Negative binomial distribution] \label{ex:NB}  
This distribution  has pmf 
\begin{align*}
  p_{r, \theta}(x) = (1-\theta)^r \theta^x \frac{\Gamma(x+r)}{x!
  \Gamma(r)}, \quad \mathcal{S}(p_{r, \theta}) = \N
\end{align*}
The cdf $P_{r, \theta}$ and survival function $\bar P_{r, \theta}$ do
not bear an explicit expression. The mean is $\theta
r/(1-\theta)$. Simple computations show that
\begin{align*}
  &    \rho^-(x) = 1-\frac{x}{\theta(x-1+r)} &  \rho^+(x) = \frac{(x+r)\theta}{x+1}-1 \\
  & \tau^-(x) = \frac{\theta}{1-\theta}(r+x) &   \tau^+(x) = \frac{1}{1-\theta}x.
\end{align*}
 The Stein equations associated to $\rho^{\pm}$ are, on the one hand,
   \begin{align*}
 \left( 1-\frac{x}{\theta(x-1+r)} \right)  g_1^{-}(x) +
 \Delta^+g_1^{-}(x) =     h(x) - \mathbb{E}h(X) 
 \end{align*}
 with conditions
 \begin{align*}
 \sum_{j=0}^\infty \left| \Delta^- (g_1^-(j+1) p_{r, \theta}(j)) \right| < \infty
 \mbox{ and }
 \lim_{n \to \infty} g_1^-(n+1) p_{r, \theta}(n) = 0 \end{align*}
 and solution
 \begin{align*}
 g_1^-(x) = \frac{1}{p_{r, \theta}(x-1)} \sum_{j=0}^{x-1} (h(j) -
 \mathbb{E}h(X)) p_{r, \theta}(j)
 \end{align*}
  and,  on the other hand,  
 \begin{align*}
\left(\frac{(x+r)\theta}{x+1}-1 \right)  g_1^{+}(x) +
\Delta^-g_1^{+}(x) = h(x) - \mathbb{E}h(X)  
\end{align*}
with conditions
\begin{align*}
\sum_{j=0}^\infty \left| \Delta^+(g_1^+(j-1) p_{r, \theta}(j)) \right| < \infty \mbox{
	and }  g_1^+(-1) p_{r, \theta}(0) = 0
\end{align*}
and solution
\begin{align*}
g_1^+(x) = \frac{1}{p_{r, \theta}(x+1)} \sum_{j=0}^{x} (h(j) -
\mathbb{E}h(X)) p_{r, \theta}(j).
\end{align*} 
These functions satisfy
 \begin{align*}
& |g_1^-(x)| \le \kappa_1 M_{r,\theta}(x-1) ;\\
& |g_1^+(x)|   \le  \kappa_1 M_{r, \theta}(x) \frac{(x+r)\theta}{x+1}  ; \\
& |g_1^-(x)| \le \kappa_2 \frac{\theta (r+x-1)}{1-\theta} ; \\
& |g_1^+(x)| \le \kappa_2  \frac{1}{1-\theta}(x+1); \\
& |\Delta^+{g_1^-}(x)| \le  \kappa_1   \left( 1 + \left|1-\frac{x}{\theta(x-1+r)}\right| M_{r,\theta}(x-1) \right);
\\ 
& |\Delta^-{g_1^+}(x)| \le  \kappa_1   \left( 1 + \frac{|x(1-\theta)+1-r\theta|}{\theta(r+x)}M_{r,\theta}(x) \right);
\\ 
& |\Delta^+g_1^-(x)| \le \kappa_2 \left(\left|x-\frac{\theta r}{1-\theta}\right| + \left|x-\frac{\theta (r-1)}{1-\theta}\right|\right); 
\\
&  |\Delta^- g_1^{+}(x)| \le \kappa_2  \left(\left|x-\frac{\theta r}{1-\theta}\right| + \left|x-\frac{1-r\theta}{1-\theta}\right|\right).
\end{align*}
  The
  Stein equations associated to $\tau^{\pm}$  are
  \begin{align*}
\left(x-\frac{\theta r}{1-\theta}\right)g_2^{-}(x) -\frac{\theta }{1-\theta}(x+r) \Delta^+g_2^{-}(x) = h(x) - \mathbb{E}h(X) 
  \end{align*}
  with condition
  \begin{align*}
\sum_{j=0}^\infty \left| \Delta^+((r+j) g_2^-(j+1) p_{r, \theta}(j))
\right| < \infty
\end{align*}
  and
  \begin{align*}
\left(x-\frac{\theta r}{1-\theta}\right)g_2^{+}(x) -\frac{1 }{1-\theta}x \Delta^-g_2^{+}(x) = h(x) - \mathbb{E}h(X) 
  \end{align*}
  with condition
  \begin{align*}
\sum_{j=0}^\infty \left| \Delta^-(j g_2^+(j-1) p_{r, \theta}(j))
\right| < \infty
\end{align*}
(in both cases the border conditions disappear because of the
premultiplying factor).  
These functions satisfy
\begin{align*}
  g_2^+(x) = g_2^-(x+1) \mbox{ and } |\Delta^+{g_2^-}(x)| = |\Delta^-{g_2^+}(x)| =: |\Delta{g_2}(x)|
\end{align*}
and  
\begin{align*}
& |g_2^+(x)| %= |g_2^-(x+1)|
\le   \kappa_1 M_{r, \theta}(x) \frac{ 1-\theta}{\theta(x+r)} \wedge \kappa_2 ; \\
& |\Delta{g_2}(x)|\le  \left( \frac{ \kappa_1(1-\theta) }{\theta(r+x)}  \left(1 + 
\left|\frac{(1-\theta)x-\theta r }{\theta(r+x-1)} \right| M_{r\theta}(x-1)\right)  \right) \\
& |\Delta{g_2}(x)|\le
\left( \frac{ \kappa_1(1-\theta) }{x}  \left(1 + 
\frac{\left|(1-\theta)x-\theta r \right| }{\theta(r+x)} M_{r\theta}(x)\right)  \right)
\\ 
&|\Delta{g_2}(x)|  \le 2  \kappa_2 |(1-\theta)x-r\theta| \left(\frac{1}{\theta(r+x)} \wedge \frac{1}{x}\right) \\
& |\Delta{g_2}(x)|  \le 2 \kappa_2 (1-\theta)^2 \left( \frac{\tilde{M}_{r,\theta}^-(x)}{\theta^2(r+x)(r+x-1)} \wedge \frac{\tilde{M}_{r,\theta}^+(x)}{ x(x+1)} \right)
\end{align*} 
 If, moreover, $h$ is an indicator function, the bound
 \eqref{eq:supder-indicator} becomes
\begin{align}\label{eq:brow}
||\Delta g_\xi||_\infty &
= (1-\theta)\frac{P_{r,\theta}(\xi-1)}{\xi} + (1-\theta)\frac{1-P_{r,\theta}(\xi)}{\theta(r+\xi)}
\leq \min\left\{\frac{1-\theta}{\xi},\frac{1-\theta}{\theta(r+\xi)}\right\} .
\end{align}
For any Borel set $A\subset \mathcal{S}(p_{r, \theta})$, the solution is bounded by \eqref{eq:supg-borel}
\begin{align*}
\|g_A\|_\infty \leq 
\left( \sum_{j\in A}p_{r,\theta}(j) \right) \sup_{\xi\in A}\left\{ \frac{1-\theta}{\xi p_{r,\theta}(\xi)}, \frac{1-\theta}{\theta(r+x) p_{r,\theta}(\xi)}\right\}
\end{align*}
and the bound \eqref{eq:der-borel} gives
\begin{align*}
||\Delta g_A||_\infty &\leq 
(1-\theta) \sup_{x \in A} \left(\frac{P_{r,\theta}(x-1)}{x} + \frac{1-P_{r,\theta}(x)}{\theta(r+x)}\right). 
\end{align*}

\medskip 

\noindent \emph{Literature review:} Something about the he case $\ell=-1$ is the most
developed in the literature (see for instance \cite{brown1999negative,
  brown2001stein, barbour2015stein, cloez2019intertwinings}). The
operator is given in \cite{barbour2015stein} (see their equation
(1.1)).  Bound \eqref{eq:brow} is the bound of \cite[Theorem
2.10]{brown2001stein}, which improves the one of \cite[Lemma
5]{brown1999negative}.
  Something is precisely     the bound (1.3) in
  \cite{barbour2015stein}  
We note that the bound $B_{2, r, \theta}^\pm(x) = 1$ yields
whereas $
    B_{4, r, \theta}$ is of the same order but (strictly)
    uniformly smaller than the corresponding bound (1.4) in
    \cite{barbour2015stein} and similar to the improved version of
    this bound \cite[Prop. 4.4]{cloez2019intertwinings}.  
 
\end{exm}

\section{More bounds on IPMs}
\label{sec:bounds-ipms-1}

In this section we apply the material from Section \ref{sec:discr},
particularly Corollaries \ref{cor:kolbound} and \ref{cor:prec2}, to
two more examples. We conclude with two examples illustrating how the
material can be used in more generality.

% \begin{exm}[Binomial vs Gaussian]
  
% \end{exm}

\begin{exm}[Rayleigh  approximation]
  We wish to compare distributions characterized by
  $ p_{\infty}(x) = 2x e^{-x^2} \mathbb{I}[0 \le x \le \infty], \quad
  P_{\infty}(x) = (1 - e^{-x^2}) \mathbb{I}[0 \le x \le \infty] $
  (Rayleigh distribution, Example \ref{sec:rayl-distr}) and the
  distribution with pdf and cdf
  $p_n(x) = {2}/{n} (n-1) x \left( 1- {x^2}/{n} \right)^{n-2}
  \mathbb{I}[0 \le x \le \sqrt{n}],$
  $P_{n}(x) = 1 - \left( {n}/{(n-x^2)} \right)^{1-n} \mathbb{I}[0 \le
  x \le \sqrt{n}], $ respectively.  We have already computed
  $\rho_{\infty}$, $\tau_{\infty}$ and $M_\infty$. We also immediately
  obtain
\begin{align*} 
  \rho_n(x) -  \rho_\infty(x) = 2x
  \frac{x^2-2}{n-x^2}. 
\end{align*}
Direct computations yield $\kappa_1^{\star} (z) = 0$ for all $z$,
which gives 
\begin{align*}
    \mathrm{Kol}(p_n, p_{\infty}) & \le \int_0^{\sqrt{n}} 2x
    \frac{|x^2-2|}{n-x^2}  \frac{1-e^{-x^2}}{2x}  p_n(x)
    \mathrm{d} x 
  % & = \int_0^{\sqrt{n}} 2x
  %   \frac{|x^2-2|}{n-x^2}  \frac{1-e^{-x^2}}{2x} \frac{2}{n} (n-1) x
  %   \left( 1-\frac{x^2}{n} \right)^{n-2} \mathrm{d} x  \\
   \le \frac{ 2(n-1)}{n^{n-1}} \int_0^{\sqrt{n}} x
   |x^2 - 2| (n-x^2)^{n-3}\mathrm{d} x  
\end{align*}
Using the change of variables $u = x^2$ and separating the integral on
$(0, \sqrt2)$ and $(\sqrt2, \sqrt{n})$ it is possible to compute this
integral to obtain 
% \begin{equation*}
% \frac{ 2(n-1)}{n^{n-1}}   \int_0^{\sqrt{n}} x
%    |x^2 - 2| (n-x^2)^{n-3}\mathrm{d} x  =  \frac{1}{n} \left( 1+ 2
%      \Big(1-\frac{2}{n}\Big)^{n-2} \right)
%  \end{equation*}
%  so that 
\begin{equation*}
   \mathrm{Kol}(p_n, P_{\infty})  \le \frac{1}{n} \left( 1+ 2
     \Big(1-\frac{2}{n}\Big)^{n-2} \right) \le \frac{1.28}{n}
 \end{equation*}
 (the upper bound is valid for $n \ge 100$).  The same bound applies
 for Total Variation distance. Finally for Wasserstein distance,
 direct computations yield (using $|h'|\le 1$),
\begin{align*}
 | \kappa_1^{\star} (h) | \le    \lim_{x \to 0}
  \frac{p_n(x)}{p_{\infty}(x)}  \int_0^{\infty}  P_{\infty}(x \wedge u)
  \bar{P}_{\infty}(x \lor u) \mathrm{d}u  + \lim_{x \to \sqrt n}
  \frac{p_n(x)}{p_{\infty}(x)}\int_0^{\infty}  P_{\infty}(x \wedge u)
  \bar{P}_{\infty}(x \lor u) \mathrm{d}u = 0. 
\end{align*}
We have to endure the non tractable function
$\tau_{\infty}(x)$ in the bound
\begin{equation*}
   \mathrm{Wass}(p_n, p_{\infty}) \le \int_0^{\sqrt{n}} 2x
  \frac{|x^2-2|}{n-x^2}   \tau_{\infty}(x)  p_n(x) \mathrm{d} x
\end{equation*}
Nevertheless using $0 \le \tau_{\infty}(x) \le \frac{1}{2}$ the above
becomes
\begin{align*}
   \mathrm{Wass}(p_n, p_{\infty}) & \le \frac{1}{2}\int_0^{\sqrt{n}} 2x
                                    \frac{|x^2-2|}{n-x^2}   p_n(x)
                                    \mathrm{d} x  % \nonumber \\
% & = \frac{1}{2}\int_0^{\sqrt{n}} 2x
%                                     \frac{|x^2-2|}{n-x^2} 
%   \frac{2}{n} (n-1) x
%                                                                                       \left( 1-\frac{x^2}{n} \right)^{n-2} \mathrm{d} x\nonumber \\
%   &\ref
    =\frac{2(n-1)}{n^{n-1}}  \int_0^{\sqrt{n}} x^2
   |x^2-2| 
    \left( n-x^2 \right)^{n-3} \mathrm{d} x.  
\end{align*}
This integral is not as nice as the previous one.  The exact integral
(obtained with the help of mathematica) is
\begin{align*}
  - \frac{\sqrt{\pi} \Gamma(n)}{4 \sqrt{n} \Gamma(n+1/2)} +
2\sqrt{2} \frac{n-1}{n^n}  \frac{(n-2)^n n (40 + 11(n-4)n) +
  (n-2)^3n^n H_2F_1 \left( -\frac{1}{2}, 3-n, \frac{1}{2}, \frac{2}{n} \right) }{(n-2)^2(2n-5)(2n-3)(2n-1)}
\end{align*}
which appears to be quite unfathomable. Numerical evaluations (up to
$n = 10^6$) indicate however that this is slightly less than
$1/n$.

NB. We are indebted to Robert Gaunt for pointing out this problem to
us. For context, details, and alternative computations of similar
quantities, we refer to paper \cite{gaunt2019lap} (in particular
Remark 4.9).

\begin{figure}
  \centering
  \includegraphics[width=0.7\textwidth]{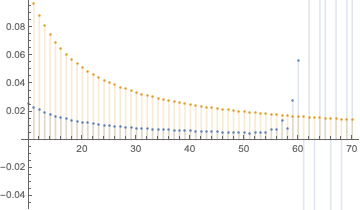}
  \caption{\label{fig:1}Numerical evaluation of the upper bound
    (orange curve) and of the exact Wasserstein distance in (blue
    curve) for $n \in \left\{ 10, 70 \right\}$. The numerical
    evaluation of the Wasserstein distance becomes unstable for
    $n \ge 50$. }
\end{figure}

\end{exm}

   \begin{exm}[Binomial vs Hypergeometric]
     If $X_n\sim \mbox{Hyper}(n,K,N)$ and
     $X_\infty\sim \mbox{Bin}(n,K/N)$, then a direct application of
     classical Stein's method gives the bound $(n-1)/(N-1)$ already
     provided in \cite{ehm1991binomial,holmes2004stein}.  { 
 Equation \eqref{eq:tvbound1} gives
\begin{align*}
TV(\mbox{Bin}(n,K/N), \mbox{Hyper}(n,K,N)) 
&\leq 
\E_H \left[ \left|  \frac{(n - X) (K (n-1) - N X)}{(N-K) (K + n - N - X-1) (1 + X)} 
\right|  \frac{P_B(X)(1-P_B(X))}{p_B(X+1)} \right]
\end{align*}
where the index $H$ denote the expectation computed for the
hypergeometric distribution and the index $B$ is associated to the
binomial distribution.  The bound is not as readable as
\cite{duembgen2019bounding}, who obtain the incredibly elegant
$1-(1-1/N)^{n-1}$.  If we choose the Hypergeometric as target
distribution, we obtain 
\begin{align*}
TV(\mbox{Bin}(n,K/N), \mbox{Hyper}(n,K,N))
&\leq 
\E_{B}\left[ \left| \frac{(n - X) (K (n-1) - N X)}{(N-K) (K + n - N - X-1) (1 + X)} \right| \frac{P_{H}(X)(1-P_{H}(X))}{p_{H}(X+1)} \right].
\end{align*}
The three bounds are graphically compared in Figure \ref{fig:TV-binHyper}.

 \begin{figure}[h]
\centering
  \includegraphics[width=.7\linewidth]{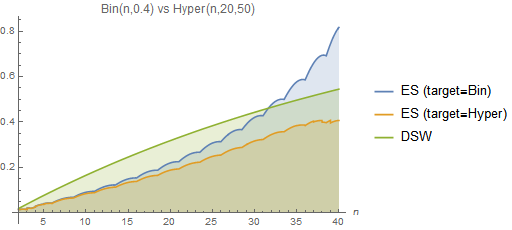}
  \caption{\small Bounds for the total variation distance between the
    binomial $(n,K/N)$ and the hypergeometric $(n,K,N)$ for $K=20$ and
    $N=50$. The blue and orange curve correspond to our bounds, and
    the green one is that from \cite{duembgen2019bounding}. }
  \label{fig:TV-binHyper}
% \quad 
% \begin{subfigure}{.4\textwidth}
%   \centering
%   \includegraphics[width=1\linewidth]{BinPois}
%   \caption{\small Bounds for the total variation distance between the binomial $(n,\theta)$ and the Poisson $n\theta$}
%   \label{fig:TV-BinPois}
% \end{subfigure}
% \caption{ }
\end{figure}
}
\end{exm}

We conclude with two examples which are outside the scope of our
Corollaries \ref{cor:kolbound} and \ref{cor:prec2}. To prepare for
these, we make some simplifying assumptions. We suppose that
$p_{\infty}$ is continuous ($\ell = 0$) with support either the half
line ($0$ may or may not be included) or the real line. Let
$\mathcal{T}_{\infty}$ and $\mathcal{L}_{\infty}$ be its Stein
operators and set $g_h(x) =
\frac{\mathcal{L}_{\infty}h(x)}{c(x)}$. Next let $X_n$ have operators
$\mathcal{T}_n^{\ell}$ and $\mathcal{L}_n^{\ell}$
($\ell \in \left\{ -1, 0, 1 \right\}$) and suppose that
$\mathcal{S}(p_n)$ has infimum (minimum) $a_n \ge 0$ and supremum
(maximum) $b_n\le \infty$.  Starting again from \eqref{eq:ehehv2}, we
know that for all sufficiently regular functions $c$ we can write
  \begin{align}
    &     \mathbb{E}h(X_n) - \mathbb{E}h(X_{\infty})  = \mathbb{E} \left[
      \mathcal{T}_{\infty}c(X_n)g_h(X_n)
      +  c(X_n) g_h'(X_n) \right] \nonumber  \\
    & =  \mathbb{E} \left[
      ( \mathcal{T}_{\infty}c(X_n) - \mathcal{T}_{n}^{\ell}c(X_n))
     g_h(X_n) \right]  +
      \mathbb{E} \left[ c(X_n)  \left(  g_h'(X_n)  -                   \Delta^{-\ell} g_h(X_n)  \right)\right]
-      \kappa_n^{\ell}(h)  \label{eq:greeeeattrump}
  \end{align}
  where
  \begin{align*}
    \kappa_n^{\ell}(h) =
    \begin{cases}
      - c(a_n)/c(a_n-1) p_n(a_n) \mathcal{L}_{\infty}h(a_n-1) & \mbox{ if } \ell = 1 \\
          \lim_{x \to b_n} \mathcal{L}_{\infty}h(x) p_n(x) -
          \lim_{x \to a_n} \mathcal{L}_{\infty}h(x) p_n(x)   & \mbox{ if
          } \ell = 0 \\
      c(b_n)/c(b_n-1) p_n(b_n) \mathcal{L}_{\infty}h(b_n) & \mbox{ if
      } \ell = -1. 
    \end{cases}
  \end{align*}
  Controlling $g_h$ and $g_h'$ via the results from Sections
  \ref{sec:repr} and \ref{sec:examples} easily leads to bounds on the
  usual probability metrics.

\begin{exm}[Maxima of independent to Fr\'echet]
  \label{ex:maxfre}
  Let the target $p_{\infty}$ be the Fr\'echet distribution studied in
  Example \ref{sec:frechet-distribution} and suppose that $X_n$ has
  continuous distribution (i.e.\ $\ell = 0$).  Then taking
  $c(x) = x^{\alpha+1}$ we have $\mathcal{T}_{\infty}c(x) = \alpha$ so
  that \eqref{eq:greeeeattrump} yields 
  \begin{align*}
    \left|     \mathbb{E}h(X_n) - \mathbb{E}h(X_{\infty}) \right| & =  \left|   \mathbb{E} \left[
                                                                    (\alpha- \mathcal{T}_{n}c_1(X_n))
                                                                    g_h(X_n) \right] 
                                                                    -      \kappa_n^{\ell}(h)  \right| \\
                                                                  &
                                                                    \le   \mathbb{E} \left[
                                                                    \left|\alpha- \mathcal{T}_{n}c_1(X_n)\right|
                                                                    \frac{1- e^{-X_n^{-\alpha}}}{\alpha} \right]  +   \left|
                                                                    \kappa_n^{\ell}(h) \right| \\
                                                                  &\le
                                                                    \mathbb{E} \left[
                                                                    \left|1-
                                                                    \frac{1}{\alpha}\mathcal{T}_{n}c_1(X_n)\right|\right]
                                                                    - \kappa_n^{\ell}(h)
  \end{align*}
  for all $h$ such that $\kappa_1 \le 1$, and therefore also for the
  Kolmogorov distance.  Now suppose that
  $M_n = \max(X_1, \ldots, X_n)$ the maximum of $n$ independent
  positive random variables with pdf $f(x)$, cdf $F(x)$ and support
  $[a, b]$. Set $X_n = M_n/r_n$ for $r_n$ some sequence of normalizing
  constants.  Then $a_n=a/r_n$, $b_n=b/r_n$, $ P_n(x) = F({r_n}x)^n$
  and $p_n(x) = n{r_n} f({r_n}x) F({r_n}x)^{n-1}$ so that
   \begin{align*}
     \mathcal{T}_nc_1(x)  = (\alpha+1) x^{\alpha} + r_n x^{\alpha+1} \left(
\frac{f'(r_nx)}{f(r_nx)} +     (n-1)\frac{f(r_nx)}{Fra_nx)}\right)
   \end{align*}
   for $a/r_n \le x \le b/r_n$. Also
   $\kappa_n^{\ell}(h) = \mathcal{L}_{\infty} h(b^-/r_n) n r_n f(b^-)
   F(b^-)^{n-1} - \mathcal{L}_{\infty} h(a^+/r_n) n r_n f(a^+)
   F(a^+)^{n-1}$.
   If, for instance, we choose
   $F(x) = (1-x^{-\alpha}) \mathbb{I}[x \ge 1]$ the Pareto
   distribution with $r_n = n^{1/\alpha}$ then $a=1$,
   $a_n = n^{-1/\alpha}$, $b=b_n = \infty$, and
   $p_n(x) = \alpha x^{-\alpha-1}(1-x^{-\alpha}/n)^{n-1}$ for
   $x \ge n^{-1/\alpha}$ so that
   $ \kappa_n^{\ell}(h) = 0 $ and 
$     \frac{1}{\alpha}\mathcal{T}_nc_1(x)  = \frac{n-1}{n} \big(
1-   \frac{x^{-\alpha}}{n}\big)^{-1},$ $x \ge n^{-1/\alpha}. 
$ We readily obtain
 \begin{equation*}
   \mathbb{E} \left[ \left| 1 -
       \frac{1}{\alpha}\mathcal{T}_nc_1(X_n)\right| \right] =
   \mathbb{E} \left[ \left| 1 - 
\frac{n-1}{n} \left(
1-   \frac{X_n^{-\alpha}}{n}\right)^{-1}\right| \right]  =
\frac{2}{n-1} \left( 1-\frac{1}{n} \right)^n \le \frac{2/e}{n-1} 
\end{equation*}
independently of $\alpha$.  See also \cite[Section 2.6]{ley2017stein}.
\end{exm}

\begin{exm}[Binomial to normal] \label{ex:binotonormsu} Consider
  $p_{\infty}$ the standard normal density, and $p_n$ the density of
  $X_n$ a standardized binomial with parameters $n, \theta$, that is
  $X_n = (B_n -n\theta)/\sqrt{n \theta (1- \theta)}$ where
  $B_n \sim \mathrm{Bin}(n, \theta)$. Let
  $r_n = \sqrt{n \theta (1-\theta)}$. Then $a_n = -n\theta /r_n$ and
  $b_n = n(1-\theta) / r_n$ and
  \begin{equation*}
    p_n(x) = \binom{n}{n \theta + r_n x} \theta^{n
      \theta + r_n x}(1-\theta)^{n(1-\theta)-r_n x } \mbox{ for } x
    \in \mathcal{S}(p_n) = \left\{ \frac{k-n\theta}{r_n}\mbox{ with } k \in \left\{ 0, \ldots, n
    \right\}\right\}.
  \end{equation*}
  An appropriate derivative in this case is
  $\Delta^\ell_nf(x) = r_n/\ell(f(x+\ell/r_n) - f(x))$, $\ell \in
  \left\{ -1, 1 \right\}$; note that if
  $f$ is twice differentiable then, from Taylor's theorem
\begin{align*}
\Delta^\ell_nf(x) = f'(x) + \frac{1}{r_n}\mathbb{E}\Big[V f''\Big(X+\frac{UV}{r_n}
                                              \Big)
                                              \Big]  
\end{align*}
where $U, V$ are independent uniform on $[0,1]$. 
The
canonical operator  for $p_n$ is 
  \begin{align*}
    \mathcal{T}_n^\ell f(x) & = \frac{\Delta^{\ell}_n (f(x)p_n(x))}{p_n(x)} \\
    & = r_n\left( f\left(x+ \frac{1}{r_n}\right) \frac{\theta}{1-\theta}
    \frac{n(1-\theta) - xr_n}{ n \theta +1 + x
      r_n} - f(x)\right) \qquad (\ell = 1)\\
        & = r_n\left( f(x) - f \left( x - \frac{1}{r_n} \right)
          \frac{1-\theta}{\theta} \frac{n\theta + x
          r_n}{n(1-\theta) 
          +1 
          - x r_n}\right) \qquad (\ell = -1)
  \end{align*}
  for $x \in \mathcal{S}(p_n) $ and $p_n$ satisfies the identities
  \begin{align*}
    \mathbb{E} [ \mathcal{T}_n^\ell c(X_n) f(X_n) + c(X_n)
    \Delta^{-\ell}_nf(X_n)] = \kappa_n^{\ell}(c, f)
  \end{align*}
  with
  $\kappa_n^{+}(c,f) =
  -c(-n\theta/r_n)f((-n\theta-1)/r_n)p_n(-n\theta/r_n)$ and
  $\kappa_n^{-}(c,f) =
  c(n(1-\theta)/r_n)f(((n(1-\theta)+1)/r_n)p_n(n(1-\theta)/r_n)$.  If
  we pick $c(x) = 1$ then, after some simplifications, 
\begin{align*}
 \mathcal{T}_{\infty}c(x)
                                              -
                                              \mathcal{T}_{n}^+c(x)
 &  =-x -r_n\left( 
    \frac{r_n^2 - xr_n \theta}{ r_n^2 + x
   r_n (1-\theta)  + (1-\theta)} - 1 \right) \\
  % & =-x \left( 1 -   \frac{1  + (1-\theta)/(xr_n)}{1 + x
  %   (1-\theta)/r_n  + (1-\theta)/r_n^2} \right)\\
    & =-x +  r_n\left(    \frac{ x r_n  -(1-\theta)}{r_n^2 + xr_n 
      (1-\theta) + (1-\theta)} \right)\\
          & =-x +  x \left(    \frac{  r_n^2  -(1-\theta)r_n/x}{r_n^2 + xr_n 
            (1-\theta) + (1-\theta)} \right)\\
  & = -(1-\theta) \frac{(x^2+1)r_n +x}{r_n^2 + xr_n 
            (1-\theta) + (1-\theta)}.
\end{align*}
This function is negative throughout $\mathcal{S}(p_n)$ and
explicit  computations (we use Mathematica) inform us that 
\begin{align*}
  \mathbb{E} \left[ \left| \mathcal{T}_{\infty}c(x)
                                              -
                                              \mathcal{T}_{n}^+c(x)
  \right| 
  \right] \le  2 \sqrt{\frac{1}{\theta}-1} \frac{1}{\sqrt n} 
\end{align*}
(the exact expression is not very enlightening).  With obvious
accommodations to the notations, we have
$\kappa_n^+(h) = (1-\theta)^n \mathcal{L}_{\infty}h(-n\theta/r_n)$. %  and
% thus 
% \begin{align*}
%   \mathbb{E}h(X_n) -  \mathbb{E}h(X_\infty) & =(1-\theta)  \mathbb{E}  \Big[
% \frac{r_n - r_n X_n(1+X_n r_n)}{ (1-\theta)(1+ X_n r_n) +   r_n^2}g_{h}(X_n)\Big]
% + \frac{1}{r_n}\mathbb{E}\Big[ V g_h''\Big(X+\frac{UV}{r_n}
%                                               \Big)
%                                               \Big]   - \kappa_n^+(h).
% \end{align*}
For the sake of brevity we only consider the case of Wasserstein
distance with $h(x)$ Lipschitz. Then $|g_h(x)| \le 1$ and
$|g_h''(x)|\le 2$ (this result is available e.g.\ from \cite[Lemma
2.4]{chen2010normal}) so that
\begin{align*}
  \mathrm{Wass}(p_n, p_{\infty}) \le  2 \sqrt{\frac{1}{\theta}-1} \frac{1}{\sqrt n} 
+ \frac{2}{\sqrt{n \theta(1-\theta)}}   + (1-\theta)^n. 
\end{align*}
We could also obtain rates in the Kolmogorov and Total Variation
distances, but this would require more work for what is, ultimately,
only a proof of concept.  As far as we are aware, the first to have
performed Stein's method of comparison of generators for comparing a
discrete and a continuous distribution are \cite{goldstein2013stein}.
\end{exm}

\

\noindent (M. Ernst) \textsc{D\'epartement de Math\'ematique,
  Facult\'e des Sciences, Universit\'e de Li\`ege, Belgium}
 
\

\noindent (Y. Swan) \textsc{D\'epartement de Math\'ematique,
  Facult\'e des Sciences, Universit\'e libre de Bruxelles, Belgium}

\

\noindent\emph{E-mail address}, M. Ernst {\tt m.ernst@uliege.be }

\noindent\emph{E-mail address}, Y. Swan {\tt yvswan@ulb.ac.be }
\clearpage

\end{document}